\documentclass[a4paper,12pt]{amsart}
\usepackage[a4paper,margin=2.4cm]{geometry}
\usepackage{hyperref}
\usepackage{amssymb}
\usepackage{amsmath}
\usepackage{mathrsfs}
\AtEndDocument{\vfill\eject\batchmode}
\newtheorem{iTheorem}{Theorem}
\newtheorem{iProp}{Proposition}

\newtheorem{Theorem}{Theorem}[section]
\newtheorem{conjecture}[Theorem]{Conjecture}
\newtheorem{lemma}[Theorem]{Lemma}
\newtheorem{prop}[Theorem]{Proposition}
\newtheorem{cor}[Theorem]{Corollary}
\theoremstyle{definition}
\newtheorem{defn}[Theorem]{Definition}
\newtheorem{problem}[Theorem]{Problem}
\theoremstyle{remark}
\newtheorem{rem}[Theorem]{Remark}
\newtheorem{rems}[Theorem]{Remarks}
\newtheorem{ex}[Theorem]{Example}

\newcounter{numl}
\newcommand{\labelnuml}{\textup{(\roman{numl})}}
\newenvironment{numlist}{\begin{list}{\labelnuml}%
{\usecounter{numl}\setlength{\leftmargin}{0pt}%
\setlength{\itemindent}{2\parindent}%
\setlength{\itemsep}{\smallskipamount}\def
\makelabel ##1{\hss \llap {\upshape ##1}}}}{\end{list}}
\newenvironment{bulletlist}{\begin{list}{\labelitemi}%
{\setlength{\leftmargin}{\parindent}\def
\makelabel ##1{\hss \llap {\upshape ##1}}}}{\end{list}}
\DeclareSymbolFont{script}{U}{eus}{m}{n}
\DeclareSymbolFontAlphabet{\amathscr}{script}
\DeclareMathSymbol{\Wedge}{0}{script}{"5E}
\DeclareMathAlphabet{\mathrmsl}{OT1}{cmr}{m}{sl}
\newcommand{\ra}{\rightarrow}
\renewcommand{\geq}{\geqslant}
\renewcommand{\leq}{\leqslant}
\newcommand{\st}{\,|\,}
\newcommand{\R}{{\mathbb R}}
\newcommand{\C}{{\mathbb C}}
\newcommand{\intZ}{{\mathbb Z}}
\newcommand{\N}{{\mathbb N}}
\newcommand{\Q}{{\mathbb Q}}
\newcommand{\T}{{\mathbb T}}
\newcommand{\G}{{\mathbb G}}
\newcommand{\PP}{{\mathbb P}}

\newcommand{\con}{\mathfrak{con}}
\newcommand{\crJ}{{\mathfrak{cr}}}
\newcommand{\aut}{\mathfrak{aut}}
\newcommand{\tor}{{\mathfrak t}}

\newcommand{\Lam}{{\mathrmsl\Lambda}}
\newcommand{\cL}{{\mathcal L}}
\newcommand{\cO}{{\mathcal O}}
\newcommand{\Cx}{{\mathcal C}}
\newcommand{\cS}{{\mathcal S}}
\newcommand{\Hol}{{\mathscr H}}
\newcommand{\tstM}{{\mathscr M}}
\newcommand{\tstL}{{\mathscr L}}
\newcommand{\tstA}{{\mathscr A}}
\newcommand{\tstC}{{\mathscr V}}
\newcommand{\tstN}{{\mathscr N}}
\newcommand{\tstJ}{{\mathscr J}}
\newcommand{\tstD}{{\mathscr D}}

\newcommand{\tstK}{{\mathscr K}}
\newcommand{\Fut}{{\mathscr F}}
\newcommand{\tstGF}{{\mathscr{G\!F}}}
\newcommand{\Futc}{\mathrm{Fut}}
\newcommand{\DF}{\mathrm{DF}}

\newcommand{\sub}{\subseteq}
\newcommand{\Id}{\mathit{Id}}
\newcommand{\vol}{\mathit{vol}}

\newcommand{\End}{\mathrmsl{End}}

\newcommand{\Con}{\mathrmsl{Con}}
\newcommand{\Diff}{\mathrmsl{Diff}}
\newcommand{\Aut}{\mathrmsl{Aut}}

\newcommand{\spn}[1]{\R{\cdot}#1}

\newcommand{\restr}[1]{|_{#1}^{\vphantom x}}
\newcommand{\iip}[1]{\langle\!\langle #1 \rangle\!\rangle}
\newcommand{\ip}[1]{\langle #1 \rangle}
\renewcommand{\d}{{\mathrmsl d}}
\newcommand{\grad}{\mathop{\mathrmsl{grad}}\nolimits}

\newcommand{\Scal}{\mathit{Scal}}
\newcommand{\Sm}{N}
\newcommand{\Cm}{V}
\newcommand{\Sph}{{\mathbb S}}
\newcommand{\Ds}{{\amathscr D}}
\newcommand{\Lv}{L}
\newcommand{\cf}{\eta}
\newcommand{\X}{X}
\newcommand{\Y}{Y}
\newcommand{\K}{K}
\newcommand{\Z}{Z}
\newcommand{\cae}{\sigma}
\newcommand{\sas}{\tau}
\newcommand{\lO}{{\mathcal O}}
\newcommand{\lo}{{\mathfrak o}}
\newcommand{\thol}{{\mathfrak h_0}}
\newcommand{\kcl}{\zeta}
\newcommand{\pt}{p}
\newcommand{\Sc}{\Xi}
\newcommand{\Sp}{\hat\Xi}
\def\th/#1#2{{#1}^{#2}}
\newcommand{\ME}{{\mathbf M}}
\newcommand{\ext}{{\mathrm{ext}}}
\newcommand{\cl}{{\mathrm{cl}}}
\newcommand{\cpot}{{\mathcal R}}
\newcommand{\IsoPot}{{\boldsymbol\Theta}}
\newcommand{\hTheta}{\hat {\boldsymbol\Theta}}
\newcommand{\flow}{\Phi}
\newcommand{\Cv}{\widehat{V}}
\newcommand{\tstV}{{\mathscr V}}
\newcommand{\tstcone}{\tstV}
\begin{document}

\title[K-stability and extremality  of Sasaki manifolds]{Weighted K-stability of polarized varieties and extremality of Sasaki manifolds}
\author[V. Apostolov]{Vestislav Apostolov} \address{Vestislav Apostolov \\ D{\'e}partement de Math{\'e}matiques\\ UQAM\\ C.P. 8888 \\ Succursale Centre-ville \\ Montr{\'e}al (Qu{\'e}bec) \\ H3C 3P8 \\ Canada} \email{apostolov.vestislav@uqam.ca}

\author[D.M.J. Calderbank]{David M. J. Calderbank} \address{David M. J. Calderbank \\ Department of Mathematical Sciences\\ University of Bath\\ Bath BA2 7AY\\ UK} \email{D.M.J.Calderbank@bath.ac.uk}

\author[E. Legendre]{Eveline Legendre}\address{Eveline Legendre\\ Institut de Math\'ematiques de Toulouse\\ Universit\'e Paul Sabatier\\  118 route de Narbonne\\ 31062 Toulouse\\ France} \email{eveline.legendre@math.univ-toulouse.fr}

\begin{abstract}
We use the correspondence between extremal Sasaki structures and weighted
extremal K\"ahler metrics defined on a regular quotient of a Sasaki manifold,
established by the first two authors in \cite{AC}, and Lahdili's theory of
weighted K-stability \cite{lahdili2} in order to define a suitable notion of
(relative) weighted K-stability for compact Sasaki manifolds of regular type.
We show that the (relative) weighted K-stability with respect to a maximal
torus is a necessary condition for the existence of a (possibly irregular)
extremal Sasaki metric.  We also compare weighted K-stability to the
K-stability of the corresponding polarized affine cone (introduced by
Collins--Sz\'ekelyhidi in \cite{CSz}), and prove that they agree on the class
of test configurations we consider.  As a byproduct, we strengthen the
obstruction to the existence of a scalar-flat K\"ahler cone metric found in
\cite{CSz} from the K-semistability to the K-stability on these test
configurations.  We use our approach to give a characterization of the
existence of a compatible extremal Sasaki structure on a principal
$\Sph^1$-bundle over an admissible ruled manifold in the sense of
\cite{HFKG3}, expressed in terms of the positivity of a single polynomial of
one variable over a given interval.
\end{abstract}

\date{\today}

\maketitle

\section*{Introduction}

The famous Calabi problem~\cite{calabi}, which seeks the existence of
canonical K\"ahler metrics in a given K\"ahler class $\kcl \in H^2(M, \R)$
on a compact complex manifold $(M, J)$, is a central and very active topic of
current research in K\"ahler geometry.  As a candidate for a canonical metric
Calabi proposed a notion of \emph{extremal K\"ahler metric} $g$, meaning that
its scalar curvature $\Scal(g)$ is a \emph{Killing potential}, i.e., the
vector field $J\grad_g \Scal(g)$ is a Killing vector field for $g$. Examples
include constant scalar curvature (CSC) K\"ahler metrics on $(M,J)$, and hence
also K\"ahler--Einstein metrics.

There is an odd-dimensional analogue of the Calabi problem---and a
corresponding notion of \emph{extremal Sasaki structures} (called
\emph{canonical} in \cite{BG-book})---associated to compact, strictly
pseudoconvex CR manifolds $(\Sm, \Ds, J)$ of Sasaki type. This may be
motivated using the construction of \emph{regular} Sasaki structures on the
total space $\Sm$ of the principal $\Sph^1$-bundle over a Hodge manifold
$(M,J,\zeta)$ determined by the given integral K\"ahler class $\kcl/2\pi \in
H^2(M, \intZ)$: any K\"ahler metric $\omega \in \kcl$ on $M$ induces a
connection $1$-form $\cf$ on $\Sm$ satisfying $\d\cf = \pi^* \omega$, and
hence a CR structure $(\Ds,J)$ with $\Ds=\ker\cf\sub T\Sm$ and $J$ the
pullback of the complex structure on $M$ to $\Ds\cong\pi^*TM$. In this case,
the vector field $\K$ induced by the fibrewise $\Sph^1$-action on $\Sm$ is
called \emph{regular Sasaki--Reeb} vector field.  The Calabi problem on $(M,
J, \kcl)$ translates to the search for CR structures $(\Ds, J)$ on $\Sm$,
compatible with the induced \emph{transversal holomorphic structure} $J^\K$ on
the vector bundle $T\Sm/\langle \K \rangle \to \Sm$, such that $(\Ds, J, \K)$
is a Sasaki structure corresponding to an extremal, CSC or K\"ahler--Einstein
metric $\omega\in \kcl$ on $M$. This leads respectively to notions of
extremal, CSC and Einstein Sasaki structures on $\Sm$.

The above construction can be generalized by relaxing the assumption that the
$\Sph^1$-action on $\Sm$ generated by the vector field $\K$ is free. If it is
instead required only to be locally free, then $\Sm$ can be realized as a
principal $\Sph^1$-orbibundle over a k\"ahlerian orbifold $(M, J, \kcl)$, and
$\K$ is then referred to as a \emph{quasiregular} Sasaki--Reeb vector field on
$\Sm$. One can also weaken the assumption that $K$ generates an
$\Sph^1$-action, and require only that the closure of the subgroup generated
by the vector field $\K$ is a compact torus $\T$ in the diffeomorphism group
of $\Sm$ preserving the given transversal holomorphic structure $J^\K$; if
$\dim\T>1$, $\K$ is called an \emph{irregular} Sasaki--Reeb vector field.

Irrespective of regularity, the correspondence between K\"ahler geometry and
Sasaki geometry holds locally, allowing one to extend the notions of extremal,
CSC and Einstein Sasaki structures: indeed, any point of a Sasaki manifold
$(\Sm,\Ds,J, \K)$ has a neighbourhood in which the leaf space $M$ of the flow
of $\K$ is a manifold and has an induced K\"ahler structure $(g,J,\omega)$.
We refer to $(M, g, J,\omega)$ as a \emph{Sasaki--Reeb quotient} of $(\Sm,
\Ds, J, \K)$, and say that $(\Sm, \Ds, J, \K)$ is extremal, CSC or
Sasaki--Einstein if any local Sasaki--Reeb quotient is extremal, CSC, or
K\"ahler--Einstein manifold, respectively.

The study of irregular extremal Sasaki metrics is more delicate than the
regular case (which reduces to the Calabi problem in the K\"ahler context).
One approach for dealing with the issue of irregularity stems from the seminal
work of Martelli--Sparks--Yau~\cite{MSY2}, who showed that the existence of a
Sasaki--Einstein metric compatible with a given transversal holomorphic
structure $(\K, J^\K)$ on $\Sm$ can be equivalently studied using strictly
plurisubharmonic smooth positive functions on the (noncompact) smooth manifold
$\Cm=\R^+ \times\Sm$ equipped with a conical complex structure $J$ inducing
the transversal holomorphic structure $(\K,J^\K)$; such functions define
K\"ahler metrics on $(\Cm, J)$ which are Ricci-flat if and only if the induced
Sasaki structure on $\Sm$ is Einstein. Furthermore, they observed that $(\Cm,
J, \K)$ can be given the structure of an affine cone embedded in $\C^k$ for
some $k \gg0$, with $\K$ being in the Lie algebra of a maximal compact
algebraic torus $\T \leq \mathrm{GL}(k, \C)$ preserving $(\Cm, J)$. The
framework of \cite{MSY2} can be adapted \cite{He-Sun,CSz,BV} to recast, more
generally, the search for compatible extremal Sasaki structures on $(\Sm, \K,
J^\K)$.

Using the algebraic nature of $(\Cm, J, \K)$,
Collins--Sz\'ekelyhidi~\cite{CSz} introduced a suitable notion of
\emph{K-stability} of $(\Cm, J, \K , \T)$, and showed that the
\emph{K-semistability} of $(\Cm, J, \K , \T)$ is a necessary condition for the
existence of a compatible CSC Sasaki structure on $(\Sm, \K, J^\K)$. A
corresponding notion of \emph{relative} K-stability for $(\Cm, J, \K, \T)$,
obstructing the existence of \emph{extremal} Sasaki structures on $(\Sm, \K,
J^\K)$, has been introduced in \cite{BV}, where the relative K-semistability
of $(\Cm, J, \K, \T)$ with respect to $\T$-equivariant test configurations of
affine cones in $\C^k$ is shown to be a necessary condition for the existence
of a compatible extremal Sasaki structure on $(\Sm, \K, J^\K)$. These works
suggest a general Yau--Tian--Donaldson correspondence---analogous to the
projective K\"ahler case~\cite{Do-02, Tian, yau}---for the existence of
compatible extremal Sasaki structures on $(\Sm, \K, J^\K)$ expressed in terms
of a suitable notion of K-stability of the affine complex cone $(\Cm, J, \K)$.
Such a correspondence has only been established in the Sasaki--Einstein case
in \cite{CSz2}, where Collins and Sz\'ekelyhidi refine their original notion
of K-stability by considering compatible \emph{special test configurations} of
$(\Cm, J, \K , \T)$, show that the existence of a compatible Sasaki--Einstein
metric implies that $(\Cm, J, \K , \T)$ is K-stable (not merely K-semistable)
for such test configurations and, conversely, prove that the K-stability of
$(\Cm, J, \K , \T)$ with respect to special test configurations yields the
existence of a Sasaki--Einstein metric on $(\Sm, \K, J^\K)$.  None of these
refinements seem to be available in the general CSC and extremal cases.

\smallskip

Thus motivated, in this paper we propose and develop a new approach to the
existence problem of extremal Sasaki structures and the corresponding notion
of (relative) K-stability.  It uses the key feature (see
\cite{CFO,FOW,legendre2,MSY1}) that one can construct (even irregular) CSC
Sasaki manifolds from K\"ahler manifolds which are not necessarily CSC.  This
follows from the observation that the space of Sasaki--Reeb vector fields $\K$
(i.e., the vector fields giving rise to transversal holomorphic structures of
Sasaki type) is open in the Lie algebra $\tor$ of a compact torus $\T$ inside
the group of CR transformation of $(\Sm, \Ds, J)$. In particular, one can
consider two different Sasaki--Reeb vector fields $\X, \K \in \tor$, such that
$\X$ is quasiregular even if $\K$ is not.  We will focus in this paper on the
case that $\X$ is a \emph{regular} Sasaki--Reeb vector field, although some of
our considerations can be extended to the case of quasiregular $\X$, with the
obvious technical difficulties coming from dealing with orbifolds rather than
smooth manifolds.

\subsection*{A correspondence between Sasaki structures}
Suppose therefore that $(\Sm, \Ds, J, \X)$ is a regular compact Sasaki
manifold over a smooth polarized K\"ahler manifold $(M, J, L)$, with induced
Sasaki--Reeb quotient K\"ahler structure $(g, \omega)$ on $M$, $\omega \in
\kcl = 2\pi c_1(L)$, and suppose also that $\K$ is another Sasaki--Reeb vector
field on $(\Sm, \Ds, J)$ commuting with $\X$. Then $\K$ projects to a Killing
vector field on $(M, J, g, \omega)$, denoted by $\check\K$, which (by the
Sasaki--Reeb condition) has a positive Killing potential $f$ with respect to
$g$.  Furthermore, as observed in \cite{AC}, $(\Ds, J, \K)$ is an extremal
Sasaki structure if and only if $(g, \omega)$ satisfies the condition that
\emph{$(m+2, f)$-weighted scalar curvature}
\begin{equation*}
\Scal_{f, m+2}(g) := f^2\Scal(g) - 2(m+1) f\Delta_g f - (m+2) (m+1) |\d f|^2_g
\end{equation*}
of $g$ is a Killing potential; here $m$ is the complex dimension of $M$, and
$\Scal(g)$ and $\Delta_g$ are the scalar curvature and riemannian Laplacian of
$g$. We refer herein to such an $(m+2,f)$-extremal K\"ahler metric as a
\emph{$\K$-extremal K\"ahler metric}: see
\S\ref{ss:calabi-problem-kahler}--\S\ref{ss:calabi-problem-sasaki} for more
details.

Following the approach in \cite{AM}, one can extend the Calabi problem to this
(more general) weighted setting, and relate it to the extremal Sasaki problem
as follows.

\begin{iProp}[see Proposition~\ref{p:reduction}] \label{p:main0} 
Let $(\Sm, \Ds, J, \X)$ be a compact quasiregular Sasaki manifold, and $(M,
J, L)$ the corresponding smooth polarized K\"ahler orbifold.  Suppose that
$\K$ is a Sasaki--Reeb vector field on $(\Sm, \Ds, J)$ which commutes with
$\X$. Then $\Sm$ admits an extremal Sasaki structure compatible with the
transversal holomorphic structure $(\K, J^\K)$ determined by $(\K, \Ds, J)$ if
and only if  $(M, J)$ admits a $\K$-extremal K\"ahler
metric in the K\"ahler class $2\pi c^{orb}_1(L)$.
\end{iProp}

\subsection*{A weighted K-stability theorem} The class of $\K$-extremal
K\"ahler metrics is a special case of a more general notion of \emph{weighted
  extremal K\"ahler metrics} studied recently by Lahdili~\cite{lahdili1,
  lahdili, lahdili2}. In particular, it follows from these works that one can
reduce the search for $\K$-extremal metrics on a smooth polarized compact
K\"ahler manifold $(M, J, L)$ by fixing a maximal compact torus $\T \leq
\Aut(M, L)$, such that $\K$ belongs to ${\rm Lie}(\T)$, and considering
$\T$-invariant K\"ahler metrics in $\kcl= 2\pi c_1(L)$. Furthermore, for any
smooth $\T$-equivariant polarized test configuration $(\tstM, \tstL)$ of $(M,
L)$, Lahdili~\cite{lahdili2} associates a notion of weighted Donaldson--Futaki
invariant $\Fut^\ext_\K (\tstM, \tstL)$ with respect to the Sasaki--Reeb
vector field $\K$, and proves that if $(M, J)$ admits a $\K$-extremal K\"ahler
metric in $\kcl$ and the smooth test configuration $(\tstM, \tstL)$ has a
reduced central fibre, then $\Fut^\ext_\K(\tstM, \tstL)\ge 0$. This leads to a
notion of a relative weighted K-semistabilty of $(M, J, L, \T, \K)$ which
obstructs the existence of $\K$-extremal K\"ahler metrics in $2\pi c_1(L)$.

Our first main result extends this weighted K-semistability to weighted
K-stability.

\begin{iTheorem}[see Theorem~\ref{Theorem:polystable}] \label{Theorem:main1} 
Let $(\Sm, \Ds, J, \X)$ be a compact regular Sasaki manifold, and $(M, J, L)$
the corresponding smooth polarized K\"ahler manifold.  Suppose that $\K$ is a
Sasaki--Reeb vector field on $(\Sm, \Ds, J)$ which commutes with $\X$, such
that $\Sm$ admits an extremal Sasaki structure compatible with the transversal
holomorphic structure $(\K, J^\K)$ determined by $(\K, \Ds, J)$.  Let $\T$ be
a maximal torus in the group $\Aut(M, L)$ of automorphisms of the polarized
variety $(M, J, L)$, such that $\X, \K$ belong to ${\rm Lie}(\T)$. Then, for
any $\T$-equivariant smooth polarized test configuration $(\tstM, \tstL)$
associated to $(M, J, L, \T)$, which has reduced central fibre and is not a
product test configuration, we have $\Fut_\K^\ext(\tstM, \tstL) > 0$,
i.e., $(M, J, L, \K, \T)$ is relative weighted K-stable on such test
configurations.
\end{iTheorem}

By Proposition~\ref{p:main0}, the existence of an extremal Sasaki structure on
$(\Sm, \K, J^\K)$ translates to the existence of a $\T$-invariant
$\K$-extremal metric $(g, \omega)$ on $M$ with $\omega \in 2\pi c_1(L)$.  In
this setting it is well-known (see in particular \cite{BDL,He-ext,Dyrefelt}
for unweighted CSC and extremal K\"ahler metrics) that a K-stability result
like Theorem~\ref{Theorem:main1} may be deduced by establishing a suitable
properness property for an appropriate Mabuchi energy functional. Furthermore,
there is an axiomatic framework~\cite{DaRu} for establishing such properness
properties, which allows us to focus on the specificities of our problem.

In our setting, we denote the relevant functional by $\ME_{\check\K, \kappa}$,
where the subscripts $\check{\K}$ and $\kappa$ stand respectively for the
holomorphic vector field on $(M,J)$ induced by $\K$, and a positive real
constant, determined by $K$, see \S\ref{ss:calabi-problem-sasaki}. This
functional is a weighted version of the relative Mabuchi energy introduced in
\cite{lahdili} defined on the space of $\T$-invariant hermitian products on
$L$ with positive curvature. To prove the relative weighted K-stability of
$(M, J, L, \K, \T)$, it suffices show that $\ME_{\check\K, \kappa}$ is proper
with respect to the natural action of the complexified torus $\T_\C$ on that
space.

In order to apply the axiomatic framework of~\cite{DaRu}, we need to extend
$\ME_{\check\K, \kappa}$ to the $d_1$-completion of the space of
$\T$-invariant relative K\"ahler potentials on $M$. The existence of such an
extension is not obvious because, in contrast to the CSC K\"ahler case, the
definition of $\ME_{\check\K, \kappa}$ involves the $\kappa$-normalized
potential of $\check\K$ with respect to a metric $\omega\in c_1(L)$, which has
no natural interpretation in the space of weakly regular relative K\"ahler
potentials.

The key for bypassing this difficulty is also the key to the proof of
Proposition~\ref{p:main0}. Namely, we make use of a bijection $\hTheta$ (see
\S\ref{ss:Thetamap}, and in particular equation~\eqref{e:bij-pot}) between
(suitably normalized) Sasaki structures on $(\Sm, \K, J^\K)$ and $(\Sm, \X,
J^\X)$ such that corresponding Sasaki structures have the same underlying CR
structure. This map comes from a correspondence between their respective
spaces of (radial) K\"ahler potentials on the cone $(\Cm,J,\K)$ that we
establish (see Corollary~\ref{bijection-potentials}) building on the work of
He--Sun~\cite{He-Sun}. While we avoid the noncompact cone in the notion of
stability we use, $(\Cm,J,\K)$ thus plays an essential role in establishing
properties of the map $\hTheta$.

First, after identifying the normalized Sasaki structures on $(\Sm, \K, J^\K)$
with $\T$-invariant hermitian forms on $L$ with positive curvature, we show
(see Lemma~\ref{M}) that
\[
\ME_{\check\K, \kappa}\circ \hTheta = \ME^\K
\]
$\ME^\K$ is the relative Mabuchi energy acting on the space of (normalized)
Sasaki structures in $(\Sm, \K, J^\K)$, studied in \cite{He1,He-Li, V}.
Secondly, in Lemma~\ref{dTheta} we compute the derivative of $\hTheta$, which
enables us to show (see Lemma~\ref{d1}) that it is bilipschitz with respect to
$d_1$. Thirdly, $\hTheta$ is equivariant for the natural action of the
complexified torus $\T_\C$ induced by $\K$ and $\X$ on the space of
$\T$-invariant hermitian products on $L$ (similar to the setting in
\cite{PS}). This is different than the action of complex automorphisms on the
space of relative K\"ahler potentials considered in \cite{BDL, DaRu,
  He-ext,Dyrefelt}, but has the practical advantage to induce also a natural
action on the space of (normalized) Sasaki structures in $(\Sm, \K, J^\K)$.

These tools allow us to define a notion of properness for $\ME_{\check\K,
  \kappa}$ with respect to the natural action of $\T_\C$ which directly
translates via $\hTheta$ to a corresponding notion of properness for $\ME^\K$
(see Corollary~\ref{central}). This allows us to use the extension of $\ME^\K$
to the $d_1$-completion and the regularity of its minimizers (obtained in
\cite{He1,He-Li}) to show the properness of $\ME^\K$ (see
Theorem~\ref{main-result}) and hence of $\ME_{\check\K, \kappa}$. A key result
for this to work is the transitivity of the action of $\T_\C$ on the space of
$\T$-invariant $\K$-extremal K\"ahler metrics in $2\pi c_1(L)$, which has been
obtained in \cite{lahdili3} using the approach in \cite{CPZ}.  Once the
$\T_\C$-properness of $\ME_{\check\K, \kappa}$ is established, the proof of
Theorem~\ref{Theorem:main1} essentially follows from the arguments
in~\cite{BDL,Dyrefelt}.

Since the complex cone $(\Cm,J,\K)$ plays such an important role in the proof,
we spend some time in Section~\ref{s:abstractcone} developing its properties,
following~\cite{BV,CSz,He-Sun,MSY2}. In particular, in
\S\ref{ss:Reebcone-rev}, we detail the relation (see~\cite{BV,CSz} and
Lemma~\ref{l:reeb}) between various known points of view on the Reeb cone
$\tor_+$ in the Lie algebra $\tor$ of $\T$, establishing that $\tor_+$
consists of vector fields $K$ whose cones $(\Cm,J,\K)$ admit a radial K\"ahler
potential.

\subsection*{Weighted K-stability versus K-stability of affine cones} 

Theorem~\ref{Theorem:main1} provides in particular an obstruction to the
existence of CSC Sasaki structures compatible with a transversal holomorphic
structure $(\K, J^\K)$ on $\Sm$, expressed in terms of the weighted
K-stability of a fixed regular quotient $(M, J, L, \T)$ with respect to
suitable test configurations $(\tstM, \tstL)$. We first remark that the
assumptions that $\tstM$ is smooth and its central fibre is reduced are needed
in the proof of Theorem~\ref{Theorem:main1} (as in~\cite{lahdili2}) both to
define $\Fut_\K^\ext(\tstM, \tstL)$ and to relate it to the slope of the
corresponding Mabuchi energy $\ME_{\check\K, \kappa}$.  However, it is shown
in \cite{DR,dyrefelt} that in the CSC case, restricting to such test
configurations of $(M, J, L)$ does not affect the corresponding (unweighted)
K-stability condition. We expect that these assumptions on $(\tstM, \tstL)$
will not limit the applicability of Theorem~\ref{Theorem:main1} either.

Secondly, we need to compare the weighted K-stability in
Theorem~\ref{Theorem:main1} to the (relative) K-stability of the corresponding
affine cone $(\Cm, J, \K , \T)$ in the sense of \cite{CSz}.  To do this, we
use the fact that the complex manifold $(\Cm,J)$ can be identified with
$(L^*)^\times$ (the total space of the dual bundle of $L$ with the zero
section removed) whereas for any test configuration $(\tstM, \tstL)$ as in
Theorem~\ref{Theorem:main1}, the corresponding complex cone $\tstC =
(\tstL^*)^\times$ gives rise to a $\T$-equivariant \emph{smooth test
  configuration} of the polarized affine cone $(\Cm, J, \K)$ in the sense of
\cite{CSz,Li-Xu}. We shall refer to such test configurations of
$(\Cm, J, \K)$ as \emph{smooth $\T$-equivariant regular test configurations}
(where we have also assumed that the central fibre of $\tstM$ is reduced).
We then have the following result.

\begin{iTheorem}[see Corollary~\ref{c:polystability}]
Let $(\Cm,J,\K)$ be the polarized complex cone associated to a compact
Sasaki manifold $(\Sm, \Ds, J, \K)$, $\T \leq \Aut(\Cm,J)^\K$ a maximal
torus and assume that the Lie algebra of $\T$ contains a regular
Sasaki--Reeb vector field $\X$. If $(\Cm, J, \K)$ admits a scalar-flat K\"ahler
cone metric, then for any $\T$-equivariant regular test configuration
$(\tstC,\K)$ of $(\Cm,J,\K)$, which is not obtained from a product
test configuration, the Donaldson--Futaki invariant of \cite{CSz} associated
to $(\tstC ,\K)$ is strictly positive.
\end{iTheorem}

This theorem follows readily from Theorem~\ref{Theorem:main1} once we
establish that the Donaldson--Futaki invariant of \cite{CSz} associated to
$(\tstL^*)^\times$ is a positive multiple of the corresponding weighted
Donaldson--Futaki invariant $\Fut_{K}(\tstM, \tstL)$. To establish such a
relationship, we propose in Definition~\ref{l:GF} a differential geometric
definition of a Donaldson--Futaki invariant for any smooth test configuration
$(\tstC, \K)$ (which itself is a complex affine cone) associated to the affine
cone $(\Cm, J, \K, \T)$ of $(\Sm, \K, J^\K)$. We refer to this quantity as the
\emph{global Futaki invariant} and believe it is of independent interest.  It
is expressed as an integral formula on the Sasaki manifold $(\tstN, \tstD,
\cf^{\K}_{\tstD})$ associated to $\tstC$ as
\begin{equation*}
\tstGF_{\K} (\tstcone) = \frac{c_{\K}}{(m+1)!}
\int_{\tstN} { \cf}^{\K}_{\tstD} \wedge (\d \cf^{\K}_{\tstD})^{m+1} 
 +  \frac{2}{m!}
\int_{\tstN} (\pi^*\rho_o -\rho) \wedge \cf^{\K}_{\Ds}\wedge (\d \cf^{\K}_{\Ds})^m,
\end{equation*}
where $\pi\colon \tstC \ra \C P^1$ is the equivariant map of the test
configuration, $c_{\K}$ a normalizing constant, $\rho$ is any representative of
the basic first Chern class $2\pi c_1^B$, and $\rho_o\in 2\pi c_1 (\C P^1)$.
We justify this definition in Proposition~\ref{p:globalDFS} by proving that
the Donaldson--Futaki invariant~\cite{CSz} of any $\T$-equivariant regular
test configuration $\tstC$ for $(\Cm, J, \K)$ is a constant positive multiple
of its global Futaki invariant. Then in Proposition~\ref{p:Weighted=GF}, we
prove that, up to a factor of $2\pi$, the weighted Futaki invariant of
$\T$-equivariant smooth polarized test configuration $(\tstM,\tstL)$ for $(M,
J, L,\T)$ agrees with the global Futaki invariant of its cone $\tstV$.

Another interesting aspect of our approach is that it allows for the
consideration (as in \cite{DR,lahdili2,dyrefelt}) of \emph{K\"ahler}
$\T$-equivariant smooth test configurations $(\tstM, \tstA)$ associated to the
regular quotient $(M, J, L)$, i.e., allowing both non-projective total spaces
$\tstM$ and transcendental K\"ahler classes $\tstA$ on $\tstM$.  Notice that
in this more general framework, the weighted K-semistability of $(M, J, L, K,
\T)$ still holds by the results in \cite{lahdili2}. The novel implication of
this to the theory of extremal Sasaki structures is that a smooth K\"ahler
$\T$-equivariant test configuration $(\tstM, \tstA)$ for which $\tstA$ is not
an integral class does not define a test configuration of the polarized
corresponding affine cone $(\Cm, J, \K)$. This suggests (see
Conjecture~\ref{conjecture0}) that Theorem~\ref{Theorem:main1} extends to
$\T$-equivariant smooth K\"ahler test configurations $(\tstM, \tstA)$
associated to $(M, J, L, \T)$.

\subsection*{A Yau--Tian--Donaldson correspondence on admissible Sasaki
manifolds} The importance of extending the K-stability criterion to K\"ahler
test configurations goes back to an example from \cite{HFKG3} which suggests
that polarized test configurations are in general insufficient to detect the
non-existence of extremal K\"ahler metrics on $(M, J, L)$.  In the final
section of the article, we obtain Sasaki analogues of some results
in~\cite{HFKG3}. To this end, we let $(g, J, \omega)$ be a K\"ahler metric on
the total space $M$ of a $\PP^1$-bundle over the product of CSC polarized
K\"ahler manifolds, obtained by the Calabi construction (also called
\emph{admissible}), and endow $M$ with a fibrewise $\Sph^1$-action which gives
rise to a Killing vector field $\check\K$ on $(M, J, \omega)$.  Furthermore,
assuming also that $\omega\in 2\pi c_1(L)$ for a polarization $L$ on $M$, we
let $\K$ be the vector field on $L$ defined by $\check\K$ and a positive
Killing potential for $\check\K$.  In this setting, the search for
$\K$-extremal K\"ahler metrics in $\kcl$ given by the Calabi ansatz has been
recently studied in \cite{AMT}.  The authors show there that (similarly to the
extremal case studied in \cite{HFKG3}), there exists a polynomial $P(z)$ of
degree $\le m+2$ which is positive on the interval $(-1,1)$ if and only if
$(M, J)$ admits a $\K$-extremal K\"ahler metric in $\kcl$, and in this case,
the extremal K\"ahler metric is given by the Calabi ansatz.  On the other
hand, the computations from \cite{AMT,lahdili2} also show that for each $z_0
\in (-1,1)$, $P(z_0)$ computes (up to a positive multiple) the relative
weighted Donaldson--Futaki invariant of a smooth {\it K\"ahler} test
configuration $(\tstM, {\tstA}_{z_0})$ associated to $(M, J, \kcl)$;
furthermore, $(\tstM, {\tstA}_{z_0})$ is a \emph{polarized} test configuration
precisely when $z_0 \in (-1,1) \cap \Q$. Thus, the weighted K-stability
established in Theorem~\ref{Theorem:main1} only yields that $P(z)$ must be
positive on $ (-1,1) \cap \Q$ should a $\K$-extremal metric exist, whereas
Conjecture~\ref{conjecture0} mentioned above would imply the positivity of
$P(z)$ everywhere on $(-1,1)$, i.e., the existence of a $\K$-extremal metric
of Calabi type. Therefore the following result provides evidence for
Conjecture~\ref{conjecture0}.

\begin{iTheorem}[see Corollary~\ref{c:extension}] \label{Theorem:main2}
Let $(\Sm, \Ds, J, \X)$ be the regular Sasaki manifold corresponding to an
admissible \textup(polarized\textup) K\"ahler manifold $(M, J, \omega)$ and
$\K$ be a Sasaki--Reeb vector field corresponding to a lift \textup(which need
not be quasiregular\textup) of the generator of fibrewise $\Sph^1$-action on
$M$. Then $\Sm$ admits an extremal Sasaki structure in $(\K, J^\K)$ if and
only if the corresponding polynomial $P(z)$ is positive on $(-1,1)$. In this
case, the extremal Sasaki structure corresponds to an $\K$-extremal K\"ahler
metric on $(M, J, \kcl)$ of Calabi type.
\end{iTheorem}

We notice that, as in~\cite{HFKG3}, Theorem~\ref{Theorem:main2} provides
examples of Sasaki manifolds $(\Sm, \K, J^\K)$ (obtained as $\Sph^1$-bundles
over admissible manifolds), which do not admit extremal Sasaki structures for
the reason that the corresponding polynomial $P(z)$ has a unique double
irrational root in $(-1,1)$.  Our theory shows that one way to detect the
non-existence in this case is to consider K\"ahler test configurations on the
regular quotient.

\section{Preliminaries on Sasaki geometry}

\subsection{Sasaki structures and transversal K\"ahler structures}
\label{ss:sasaki-transversalK} 

Let $(\Sm,\Ds)$ be a contact $(2m+1)$-manifold and denote the Levi form of
$\Ds$ by $\Lv_\Ds\colon \Ds\times\Ds\to T\Sm/\Ds$; the latter is defined, via
local sections $\X,\Y\in C^\infty_\Sm(\Ds)$, by the tensorial expression
$\Lv_\Ds(\X,\Y) = - \cf_\Ds([\X,\Y])$, where $\cf_\Ds\colon T\Sm\to T\Sm/\Ds$
is the quotient map.  A \emph{contact vector field} is a vector field $\X$ on
$\Sm$ such that $\cL_\X (C^\infty_\Sm(\Ds))\sub C^\infty_\Sm(\Ds)$.  We denote
by $\con(\Sm, \Ds)$ the Lie algebra of contact vector fields in
$C^\infty_\Sm(T\Sm)$.

We recall the following basic fact in the theory of contact manifolds (see
e.g.~\cite{BG-book}).

\begin{lemma}\label{l:contact} The map $\X\mapsto\cf_\Ds(\X)$ from contact
vector fields to sections of $T\Sm/\Ds$ is a linear isomorphism, whose inverse
$\cae\mapsto \X_\cae$ is a first order linear differential operator.

In particular, for any vector field $\X$ on $\Sm$, there is a unique contact
vector field $\X_\sas$ \textup(namely the one with $\sas=\cf_\Ds(\X)$\textup)
such that $\X-\X_\sas$ is a section of $\Ds$.
\end{lemma}

Now suppose $J\in\End(\Ds)$ is a CR structure on $(\Sm,\Ds)$, i.e., a
(fibrewise) complex structure on $\Ds$ such that $\Ds^{(1,0)}:= \{\X-i J \X\st
\X\in \Ds\}$ is closed under Lie bracket in $TN\otimes\C$; then we denote by
\[
\crJ(\Sm,\Ds,J):=\{ \X \in \con(\Sm,\Ds)\st \cL_\X J =0\}
\]
the Lie subalgebra of $\con(\Sm,\Ds)$, whose elements correspond to CR vector
fields $\X$ on $(\Sm, \Ds, J)$. If moreover $(\Ds,J)$ is strictly pseudoconvex
then $T\Sm/\Ds$ has a canonical orientation (and so $(\Sm,\Ds)$ is
\emph{co-oriented}): the positive sections $\sas$ are those for which
$\sas^{-1} \Lv_\Ds(\cdot,J\cdot)$ is positive definite.  Note that $\sas^{-1}
\Lv_\Ds=\d\cf_\Ds^\X\restr\Ds$ where $\cf_\Ds^\X:=\sas^{-1}\cf_\Ds$ is the
contact form defined by $\sas=\cf_\Ds(\X)$.  We let $\con_+(\Sm, \Ds)$ denote
the space Reeb vector fields of $(\Sm, \Ds)$ for which $\cf_\Ds(\X)$ is
positive with respect to the chosen orientation. We then have the following
fundamental definitions (see~e.g.~\cite{BGS}).

\begin{defn}\label{d:Sasaki-cone} Let $(\Sm, \Ds, J)$ be a strictly
pseudoconvex CR manifold. Then the \emph{Sasaki--Reeb cone} of $(\Sm, \Ds, J)$
is $\crJ_+(\Sm,\Ds,J):=\crJ(\Sm,\Ds,J)\cap\con_+(\Sm,\Ds)$. If
$\crJ_+(\Sm,\Ds,J)$ is nonempty then $(\Sm,\Ds,J)$ is said to be of
\emph{Sasaki type}, an element $\X\in\crJ_+(\Sm,\Ds,J)$ is called
\emph{Sasaki--Reeb vector field} or a \emph{Sasaki structure} on
$(\Sm,\Ds,J)$, and $(\Sm,\Ds,J, \X)$ is called a \emph{Sasaki manifold}. We
say $\X$ is \emph{quasiregular} if the flow of $\X$ generates an $\Sph^1$
action on $\Sm$, and moreover \emph{regular} if this action is free.
\end{defn}

The following well-known construction provides a standard way (see
e.g.~\cite{BG-book}) to extend geometric notions on K\"ahler manifolds to
Sasaki manifolds.

\begin{ex}\label{e:regular} Let $(M, J, g, \omega)$ be a K\"ahler manifold
such that $[\omega/2\pi]$ is an integral de Rham class.  Then there is a
principal $\Sph^1$-bundle $\pi\colon \Sm\to M$ with a connection $1$-form
$\cf$ satisfying $\d\cf = \pi^* \omega$. Thus $(\Sm,\Ds,J, \X)$ is a Sasaki
manifold, where $\Ds=\ker\cf\sub T\Sm$, $J$ is the pullback of the complex
structure on $TM$ to $\Ds\cong\pi^*TM$ and $\X$ is the generator of the
$\Sph^1$ action (with $\cf(\X)=1$, so $\cf=\cf_\Ds^\X$).

Conversely, if $\X \in\crJ_+(\Sm,\Ds,J)$ is (quasi)regular Sasaki--Reeb vector
field on $(\Sm, \Ds, J)$, then $\Sm$ is a principal $\Sph^1$-bundle (or
orbibundle) $\pi\colon \Sm\to M$ over a K\"ahler manifold (or orbifold)
$M$. Irrespective of regularity, this correspondence between K\"ahler geometry
and Sasaki geometry holds locally: any point of a Sasaki manifold
$(\Sm,\Ds,J,\X)$ has a neighbourhood in which the leaf space $M$ of the flow
of $\X$ is a manifold and has a K\"ahler structure $(g,J,\omega)$ induced,
using the local identification of $\Ds$ and $\pi^*TM$, by the
\emph{transversal K\"ahler structure} $(g_\X, J, \omega_\X)$ on $\Ds$, where
$\omega_\X:=\d\cf_\Ds^\X\restr\Ds$ and $g_\X:= \omega_\X(\cdot,
J\cdot)$. Indeed $g_\X$, $J$, and $\omega_\X$ are all $\X$-invariant, so they
all descend to $M$, and we refer to $(M,g,J,\omega)$ as a \emph{Sasaki--Reeb
  quotient} of $(\Sm,\Ds,J,\X)$.
\end{ex}

For $\X \in \crJ(\Sm, \Ds, J)$ we set
\begin{align*}
\con(\Sm, \Ds, J)^\X &:=\{\K \in  \con(\Sm,\Ds)\st [\X,\K]=0\}, \\
\crJ(\Sm, \Ds, J)^\X&:=\con(\Sm, \Ds, J)^\X \cap\crJ(\Sm, \Ds, J),  \\
\crJ_+(\Sm, \Ds, J)^\X & :=\con(\Sm, \Ds, J)^\X\cap\crJ_+(\Sm, \Ds, J).
\end{align*}
If in addition $\X \in \crJ_+(\Sm, \Ds, J)$ then
\[
C^\infty_\Sm(\R)^\X:=\{f\in C^\infty_\Sm(\R): \d f(\X)=0\}
\]
is a Lie algebra under the \emph{transversal Poisson bracket}
$\{f_1,f_2\}:=-\omega_\X^{-1} (\d f_1\restr\Ds,\d f_2\restr\Ds)$, and we have
the following elementary but central lemma (see~\cite{AC}).

\begin{lemma}\label{l:potential} The map $\K\mapsto f=\cf_\Ds^\X(\K)$ is a Lie
algebra isomorphism from $\con(\Sm, \Ds, J)^\X$ to $C^\infty_\Sm(\R)^\X$, and
$\K\in\crJ(\Sm, \Ds, J)^\X$ if and only if $f$ is a transversal Killing
potential for $(g_\X, \omega_\X)$, i.e., $-\omega_\X^{-1} (\d f)$ is a
transversal Killing vector field for $g_\X$.
\end{lemma}

\subsection{Transversal holomorphic structures}\label{ss:trans-hol}

Let $\Sm$ be a connected $(2m+1)$-manifold, let $\X$ be a nowhere zero vector
field, and let
\[
\beta_\X\colon T\Sm\to\Ds_\X:=T\Sm/\spn{\X}
\]
be the quotient of $T\Sm$ by the span of $\X$. Thus $\Ds_\X$ is everywhere
locally canonically isomorphic to the pullback of the tangent bundle of local
quotients of $\Sm$ by $\X$.

\begin{defn}\label{d:transversal-holomorphic} A \emph{transversal holomorphic
structure} on $(\Sm,\X)$ is a complex structure $J^\X$ on $\Ds_\X$ which is
  everywhere locally the pullback of a complex structure (i.e., an integrable
  almost complex structure) on a local quotient of $\Sm$ by $\X$.
\end{defn}
Any Sasaki structure $(\Ds,J, \X)$ on $\Sm$ induces a transversal holomorphic
structure on $(\Sm,\X)$: since $\Ds$ is transverse to $\X$,
$\beta_\X\restr\Ds$ is a bundle isomorphism $\Ds\to \Ds_\X$, and the complex
structure on $\Ds_\X$ induced by $J$ is $\X$-invariant and integrable
because $J$ is.

\begin{defn}\label{d:sasaki-polarization}~\cite{BGS}
A transversal holomorphic structure $J^\X$ on $(\Sm,\X)$ has \emph{Sasaki
  type} if $\Sm$ admits a Sasaki structure $(\Ds,J,\X)$ which is
\emph{compatible} with $(\X,J^\X)$, i.e.,
\[
\beta_\X\restr\Ds\colon\Ds\to \Ds_\X \text{ intertwines $J$ and $J^\X$}.
\]
Such compatible Sasaki structures on $(\Sm,\X,J^\X)$ are completely determined
by their contact forms or, equivalently, by the corresponding contact
distributions: we let $\cS(\X,J^\X)\leq \Omega^1(\Sm)$ be the subspace of
contact forms of Sasaki structures compatible with $(\X,J^\X)$, and write
$\cf_\Ds^\X$ for the unique element of $\cS(\X,J^\X)$ with
$\ker\cf_\Ds^\X=\Ds$.
\end{defn}

As is well-known, the geometry of local quotients of $\Sm$ may be described
using the basic de Rham complex
\[
\Omega^\bullet_\X(\Sm)=\{\alpha\in
\Omega^\bullet(\Sm): \imath_\X\alpha = 0 =\cL_\X \alpha\},
\]
with differential $\d_\X$ given by restriction of $\d$ (which preserves basic
forms); note that $\Omega^0_\X(\Sm)=C^\infty_\Sm(\R)^\X$---we also denote the
closed and exact forms in $\Omega^\bullet_\X(\Sm)$ by
$\Omega^\bullet_{\X,\cl}(\Sm)$ and $\Omega^\bullet_{\X,\mathrm{ex}}(\Sm)$
respectively. Any vector field $\Y$ with $\cL_\X \Y= f \X$ is projectable onto
local quotients, and hence induces Lie derivatives on both basic forms and
$\X$-invariant sections of $\Ds_\X$ (i.e., $\Z\in C^\infty_\Sm(\Ds_\X)$ with
$\cL_\X \Z=0$); these Lie derivatives depend on $\Y$ only via the
$\X$-invariant section $\beta_\X\Y$ of $\Ds_\X$, hence may be denoted
$\cL_{\beta_\X\Y}$.

A complex structure $J^\X$ on $\Ds_\X$ is thus a transversal holomorphic
structure if and only if $\cL_\X J^\X=0$ and for any $\X$-invariant section
$\Z$ of $\Ds$, $\cL_{J^\X \Z} J^\X=J^\X\circ \cL_\Z J^\X$. Note
that $\{\alpha\in \Wedge^kT^*\Sm\st \imath_\X\alpha=0\}$ is naturally isomorphic
to $\Wedge^k\Ds_\X^*$, so the $\X$-invariance of $J^\X$ implies that it
induces an operator on $\Omega^\bullet_\X(\Sm)$, as in complex geometry. We
thus obtain a twisted exterior derivative $\d^c_\X=J^\X\circ
\d_\X\circ(J^\X)^{-1}$ on $\Omega^\bullet_\X(\Sm)$. Also, the integrability
of $J^\X$ means equivalently that the basic $(1,0)$-forms generate a
differential ideal.

If $(\Sm,\X,J^\X)$ is compact of Sasaki type, then $\Omega^\bullet_\X(\Sm)$
has a Hodge decomposition with respect to the induced transversal metric of
any $\cf\in \cS(\X,J^\X)$, and the $\d_\X\d^c_\X$-lemma and transversal
K\"ahler identities are satisfied~\cite{elkacimi}, with the following
consequences, cf.~\cite{BG-book,BGS}.

\begin{lemma}\label{l:ddc} A $1$-form $\gamma$ on $\Sm$ is basic
with $J^\X$-invariant exterior derivative if and only if $\gamma=
\d_\X^c\varphi+\alpha$ for a basic function $\varphi\in \Omega^0_\X(\Sm)$ and
a closed basic $1$-form $\alpha\in\Omega^1_{\X,\cl}(\Sm)$\textup; further,
$\alpha$ is uniquely determined by $\gamma$, as is $\varphi$ up to an additive
constant. In particular\textup:
\begin{numlist}
\item $\cS(\X, J^\X)$ is an open subset of an affine space with translation
  group $\Omega^0_\X(\Sm) / \R\times \Omega^1_{\X,\cl}(\Sm)$, where
  $\Omega^0_\X(\Sm) / \R$ denotes the quotient of $\Omega^0_\X(\Sm)$ by
  constants\textup; furthermore this open subset is
  $\Omega^1_{\X,\cl}(\Sm)$-invariant and convex.
\item if $\Y$ is a projectable vector field, and $\cf\in \cS(\X, J^\X)$, then
  $\cL_\Y\cf$ is tangent to $\cS(\X, J^\X)$ \textup(at $\cf$\textup) if and
  only if
\begin{equation}\label{e:vec-param}
\Y =  J \X_\varphi + \Y^H + \X_f  + h \X
\end{equation}
where $\X_\varphi$ and $\X_f$ are contact vector fields with respect to $\cf$,
$J$ is the CR structure induced by $\cf$ with $J\X=0$, $\varphi$, $f$ and $h$
are basic functions uniquely defined \textup(by $\Y$ and $\cf$\textup) up to
$\varphi\mapsto\varphi+b$, $f\mapsto f+c$ and $h\mapsto h-c$ for constants
$b,c\in\R$, and $\alpha^H=\iota_{\Y^H}\d\cf$ is a uniquely defined basic
harmonic $1$-form with respect to the transversal metric induced by $\cf$.
Then $\cL_\Y\cf = \d_\X^c\varphi+\alpha^H+\d h$, so any tangent vector to
$\cS(\X, J^\X)$ arises in this way.
\end{numlist}
\end{lemma}
\begin{proof} If $\gamma=\d_\X^c\varphi+\alpha$ (as stated), it is basic
with $\d_\X\gamma=\d_\X\d_\X^c\varphi$, which is $J^\X$-invariant because
$J^\X \d_\X\d_\X^c\varphi= -\d_\X^c\d_\X\varphi=\d_\X\d_\X^c\varphi$.
Conversely if $\d_\X\gamma$ is $J^\X$-invariant, $\gamma$ is
$\d_\X\d^c_\X$-closed, so $\gamma=\d_\X^c\varphi+\alpha$ (as stated) by the
$\d_\X\d^c_\X$-lemma.
\begin{numlist}
\item Note that $\cf\in \Omega^1(\Sm)$ is in $\cS(\X,J^\X)$ if and only if
  $\cf(\X)=1$, $\cL_\X\cf=0$ and $\d\cf\in\Omega^2_\X(\Sm)$ is $J^\X$-invariant
  and $J^\X$-positive. Thus if also $\cf_0\in \cS(\X, J^\X)$,
  $\cf-\cf_0$ is basic with $J^\X$-invariant exterior derivative, hence
\begin{equation}\label{e:contact-param}
\cf = \cf_0 + \d_\X^c \varphi + \alpha
\end{equation}
with $\varphi\in \Omega^0_\X(\Sm)$ and $\alpha\in \Omega^1_{\X,\cl}(\Sm)$.
Conversely, given $\cf_0\in \cS(\X, J^\X)$, such an $\cf$ is in $\cS(\X,
J^\X)$ if $\d\cf$ is $J^\X$-positive---a convex open condition depending only
on $\cf$ mod $\Omega^1_{\X,\cl}(\Sm)$.
\item For any basic function $h$, $\cL_{h \X}\cf= \d(\cf(h\X))=\d h$, which is
tangent to $\cS(\X,J^\X)$. Hence if $\cL_\Y \cf$ is tangent to
$\cS(\X,J^\X)$, then $\Y = \Y_0 + \cf(\Y) \X$ with $\cf(\Y_0)=0$ and
$\cL_{\Y_0}\cf= \iota_{\Y_0}\d\cf$ tangent to $\cS(\X,J^\X)$, i.e., of the
form $\d^c\varphi+\alpha_0$. We write $\alpha_0=\alpha^H-\d f$ with $f$ a
basic function and $\alpha^H$ a basic harmonic form with respect to the
transversal metric induced by $\cf$. Thus $\Y_0= J \X_\varphi + \Y^H + \X_f -
f \X$ with $\iota_{\Y^H}\d\cf=\alpha^H$, and $\Y$ has the stated form with
$h=\cf(\Y)-f$. Clearly any such $\Y$ has
\[
\cL_\Y\cf=\d\iota_\Y \cf+\iota_\Y\d\cf=\d(f+h)+\d^c\varphi+\alpha^H- \d f=
\d^c\varphi+\alpha^H+\d h,
\]
which is tangent to $\cS(\X,J^\X)$, and the uniqueness claims follow
straightforwardly using the Hodge decomposition.  \qedhere
\end{numlist}
\end{proof}

\begin{rem}\label{r:equiGM} Lemma~\ref{l:ddc}(i) shows that $\cS(\X,J^\X)$
is connected. It follows that for any $\cf,\cf_0\in \cS(\X, J^\X)$, the
Gray--Moser theorem implies that there exists $\Phi\in\Diff_0(\Sm)$, the
identity component of the diffeomorphism group of $\Sm$, with
$\Phi^*\cf=\cf_0$. Lemma~\ref{l:ddc}(ii) is an infinitesimal version of this
fact: any element of $T_\cf\cS(\X,J^\X)$ is of the form $\cL_\Y\cf$ with $\Y$
as in~\eqref{e:vec-param} uniquely determined up to a contact vector field of
$\Ds=\ker\cf$. By Lemma~\ref{l:contact} we can fix the choice of $\Y$ with
$\cL_\Y\cf=\d^c\varphi+\alpha^H+\d h$ by subtracting the contact vector field
$\X_{\cf(\Y)}=\X_{f+h}$ to obtain a section
\[
\Z= J \X_\varphi + \Y^H - \X_h  + h \X
\]
of $\Ds=\ker\cf$. This yields a bijection between $T_\cf\cS(\X,J^\X)$ and
sections $\Z$ of $\Ds$ such that the basic $2$-form $\d\cL_\Z\cf=\cL_\Z\d\cf$
is $J^\X$-invariant.
\end{rem}

\begin{defn}\label{d:torus}
The diffeomorphism group $\Diff(\Sm)$ acts naturally on pairs $(\X,J^\X)$,
and we define the \emph{automorphism group} $\Aut(\Sm,\X,J^\X)$ to be the
stabilizer of $(\X,J^\X)$; its formal Lie algebra $\aut(\Sm,\X,J^\X)$
consists of $\X$-invariant vector fields $\Y$ with $\cL_{\beta_\X\Y}J^\X=0$.
\end{defn}

The automorphism group is infinite dimensional: its formal Lie algebra
contains an infinite dimensional abelian ideal $\lo(\Sm,\X)$ of $\X$-invariant
vector fields in the span of $\X$; however, the quotient
$\thol(\Sm,\X,J^\X):=\aut(\Sm,\X,J^\X)/\lo(\Sm,\X)$ may be identified
with the space of holomorphic sections of $\Ds_\X$ and hence is a finite
dimensional complex Lie algebra if $\Sm$ is compact~\cite{BGS}. In this case,
we let $\lO(\Sm,\X)$ denote the abelian normal subgroup of
$\Aut(\Sm,\X,J^\X)$ generated by the exponential image of
$\lo(\Sm,\X)$. However, the quotient group need not be well behaved: for
example if $\T$ is a torus in $\Aut(\Sm,\X,J^\X)$ containing $\X$, then
$\lO(\Sm,\X)\cap\T$ is the group generated by $\X$, which is not closed if
$\X$ is irregular.

By construction, $\Aut(\Sm,\X,J^\X)$ acts on $\cS(\X,J^\X)$ and hence
elements $\Y\in\aut(\Sm,\X,J^\X)$ have the form~\eqref{e:vec-param} with
respect to any $\cf\in\cS(\X,J^\X)$ and induce vector fields
$\cf\mapsto\cL_\Y\cf$ on $\cS(\X,J^\X)$. In these terms, elements of
$\thol(\Sm,\X,J^\X)$ have representatives in $\Ds$ of the form
$\Y=J\X_\varphi+\Y^H+\X_f-f\X$, and the complex structure on
$\thol(\Sm,\X,J^\X)$ (coming from the action of $J^\X$ on sections of
$\Ds_\X$) sends such a representative to $J\X_f +J\Y^H - \X_\varphi + \varphi
\X$.

\begin{rem}\label{r:trans-class} An element $h\X$ of $\lo(\Sm,\X)$ induces
the constant vector field $\cf\mapsto \d h$ on $\cS(\X,J^\X)$ and so if $\Sm$
is compact, elements of $\cS(\X,J^\X)$ related by an exact basic $1$-form $\d
h\in \Omega^1_{\X,\mathrm{ex}}(\Sm)$ induce isomorphic Sasaki structures via
the action of $\lO(\Sm,\X)$, and $\cS(\X,J^\X)/\lO(\Sm,\X)=
\cS(\X,J^\X)/\Omega^1_{\X,\mathrm{ex}}(\Sm)$.

On the other hand, elements of $\cS(\X,J^\X)$ related by a closed basic
$1$-form induce the same transversal K\"ahler geometry (of local quotients):
the transversal K\"ahler form $\d\cf\in\Omega^2_\X(\Sm)$ depends only on
$\cf+\Omega^1_{\X,\cl}(\Sm)\in \cS(\X,J^\X)/\Omega^1_{\X,\cl}(\Sm)$. Since any
two such forms are related by $\d_\X\d^c_\X\varphi$ for $\varphi\in
\Omega^0_\X(\Sm)$, the quotient $\cS(\X,J^\X)/\Omega^1_{\X,\cl}(\Sm)$ may be
viewed as the \emph{transversal K\"ahler class} of $(\Sm,\X,J^\X)$.
\end{rem}

This paper is concerned primarily with notions that depend only on the
transversal K\"ahler geometry (of local quotients) and so it is desirable to
fix the irrelevant translational gauge freedom by $\Omega^1_{\X,\cl}(\Sm)$, and
hence identify the transversal K\"ahler class with a subspace $\Sc$ of
$\cS(\X,J^\X)$. This prompts the following definition.

\begin{defn}\label{d:sasaki-slice} A \emph{marking} of a transversal
holomorphic manifold $(\Sm,\X,J^\X)$ of Sasaki type is an affine slice
$\Sc\sub\cS(\X,J^\X)$ to the action of $\Omega^1_{\X,\cl}(\Sm)$, i.e., a
subspace of the form
\begin{equation}\label{sasaki-marking}
\Sc=\Sc(\X,J^\X, \cf_0):=\{\cf_\varphi\in\cS(\X,J^\X) \st
\varphi\in\Omega^0_\X(\Sm) \}
\end{equation}
for some basepoint $\cf_0\in \cS(\X,J^\X)$, where
$\cf_\varphi:=\cf_0+\d_\X^c\varphi$. We then say $(\Sm,\X,J^\X,\Sc)$ is
\emph{marked}. Elements of
\begin{equation}\label{sasaki-potential}
\Sp(\X,J^\X, \cf_0):=\{\varphi\in \Omega^0_\X(\Sm) \st \cf_\varphi \in \Sc(\X,J^\X, \cf_0)\}
\end{equation}
will be called \emph{Sasaki potentials} with respect to $\cf_0$.
\end{defn}
For any marking $\Sc$, we have $\Sc=\Sc(\X,J^\X, \cf_0)$ for any $\cf_0\in
\Sc$. However, analogously to K\"ahler potentials, the space $\Sp(\X,J^\X,
\cf_0)$ of Sasaki potentials (which is an extension of $\Sc$ by $\R$ since
$\Sm$ is connected) depends upon the choice of basepoint $\cf_0$.

\begin{rem} The automorphism group $\Aut(\Sm,\X,J^\X)$ does not preserve a
marking $\Sc$. A more invariant object is the image of $\Sc$ in
$\cS(\X,J^\X)/\Omega^1_{\X,\mathrm{ex}}(\Sm)$: the fibres of the projection of
$\cS(\X,J^\X)/\Omega^1_{\X,\mathrm{ex}}(\Sm)$ to the transversal K\"ahler
class $\cS(\X,J^\X)/\Omega^1_{\X,\cl}(\Sm)$ are affine spaces modelled on the
basic de Rham cohomology group $H^1_\X(\Sm,\R)$, and the identity component
$\Aut_0(\Sm,\X,J^\X)$ acts trivially on basic de Rham cohomology, hence it
preserves the image of the marking.
\end{rem}

\subsection{The Reeb cone of a transversal torus}\label{ss:Reebcone}

Suppose $\Sm$ is compact with transversal holomorphic structure $(\X,J^\X)$;
then the automorphism group of any compatible Sasaki structure is
compact. Hence if $(\X,J^\X)$ has Sasaki type, there is a torus $\T \leq
\Aut(\Sm,\X,J^\X)$ whose Lie algebra $\tor$ contains $\X$, and we now fix such
a $\T$. The considerations of the previous subsection apply equally to the
space $\Omega^\bullet_\X(\Sm)^\T$ of $\T$-invariant basic forms, and the space
$\cS(\X,J^\X)^\T$ of $\T$-invariant elements of $\cS(\X,J^\X)$. This space is
acted upon by $\T$-equivariant automorphisms, i.e., by the centralizer
$\Aut(\Sm,\X,J^\X)^\T$ of $\T$ in $\Aut(\Sm,\X,J^\X)$. Also, for any basepoint
$\cf_0\in\cS(\X,J^\X)^\T$, we define
\begin{equation}\label{e:Tmark}
\Sc(\X,J^\X, \cf_0)^\T= \Sc(\X,J^\X,\cf_0) \cap \cS(\X,J^\X)^\T
\end{equation}
and a corresponding space $\Sp(\X,J^\X, \cf_0)^\T=\Sp(\X,J^\X,
\cf_0)\cap C^\infty_\Sm(\R)^\T$ of $\T$-invariant Sasaki potentials
with respect to $\cf_0$. We refer to a subspace of $\cS(\X,J^\X)^\T$
of the form~\eqref{e:Tmark} as a (\emph{$\T$-invariant})
\emph{marking}, with the $\T$-invariance generally being understood
from the context.

Any $\K\in\tor$ is a CR vector field for the CR structure induced by any
$\cf\in\cS(\X,J^\X)^\T$. Furthermore, we have the following observation.

\begin{lemma}\label{l:Reeb} Suppose that $\Sm$ is compact and $\K\in\tor$
is a Sasaki--Reeb vector field with respect to the CR structure induced by
some $\cf_0\in\cS(\X,J^\X)^\T$.  Then $\K$ is a Sasaki--Reeb vector field with
respect to the CR structure induced by any $\cf\in\cS(\X,J^\X)^\T$.
\end{lemma}
\begin{proof} By assumption $\cf_0(\K)>0$ on $\Sm$, and we need to show that
$f:=\cf(\K) >0$.  As $\K$ is contact with respect to $\Ds:=\ker\cf$, we have
$\K=f \X- (\d\cf\restr\Ds)^{-1}(\d f\restr\Ds)$ and hence
\[
\cf_0(\K) =
f - (\d\cf\restr\Ds)^{-1}(\d f\restr\Ds,\cf_0\restr\Ds).
\]
Evaluating this equation at a global minimum $\pt$ of $f$ we obtain $f(\pt) =
\cf_0(\K) (\pt) >0$.
\end{proof}

\begin{defn}\label{d:Reeb-cone} The \emph{Reeb cone} of $(\X, J^\X,\T)$ is the
cone~\cite{MSY2}
\begin{equation}\label{e:Reeb-cone}
\tor_+^\X := \{\K \in \tor \st  \cf(\K)>0 \}
\end{equation}
in $\tor$ determined by some (and hence any) $\cf \in \cS(\X,J^\X)^\T$.  Given
any such $\cf$, its \emph{contact momentum map} $\mu_\cf:\Sm \to \tor^*$ is
defined~\cite{lerman1} by $\langle\mu_\cf,\K\rangle = \cf(\K)$.
\end{defn}

Since $\langle\mu_\cf,\X\rangle\equiv 1$, the image of $\mu_\cf$ in $\tor^*$
does not contain $0$. Furthermore, if $\Sm$ is compact, results
of~\cite{FT,lerman1} therefore imply that $\Sigma:= \R^+\cdot \mu_\cf(\Sm)
\sub \tor^*$ is a strictly convex polyhedral cone (with the vertex $0\in
\tor^*$ removed) called the \emph{momentum cone} of $\cf$. Moreover, as
highlighted in \cite{MSY2}, its dual cone $\Sigma^*$ coincides with the Reeb
cone~\eqref{e:Reeb-cone}:
\begin{equation}\label{e:ReebCone=dualMOMcone}
\begin{split}
\Sigma^*&:= \{ \K\in \tor \st \forall\, x \in \Sigma,\; \ip{ x, \K} >0\,\}
= \{ \K\in \tor \st \ip{\mu_\cf,\K} >0 \, \mbox{ on } \Sm \}\\
&= \{ \K\in \tor \st \cf(\K) >0 \, \mbox{ on } \Sm \} = \tor_+^\X.
\end{split}
\end{equation} 
The image of $\mu_\cf$ itself is the transversal slice
\begin{equation}\label{e:transversal-polytope}
P_\X:=\mu_\cf(\Sm)= \{ x\in \Sigma \st \langle x,  \X\rangle =1\}
\end{equation}
of the momentum cone; it is thus a compact convex polytope which,
following~\cite{legendre2}, we call the \emph{transversal polytope} of $(\Sm,
\X, J^\X, \T)$---in \cite{MSY1}, it is instead called the \emph{characteristic
  polytope}. We thus obtain the following immediate consequences of
Lemma~\ref{l:Reeb}, Definition~\ref{d:Reeb-cone} and the
duality~\eqref{e:ReebCone=dualMOMcone}.

\begin{lemma}\label{l:momentum-cone-coincide} Let $(\Sm, \X, J^\X, \T)$
be a compact of Sasaki type. Then all $\cf\in\cS(\X,J^\X)^\T$ have the same
momentum cone $\Sigma$ and the same transversal polytope $P_\X$. Furthermore,
the duality~\eqref{e:ReebCone=dualMOMcone} identifies $\tor$ with the space of
affine-linear functions $\mbox{Aff}(P_\X,\R)$ on the transversal polytope
$P_\X$, with $\tor_+^\X$ being the subset of strictly positive such functions.
\end{lemma}

\begin{rems} The fact that $P_\X\sub \tor^*$ is independent of the choice
of $\cf \in \cS(\X, J^\X)^\T$ may also be deduced from the $\T$-equivariant
Gray--Moser isotopy argument as in Remark~\ref{r:equiGM}.

When $\X$ is quasiregular, $P_\X$ is the (natural) momentum polytope of the
Sasaki--Reeb quotient $(M, J, \omega)$ with respect to the induced action of
$\T/\Sph^1_\X$ where $\Sph^1_\X \leq \T$ denotes the $\Sph^1$-action generated
by $\X$.
\end{rems}

More generally, the above discussion may be reinterpreted in the transversal
K\"ahler geometry of $(\Sm, \X, J^\X, \T)$, where we have in particular the
following observation.

\begin{lemma}\label{l:potentials} Suppose that $\Sm$ is compact, and
that $\cf_0,\cf\in\cS(\X,J^\X)^\T$ are related by \eqref{e:contact-param};
then the induced transversal K\"ahler forms are related by $\omega = \omega_0
+ \d_\X\d^c_\X \varphi$, i.e., they belong to the same transversal K\"ahler
class.  Moreover any $\K\in\tor$ induces a transversal Killing field, denoted
$\check\K$, for both $\omega_0$ and $\omega$, whose respective transversal
Killing potentials $f_0 := \cf_0(\K)$ and $f := \cf(\K)$ are related by
\begin{equation}\label{e:potentials}
f=f_0 + \d^c \varphi (\check\K).
\end{equation}
\end{lemma}
\begin{proof} The first part is immediate from~\eqref{e:contact-param}, and
this also implies $f = f_0 + \d^c \varphi (\K) + \alpha(\K)$. Since $\alpha$
is closed with $\cL_\K\alpha=0$, the term $\alpha(\K)$ is a constant.  Using
that $\alpha$ is basic with respect to $\X$, the argument of
Lemma~\ref{l:Reeb} shows that $f_0 + \d^c\varphi(\K)$ and $f$ agree at a
global minimum of $f$, so $\alpha(K)=0$ and~\eqref{e:potentials} holds.
\end{proof}

\section{Sasaki structures via complex cones}\label{s:abstractcone}

There is a well-known link between Sasaki structures and conical K\"ahler
metrics, which is used in many places, see e.g.~\cite{BG-book,BHL,FOW,MSY2}.

Let $(\Sm, \Ds, J, \X)$ be a Sasaki manifold and $\R^+$ the multiplicative
group of positive real numbers; then $\Cm=\R^+\times\Sm$ has a K\"ahler
structure $(\hat J,\hat\omega,\hat g)$ defined as follows. Let $r\colon \Cm\to
\R^+$ and $\pi\colon \Cm\to \Sm$ denote the projections and let $\Z\in
C^\infty_\Cm(T\Cm)$ be the generator of the natural $\R^+$ action on $\Cm$; then
we have direct sum decompositions
\[
T\Cm \cong \spn\Z \oplus \pi^*T\Sm
= \spn\Z \oplus \pi^*(\spn\X) \oplus \pi^*\Ds
\]
and we let $\hat J$ be the unique almost complex structure on $\Cm$ which
restricts to $\pi^*J$ on $\pi^*\Ds$ and sends $\Z$ to $\pi^*\X$. Since $J$
induces a transversal holomorphic structure with respect to $\X$ on $\Sm$ and
$\d \cf_\Ds^\X\in\Omega^{(1,1)}_\X(\Sm)$, $\hat J$ is integrable: indeed,
using $\d r(\Z)=r$, we have $\d^c r= r\,\pi^* \cf_\Ds^\X$ and so $\d r +i J r
\,\pi^*\cf_\Ds^\X$ is a $(1,0)$-form on $\Cm$, as are pullbacks by $\pi$ of
basic $(1,0)$-forms on $\Sm$, and together these generate a differential
ideal. We now set
\begin{align*}
\hat\omega &:= \tfrac12 \d(r^2\pi^*\cf_\Ds^\X)=  r\d r\wedge\pi^*\cf_\Ds^\X
  +\tfrac12 r^2 \pi^*\d\cf_\Ds^\X,\\
\hat g &:= -\hat\omega\hat J
=\d r^2 + r^2\,\pi^*g_\Sm,\quad\text{where}\quad
g_\Sm:=(\cf_\Ds^\X)^2 + \tfrac12\d\cf_\Ds^\X(\cdot,J\cdot),
\end{align*}
and these evidently define a K\"ahler metric on $\Cm$ compatible with $\hat
J$.

\begin{rems} $(\Cm,\Z,\hat J,\hat\omega,\hat g)$ is sometimes called
the \emph{K\"ahler cone} of $(\Sm,\Ds, J, \X)$ and $g_\Sm$ its \emph{Sasaki
  metric}, due to the following more general considerations.
\begin{numlist}
\item Let $\Cm$ be any smooth manifold equipped with a vector field $\Z$; then
  a riemannian metric $\hat g$ or symplectic form $\hat\omega$ is
  \emph{conical} with respect to $\Z$ if $\cL_\Z \hat g= 2\hat g$ or $\cL_\Z
  \hat\omega = 2\hat\omega$ respectively. If in addition $\Z$ generates a free
  proper action of $\R^+$ (so that $\Cm/\R^+$ is a manifold) then
  $(\Cm,\Z,\hat g)$ or $(\Cm,\Z,\hat\omega)$ is called a \emph{riemannian} or
  \emph{symplectic cone} respectively.
\item For any riemannian manifold $(\Sm,g_\Sm)$, $(\R^+\times \Sm,Z,\d r^2 +
  r^2\pi^*g_\Sm)$ is a riemannian cone, where $Z$ generates the $\R^+$ action
  and $r,\pi$ are the projections (as above).  Conversely, if
  $(\Cm,\tilde\Z,\hat g)$ is a riemannian cone and $\tilde r^2:=\hat g(\tilde
  \Z,\tilde \Z)$, there is a unique metric $g_\Sm$ on $\Sm=\Cm/\R^+$ such that
  the quotient $\tilde \pi$ is a riemannian submersion from $(\Cm,\tilde
  r^{-2}\hat g)$ to $(\Sm,g_\Sm)$; then $(\tilde r,\tilde \pi)$ defines an
  isomorphism of riemannian cones (an $\R^+$-equivariant isometry) from
  $(\Cm,\tilde \Z,\hat g)$ to $(\R^+\times\Sm,Z,\d r^2 + r^2\pi^*g_\Sm)$. In
  particular, the riemannian cone of a Sasaki metric is K\"ahler and this is
  often used as a definition of Sasaki manifolds---see e.g.~\cite{MSY2}.
\item Any co-oriented contact manifold $(\Sm, \Ds)$ has a
  \emph{symplectization} $(\Ds^0_+,\Omega)$, where $\Ds^0_+ \sub T^*\Sm$ is
  the component of the annihilator of $\Ds \sub T\Sm$ distinguished by the
  co-orientation, and $\Omega$ is the pullback to $\Ds^0_+$ of the canonical
  Liouville symplectic form on $T^*\Sm$. Then (see e.g.~\cite{lerman1})
  $(\Ds^0_+,\Z,\Omega)$ is a symplectic cone, where $\Z$ is the generator of
  $\R^+$ action $(\lambda,p)\mapsto \lambda^2p$, for $\lambda\in \R^+$ and
  $p\in \Ds^0_+$ (thus $\Z$ is twice the generator of the natural action by
  scalar multiplication). Conversely any symplectic cone $(\Cm,\tilde
  \Z,\hat\omega)$ arises in this way up to isomorphism ($\R^+$-equivariant
  symplectomorphism).
\end{numlist}
\end{rems}

In this section, which is a synthesis of approaches in He--Sun~\cite{He-Sun},
Collins--Sz\'ekelyhidi~\cite{CSz} and Boyer--van Coevering~\cite{BV}, we study
the cone $\Cm$ from a complex viewpoint. In contrast to symplectic and
riemannian cones, $(\Cm,\Z,\hat J)$ cannot be constructed using only the
transversal holomorphic structure of $\Sm$, but this proves to be fortuitous
for our purposes, as the additional data needed is a marking in the sense of
Definition~\ref{d:sasaki-slice}.

\subsection{Complex cones and cone potentials}

In order to characterize complex cones $(\Cm,\Z,\hat J)$ and their compatible
K\"ahler cone metrics $(\hat g, \hat\omega)$, a key fact is that for any such
cone metric, the positive function $r$ defined by
\[
r^2 := \hat g( \Z, \Z)= \hat \omega( \Z, \hat J \Z)
\]
satisfies
\begin{equation}\label{cone-potential}
\cL_\X r =0, \quad \cL_\Z r = r, \quad \hat\omega = \tfrac{1}{4} \d\d^c (r^2),
\end{equation}
where $\X=\hat J\Z$. Thus the function $r$ suffices to define $\hat\omega$,
and hence $\hat g=-\hat\omega\hat J$, which will be a compatible K\"ahler
metric on $(\Cm,\hat J)$ provided $\d\d^c(r^2)>0$, i.e., $r^2$ is (strictly)
plurisubharmonic. Henceforth we drop the hat on $J$ and describe compatible
cone metrics purely in terms of such functions $r$. We also take the
vector field $\X$ (not $\Z$) as primitive.

\begin{defn}\label{d:cone-polarization} On a complex manifold $(\Cm,J)$ with
a holomorphic vector field $\X$, we let
\begin{equation}\label{e:cpot}
\cpot_\X(\Cm,J) :=
\{ r \in C^\infty_\Cm(\R^+) :\cL_\X r= 0,\  \cL_{-J \X} r =r, \ \d\d^cr >0\}
\end{equation}
be the space of (radial) \emph{cone potentials} ($\X$-invariant
plurisubharmonic functions of homogeneity $1$ with respect to $-J\X$).  If
$(\Cm,J,\X)$ admits a surjective cone potential $r\colon \Cm\to \R^+$, then
$\X$ is called a \emph{polarization} of $(\Cm,J)$ and $(\Cm,J,\X)$ is called a
\emph{polarized complex cone}.
\end{defn}

\begin{ex}\label{e:regular-cone}
Let $(\Sm,\Ds_0, J_\Sm, \X)$ be a regular Sasaki manifold as in
Example~\ref{e:regular}, and let $(M, J_M, \omega_0)$ be its Sasaki--Reeb
quotient, a compact K\"ahler manifold. Associated to the principal
$\Sph^1$-bundle $\pi\colon \Sm \to M$ is a holomorphic line bundle $\pi: L \to
M$, endowed with a hermitian product $h$, such that Chern curvature $2$-form
of $(L, h)$ is $\omega_0$.  It this situation, $\Sm$ can be identified with
the subset of unit vectors in the dual hermitian line bundle $(L^*, h^{-1})$,
with its induced hypersurface CR structure and Sasaki--Reeb vector field $\X=J
\Z$, where $\Z$ generates the natural $\R^+$-action on the bundle $L^*$.
Furthermore, the conical K\"ahler metric associated to the Sasaki manifold
$\Sm$ is isomorphic ($\R^+$-equivariantly isometric) to $((L^*)^\times,\Z)$
with the metric $(\hat\omega=\frac{1}{4}\d\d^c (r_h^2), \hat g = -\hat\omega
J)$, where $\Cm= (L^*)^\times$ is the complex manifold obtained from the total
space of $L^*$ by removing the zero section, and $r_h$ is the norm function on
$L^*$ defined by the hermitian metric $h^{-1}$.  Thus, in this case, the
polarized complex cone is $((L^*)^\times, \X)$ where $\X$ is the generator of
the fibrewise $\Sph^1$-action on $L^*$. Furthermore, we can identify the space
$\cpot_\X(\Cm,J)$ with the space of hermitian metrics on the ample line bundle
$L$ whose curvature form $\omega$ is positive definite, whereas the space
$\Sp( \X ,J^\X, \cf_0^\X)$, where $\cf_0^\X$ is the contact form of $(\Ds_0,
\X)$, is isomorphic to the space $\left\{ f \in C^\infty_M(\R)\st
\omega=\omega_0+ \d\d^cf >0\right\}$ of relative K\"ahler potentials on $(M,
J, \omega)$.
\end{ex}

We want to show more generally that $\cpot_\X(\Cm,J)$ is a basepoint free
version of the space $\Sp(\X ,J^\X, \cf_0^\X)$ of Sasaki
potentials~\eqref{sasaki-potential}. We begin with some straightforward
observations.

\begin{lemma}\label{l:cone-quotient} Let $(\Cm,J,\X)$ be a polarized complex
cone. Then\textup:
\begin{numlist}
\item $\Z:=-J\X$ generates a free proper action of $\R^+$ and we let
  $\Sm^\X=\Cm/\exp(-J\X)$ be the quotient manifold with quotient map
  $\pi_\X\colon \Cm\to \Sm^\X$\textup;
\item $\X$ descends to a vector field on $\Sm^\X$, which we also denote by
  $\X$, and $J$ to a transversal holomorphic structure $J^\X$ on $\Ds_\X=
  T\Sm^\X/\spn{\X}$\textup;
\item for all $r\in \cpot_\X(\Cm,J)$ and $c\in \R^+$, $\Sm^\X$ is
  diffeomorphic to $r^{-1}(c)$, so any such $r\colon \Cm\to \R^+$ is surjective
  and $(r,\pi_\X)\colon \Cm\to \R^+\times\Sm^\X$ is a diffeomorphism\textup;
\item $\Sm^\X$ is compact if and only if some, hence any, cone potential
  $r\colon \Cm\to \R^+$ is proper.
\end{numlist}
Conversely if $(\Cm,J)$ is a complex manifold with a holomorphic vector field
$\X$ such that $-J\X$ generates a proper action of $\R^+$, then any cone
potential is surjective, and so $(\Cm,J,\X)$ is a polarized complex cone if
and only if $\cpot_\X(\Cm,J)$ is nonempty.
\end{lemma}
\begin{proof} (i) and (iii) follow because $-J\X$ is transverse to the level
surfaces of any cone potential $r$, (ii) because $\X,\Z$ are holomorphic with
$[\X,\Z]=0$, (iv) by definition, and the converse by the equivariance of cone
potentials with respect to the $\R^+$ action.
\end{proof}
For any $r \in C^\infty_\Cm(\R^+)$, we set $\Sm_r:=r^{-1}(1)$
and $\tilde\cf_r:= \d^c r/r$.
\begin{lemma}\label{l:cone-potentials} Let $(\Cm,J)$ be a complex manifold
with a holomorphic vector field $\X$ and a function $r \in C^\infty_\Cm(\R^+)$
such that $\cL_\X r=0$ and $\cL_{-J\X} r =r$. Then $\tilde\cf_r$ is the
pullback $\pi_\X^*\cf_r$ of a $1$-form $\cf_r$ on $\Sm^\X$ with $\cf_r(\X)=1$
and $\cL_\X\cf_r=0$, hence also $\iota_\X\d\cf_r=0$. Further, for any
$\lambda\in\R^+$ the following are equivalent\textup:
\begin{numlist}
\item $r$ is plurisubharmonic, i.e., is in $\cpot_\X(\Cm,J)$\textup;
\item $r^\lambda$ is plurisubharmonic, hence is in
$\cpot_{\X/\lambda}(\Cm,J)$\textup;
\item $\cf_r\in \cS(\X,J^\X)$.
\end{numlist}
Thus $r\mapsto r^\lambda$ is a bijection from $\cpot_\X(\Cm,J)$ to
$\cpot_{\X/\lambda}(\Cm,J)$, and if $r\in\cpot_\X(\Cm,J)$, $\pi_\X|_{\Sm_r}$
is a CR isomorphism from $\Sm_r$ \textup(with the hypersurface CR
structure\textup) to $\Sm^\X$ \textup(with the CR structure induced by $J^\X$
and $\Ds_r=\ker\cf_r$, i.e., by $\cf_r\in \cS(\X,J^\X)$\textup).
\end{lemma}
\begin{proof} Set $\rho=\log r$. For the first part, we compute $\cL_\X\rho=0$
and $\cL_{-J\X}\rho =1$, from which it follows that $\tilde\cf_r$ is basic
with respect to the holomorphic vector field $\Z=-J\X$. Hence it is a pullback
as stated, and the properties of $\cf_r$ follow straightforwardly. For the
equivalence, observe that the function $r^\lambda=e^{\lambda\rho}$ satisfies
$\cL_{\X/\lambda}(r^\lambda)=0$, $\cL_{-J\X/\lambda}(r^\lambda)= r^\lambda$,
and
\[
\d \d^c (r^\lambda) = \lambda\, \d (r^\lambda \tilde\cf_r) = \lambda r^\lambda
\bigl(\lambda (\d r/r)\wedge\pi_\X^*\cf_r +\pi_\X^*\d\cf_r\bigr)
\]
This is positive if and only if $\d\cf_r$ is positive on $\ker\cf_r$. Thus
(ii) is equivalent to (iii), hence also (i) by taking $\lambda=1$.

The bijection follows immediately, and the last part follows because
$T\Sm_r=\ker\d r|_{\Sm_r}=\ker\d \rho|_{\Sm_r}$, so the hypersurface CR structure
has contact distribution $\ker\tilde\cf_r|_{\Sm_r}$ with the complex structure
induced by $J$, hence is $\pi_\X$-related to $\Ds_r$ with the complex
structure induced by $J^\X$.
\end{proof}
\begin{cor}\label{c:sasaki-type} If $(\Cm,J,\X)$ is a polarized complex cone,
then the induced transversal holomorphic structure $(\X,J^\X)$ on $\Sm^\X$ has
Sasaki type.
\end{cor}

\begin{rem} For any cone potential $r\in \cpot_\X(\Cm,J)$, and
any $a,\lambda\in \R^+$, $\d\d^c (ar^{\lambda})$ is thus a K\"ahler metric on
$(\Cm,J)$. To obtain a conical K\"ahler metric of homogeneity $2$ with respect
to $\Z=-J\X$, we take $\lambda=2$ and (conventionally) $c=1/4$ so that the
K\"ahler form is
\[
\omega_r:=\tfrac14 \d\d^c(r^2)=\tfrac12 \d(r^2 \tilde\cf_r)
= r\,\d r\wedge\tilde\cf_r+\tfrac12 r^2 \d \tilde\cf_r,
\]
and $\frac12 r^2 = \frac12 \omega_r(\X, J\X)$ is a Killing potential for
$\X$. However, we generally find it simpler to work directly with the cone
potential $r$ than a conventional choice of conical metric.
\end{rem}

\begin{lemma}\label{l:cone-vs-sasaki-pot} Let $(\Cm,J,\X)$ be a polarized
complex cone, $r_0\in \cpot_\X(\Cm,J)$, $r\in C^\infty_\Cm(\R^+)$, and
$\cf_0:=\cf_{r_0}\in \cS(\X,J^\X)$. Then $r\in \cpot_\X(\Cm,J)$ if and only if
there exists a Sasaki potential $\varphi\in\Sp(\X ,J^\X, \cf_0)$ such that
$r=\exp(\pi_\X^*\varphi)r_0$, and then $\cf_r=\cf_0+\d^c\varphi$.
\end{lemma}
\begin{proof} We may write $r=\exp(\tilde\varphi)r_0$, for
some $\tilde\varphi\in C^\infty_\Cm(\R)$; then $\cL_{-J\X}r=r$ and $\cL_\X r
=0$ hold if and only if $\tilde\varphi=\pi_\X^*\varphi$ with
$\cL_\X\varphi=0$. We then have $\cf_r=\cf_0+\d^c\varphi$ and the result
follows by Lemma~\ref{l:cone-potentials}.
\end{proof} 
Note that $\varphi$ in Lemma~\ref{l:cone-vs-sasaki-pot} is determined by
$\pi_\X^*\varphi=\log r - \log r_0$. Thus the above observations yield the
following crucial descriptions of $\cpot_\X(\Cm,J)$.

\begin{prop}\label{p:links=potentials} Let $(\Cm,J,\X)$ be a polarized complex
cone. Then\textup:
\begin{numlist}
\item the image of $r\mapsto \cf_r\colon \cpot_\X(\Cm,J)\to \cS(\X,J^\X)$ is a
  marking $\Sc$ of $(\Sm,\X,J^\X)$ in the sense of Definition~\textup{\ref
    {d:sasaki-slice}}, and $\cf_r = \cf_{r_0}$ if and only if $r=e^a r_0$ with
  $a\in \R$ constant\textup;
\item for any $r_0\in \cpot_\X(\Cm,J)$, we have $\Sc=\Sc(\X,J^\X,\cf_0)$
with $\cf_0=\cf_{r_0}$, and the map
\begin{equation} \label{potential-map}
  r \in \cpot_\X(\Cm,J) \mapsto \varphi\in C^\infty_\Sm(\R)
  \quad\text{with}\quad \pi_\X^*\varphi=\log r - \log r_0
\end{equation}
is a bijection from $\cpot_\X(\Cm,J)$ to the space $\Sp(\X ,J^\X, \cf_0)$ of
Sasaki potentials~\eqref{sasaki-potential} on $\Sm^\X$\textup;
\item $r\mapsto \Sm_r$ is a bijection from $\cpot_\X(\Cm,J)$ to the set of
strictly pseudoconvex \textup(images of\textup) $\X$-invariant sections of
$\pi_\X$.
\end{numlist}
\end{prop}
\begin{proof} (i)--(ii) are immediate from Lemma~\ref{l:cone-vs-sasaki-pot}
and the connectedness of $\Cm$.

(iii) By Lemma~\ref{l:cone-potentials}, for any $r\in \cpot_\X(\Cm,J)$,
$\Sm_r$ is strictly pseudoconvex and is the image of an $\X$-invariant section
of $\pi_\X$.  Conversely given such a section, the condition that $r=1$ on the
image and that $\cL_{-J\X} r=r$ determine that $r=e^\rho$ where $\rho(y)$ is
the time for the flow of $J\X$ starting at $y\in \Cm$ to meet the
section. Since the section is $\X$-invariant and $\X$ is holomorphic, $r$ so
defined is a smooth function on $\Cm$ with $\cL_\X r = 0$ and $\cL_{-J\X} r
=r$; now $\d\d^c r$ is positive at $r=1$ by pseudoconvexity of the image of
the section, hence it is positive on $\Cm$, since it is homogeneous of degree
$1$ with respect to $\Z=-J\X$.
\end{proof}

Part (i) of this result shows that there is a canonical bijection up to
isomorphism\footnote{It is only an equivalence of categories if one works with
  the quotient groupoid of $\R^+$-equivariant isomorphisms of polarized
  complex cones by the action of $\R^+$.} between polarized complex cones
$(\Cm,J,\X)$ and marked transversal holomorphic manifolds $(\Sm,\X,J^\X,\Sc)$,
where an isomorphism of polarized complex cones is an $\R^+$-equivariant
biholomorphism of the corresponding manifolds, and an isomorphism of marked
transversal holomorphic manifolds is a diffeomorphism of corresponding
manifolds intertwining their transversal holomorphic structures and markings.

Part (ii) shows that $\cpot_\X(\Cm,J)$ is an (infinite dimensional) affine
manifold with a global affine chart $\Sp(\X ,J^\X, \cf_0)$---indeed the space
$\{\rho=\log r\st r\in\cpot_\X(\Cm,J)\}$ of logarithmic cone potentials is
an affine space modelled on $\Sp(\X ,J^\X, \cf_0)$.

\subsection{Reeb cones revisited}\label{ss:Reebcone-rev} 

Let $(\Cm, J, \X)$ be a polarized complex cone with corresponding marked
transversal holomorphic manifold $(\Sm^\X,\X,J^\X,\Sc)$, assumed compact.  In
particular $(\X,J^\X)$ has Sasaki type (Corollary~\ref{c:sasaki-type}) and so,
as in \S\ref{ss:Reebcone}, we may fix a compact torus $\T \leq
\Aut(\Sm,\X,J^\X)$, with $\X \in \tor= {\rm Lie}(\T)$, which preserves some
compatible Sasaki structure $\cf_0\in \Sc$. Thus $\X$, $J^\X$, and
$\Sc\cong\Sc(\X,J^\X,\cf_0)$ are $\T$-invariant, hence there is an induced
holomorphic action of $\T$ on $(\Cm, J, \X)$. We denote by $\Sc^\T$ the space
of $\T$-invariant elements of $\Sc$ (thus $\Sc^\T\cong \Sc(\X,J^\X,\cf_0)^\T$
for any $\cf_0\in \Sc^\T$) and by $\cpot_\X(\Cm,J)^\T$ the space of
$\T$-invariant cone potentials in $\cpot_{ \X}(\Cm,J)$. The considerations of
the previous subsection apply \emph{mutatis mutandis} in the $\T$-invariant
context to show that
\begin{numlist}
\item $r\mapsto\cf_r$ is a surjection with $1$-dimensional fibres from
  $\cpot_\X(\Cm,J)^\T$ to $\Sc^\T\sub  \cS(\X,J^\X)^\T$;
\item for any $\cf_0\in \Sc^\T$,  $\cpot_\X(\Cm,J)^\T$ is isomorphic to
  the space $\Sp(\X,J^\X,\cf_0)^\T$ of $\T$-invariant Sasaki potentials;
\item $r\mapsto \Sm_r$ identifies $\cpot_\X(\Cm,J)^\T$ with the space of
  $\T$-invariant sections of $\pi_\X\colon \Cm\to \Sm^\X$.
\end{numlist}

Let $\tor_+^\X\sub \tor$ the Reeb cone of $\T$ (Definition~\ref{d:Reeb-cone}).
An important observation in~\cite{BV,CSz} is that $\tor_+^\X$ can be defined
intrinsically on $(\Cm,J,\T)$, and indeed can be identified with
\begin{equation}\label{e:intrinsic-Reeb}
\tor_+:=\{K\in \tor\st K \text{ is a polarization of } (V,J) \}.
\end{equation}
To establish this, it is useful to reformulate the definition of $\tor_+^\X$
on $\Cm$ (see e.g.~\cite{He-Sun}).
\begin{lemma}\label{l:complexified-action} For any $r\in \cpot_\X(\Cm,J)^\T$,
\begin{equation}\label{polarizations}
 \tor_+^\X =\{ \K\in \tor \st \cL_{-J \K}r >0\}
=\{K\in \tor\st JK \text{ is transverse to } \Sm_r:=r^{-1}(1)\}.
\end{equation} 
Consequently, $(\Cm,J)$ has a proper holomorphic action of a complex torus
$\T_\C$ whose real part coincides with $\T$. Furthermore $\tor_+^\X\sub
\tor_+$, and $\K\in \tor_+^\X$ if and only if $\X\in\tor_+^\K$.
\end{lemma}
\begin{proof}
Since $\cL_{-JX} r = (\d^c r)(X) = r (\pi_X^*\cf_r)(X)$, the first equality
in~\ref{polarizations} is immediate, and the second follows because
$\pi_X^*\cf_r$ is a pullback and $\Sm_r$ is a section of $\pi_\X$.

It follows that for any $\K\in \tor_+$, $J \K$ is a complete vector field on
$(\Cm,J)$. As $\tor_+$ is open in $\tor$, this gives the second claim.  Now
for $K\in \tor_+^\X$, $\Sm_r$ is a $\T$-invariant section of $\pi_\K$ which
yields a surjective cone potential $\tilde r\in \cpot_K(\Cm,J)$ by
Proposition~\ref{p:links=potentials}(iii) and Lemma~\ref{l:cone-quotient}, so
$\K\in \tor_+$. As $\Sm_{\tilde r}:=\tilde r^{-1}(1) = \Sm_r$ is a section of
$\pi_\X$, $\X\in \tor_+^\K$; the converse holds symmetrically.
\end{proof}
\begin{rem}\label{r:connected-Reeb} It follows
straightforwardly (cf.~\cite[Lemma~2.3]{He-Sun}) that $\tor_+^\X$ is a
connected component of $\tor_+$. Indeed, for $\K\in \tor_+$, the condition to
be in $\tor_+^\X$ is open because $\Sm_r$ is compact. Suppose now that $\K$ is
a limit point of $\tor_+^\X$ which is not in $\tor_+^\X$. Then
$\pi_\X^*\cf_r(\K)\geq 0$ on $\Sm_r$, but $\exists\,\pt\in \Sm_r$ with
$\cf_r(\K)(\pi_\X(\pt))=0$. Since this is a minimum
$\d(\cf_r(\K))_{\pi_\X(\pt)}=0$ and hence $\d\cf_r(\K,\cdot)_{\pi_\X(\pt)}=0$ as
$\K\in \tor$ preserves $\cf_r$. If $\K\in\tor_+$, with surjective cone
potential $\tilde r\in\cpot_\K(\Cm,J)$, then
\[
\d \tilde r\wedge\pi_\K^*(\cf_{\tilde r}\wedge\d\cf_{\tilde r}^{\,\wedge m})
= f\,\d r\wedge \pi_\X^*(\cf_r\wedge \d\cf_r^{\,\wedge m})
\]
for some function $f$ on $\Cm$. Contracting with $\K$ and evaluating at $p$
then yields $(\d\cf_{\tilde r}^{\wedge m})_{\pi_\K(\pt)}=0$, which contradicts
Lemma~\ref{l:cone-potentials}. Thus $\tor_+^\X$ is both open and closed in
$\tor_+$.
\end{rem}

To prove $\tor_+\sub \tor_+^\X$, we proceed indirectly,
following~\cite{BV,CSz}, by introducing another equivalent definition of the
Reeb cone. To explain this, we first observe if $\Lam\leq \tor$ is the lattice
of circle subgroups of $\T$, then $(\R^+\cdot\Lam)\cap\tor_+$ is dense in
$\tor_+$, and we may thus assume $\X$ is quasiregular ($\tor_+^\X$ varies
continuously with $\X$ by~\eqref{polarizations}---indeed it is locally
constant in $\X$ by Remark~\ref{r:connected-Reeb}). Then, similarly to the
regular case (Example~\ref{e:regular-cone}), $\Cm$ can be identified with the
space of nonzero vectors in the dual of an orbi-ample orbiline bundle $L$ over
the compact K\"ahler orbifold $(M, J_M, \omega_0)$ which is the Sasaki--Reeb
quotient of $\Sm$ by the circle action generated by $\X$. Thus $\Cm$ has a
natural one point compactification $\Cv=\Cm\cup\{0\}$ which the singular space
obtained by blowing down the zero section in the total space of $L^*$. Since
cone potentials correspond to norms of hermitian metrics on $L$, for any $r\in
\cpot_\X(\Cm,J)$, the apex $0$ of the cone $\Cv$ is characterized as the limit
of points $\pt\in\Cm$ with $r(\pt)\to 0$. Let
\begin{equation}\label{e:hol}
\Hol \cong \bigoplus_{k\in\N} H^0(M,L^k)
\end{equation}
be the space of continuous complex-valued functions of $\Cv$ which are
holomorphic on $\Cm\cong (L^*)^\times$ and polynomial on
each fibre of $(L^*)^\times$ over $M$. Here $H^0(M,L^k)$ is the space of holomorphic sections $s$ of $L^k$, which define fibrewise
polynomial functions $f_s$ on $L^*$ by $f_s(\pt)=\ip{s,\pt^k}$. Now $\T_\C$
acts on $\Hol$ and for $\alpha\in\tor^*$ we let
\[
\Hol_\alpha:= \{f\in \Hol\st \forall\,\K\in \tor,\;\cL_{-J\K}f = \alpha(\K) f\}
\]
be the $\alpha$-weight space and
\[
\Gamma:= \{\alpha\in \tor^*\st \Hol_\alpha\neq 0\}
\]
be the set of (integral) weights of the action. Thus there is a weight space
decomposition
\begin{equation}\label{e:hol-alpha}
\Hol \cong \bigoplus_{\alpha\in \Gamma} \Hol_\alpha,
\end{equation}
where the degree $k$ component of~\eqref{e:hol} is the direct sum of the
weight spaces $\Hol_\alpha$ with $\alpha(\X)=k$. The key fact we need is that
(by ampleness of $L$) $\Hol$ separates points of $\Cv$: in particular for any
$p\in\Cv$ there exists $f\in \Hol$ with $f(0)=0$ and $f(\pt)\neq 0$.

\begin{rem}\label{r:regular} In fact orbifold versions of the Kodaira embedding
theorem (see~\cite{RT}) embed $\Cv$ as an affine variety in $\C^N$ which is a
cone with the singular apex at the origin. The functions in $\Hol$ are then
the regular functions on $\Cv$, which separate points by definition.
\end{rem}

The incisive idea in~\cite{CSz} is to define the Reeb cone in $\tor$ as dual
to the cone generated by $\Gamma$ in $\tor^*$, i.e., to set
\[
\tor_+^\Gamma = \{\K \in \tor\st \forall \alpha\in \Gamma\setminus\{0\},\;
\alpha(\K)>0\}.
\]
Now~\cite[Prop. 2.1]{CSz} asserts that $\tor_+^\Gamma=\tor_+^\X$,
while~\cite[Prop. 2]{BV} asserts that $\tor_+^\Gamma=\tor_+$, which are the
results we need. However,~\cite{BV} refers to~\cite{CSz} for the inclusion
$\tor_+\sub \tor_+^\Gamma$, while we have been unable to verify the details of
the argument in~\cite{CSz} that $\tor_+^\Gamma\sub\tor_+^\X$. Since these
claims are crucial for our work, for the convenience of the reader we present
proofs here.
\begin{lemma}\label{l:reeb} If $\X\in \tor_+$ is quasiregular
then $\tor_+\sub\tor_+^\Gamma$ and $\tor_+^\Gamma\sub \tor_+^\X$.
\end{lemma}
\begin{proof} For $\tor_+\sub\tor_+^\Gamma$, we follow the main ideas
in the proof of~\cite[Prop. 2.1]{CSz}. Let $\K\in \tor_+$ and $r\in
\cpot_\K(\Cm,J)$. As noted already, $\cL_{-J\K} r =r$ implies $r$ is
equivariant with respect to the $\R^+$ action generated by $-J\K$.  Thus under
the flow $\flow_t^{-J\K}$ of $-J\K$, for any $\pt\in\Cm$,
\[
r(\flow_t^{-JK}(\pt))\to 0\text{ as } t\to 0
\]
and so $\flow_t^{-JK}(\pt)\to 0$, the apex of $\Cv$. Now let
$\alpha\in\Gamma\setminus\{0\}$ and $f\in \Hol_\alpha\setminus\{0\}$.  Then
$\cL_{-J\K} f = \alpha(K) f$, $f(0)=0$, and there exists $\pt\in \Cm$ with
$f(\pt)\neq 0$. Since
\[
\frac{\d}{\d t} f(\flow_t^{-JK}(\pt)) = \alpha(\K) f(\flow_t^{-JK}(\pt))
\]
and $f(\flow_t^{-JK}(\pt))\to f(0)=0$ as $t\to 0$, we must have $\alpha(\K)>0$.
Thus $\K\in \tor_+^\Gamma$.

For $\tor_+^\Gamma\sub \tor_+^\X$, we give a novel argument. Let $\K\in
\tor_+^\Gamma$ and $r\in \cpot_\X(\Cm,J)$. By~\eqref{e:ReebCone=dualMOMcone},
it suffices to show that $\K$ is in the dual of the momentum cone of the
contact momentum map of $\cf_r$ on $\Sm^\X$ (Definition~\ref{d:Reeb-cone}).
For this, we first let $\pt\in \Sm_r$ have maximal stabilizer with respect to
the $\T$-action, i.e., $\tor=\mathfrak{stab}_\T(\pt)\oplus \spn{\X}$. Thus
$\{\Y_\pt\st \Y\in \tor\} =\spn{\X_\pt}$ and $\K_\pt-\cf_r(\K)_{\pi_\X(\pt)} \X_\pt
\in \mathfrak{stab}_\T(\pt)$. Since $\Hol$ separates points, there exists
$f\in \Hol_\alpha$ with $\alpha\in \Gamma\setminus\{0\}$ and $f(\pt)\neq
0$. Let $\flow_t^{-J\K}$ be the flow of $-J\K$ as before and observe that
\[
f\bigl(\flow_t^{-J\K}(\pt)\bigr) = f\bigl(\flow_t^{-\cf_r(\K)_{\pi_\X(\pt)} J\X}(\pt)\bigr)
\]
so that
\[
\alpha(\K) f(\pt) = \cf_r(\K)_{\pi_\X(\pt)} \alpha(\X) f(\pt).
\]
Since $\alpha(\X)>0$ and $\alpha(\K)>0$ by assumption, we have
$\cf_r(\K)_{\pi_\X(\pt)}>0$. We conclude that $\ip{\xi,\K}>0$ at points
$\xi=\mu_{\cf_r}(\pt)$ on the edges of the momentum cone. Since the momentum
  cone is a strictly convex polyhedral cone, the result follows.
\end{proof}

Since $\tor_+^\X\sub\tor_+$ by Lemma~\ref{l:complexified-action}, it follows
that $\tor_+^\X=\tor_+^\Gamma=\tor_+$.  Henceforth we drop superscripts and
refer to $\tor_+$ (with these interchangeable definitions) as the \emph{Reeb
  cone} of $(\Cm,J)$. We also note that the momentum cone of any
$\cf\in\cS(\X,J^\X)$ coincides with the weight cone generated by $\Gamma$.

\begin{rem} For any $r\in \cpot_\X(\Cm,J)^\T$, $\hat\omega:=
\frac14 \d\d^c(r^2)$ is a $\T$-invariant exact symplectic form (see
Lemma~\ref{l:cone-potentials}), so the $\T$-action on $(\Cm,\hat\omega)$ is
hamiltonian with momentum map $\mu_r\colon\Cm\ra \tor^*$ defined by
\begin{equation}\label{e:coneMOMmap}
  \langle \mu_r, \K\rangle = \tfrac{1}{4} \d^c(r^2) (\K)=\tfrac12 r\d^c r(Y).
\end{equation}
for $\K\in \tor$. It is in fact the unique momentum map which is homogeneous
of order $2$ with respect to the $\R^+$-action generated by $-J \X$. In our
setting, for $\pt\in\Sm_r$,
\[
\mu_r(\pt) = \tfrac{1}{2}\mu_{\cf_r}(\pi_\X(\pt)).
\]

In particular $\d^c r(\K)=r$ so $\frac12 r^2$ is a Killing
potential for $\K$. Hence the flow of $-J\K$ is the gradient flow of the
positive function $\frac12r^2$ with respect to the K\"ahler metric
$\hat g=-\hat\omega J$.

Observe that $\langle \mu_r(\pt), \X\rangle = 1/2$ for any $r\in
\cpot_\X(\Cm,J)^\T$ and for any $\pt\in \Sm_r$ thanks to the homogeneity
condition~\eqref{cone-potential}, so $N_r$ has momentum image $(1/2) P_\X$.
\end{rem}

\subsection{Change of polarization}\label{ss:Thetamap}
The $\T$-invariant version of Proposition~\ref{p:links=potentials}(iii) and
the final part of Lemma~\ref{l:complexified-action} have the following crucial
corollary, which develops an idea used by He--Sun in the proof of~\cite[Lemma
  2.2]{He-Sun}.

\begin{cor}\label{bijection-potentials}
Let $(\Cm, J, \X)$ be the polarized complex cone of a compact $\T$-invariant
Sasaki manifold $(\Sm, \Ds, J, \X)$ with $X\in\tor$, and let $\K \in
\tor_+$. Then there is a bijection
\begin{equation}\begin{split}
\IsoPot\colon \cpot_\X(\Cm,J)^\T &\to \cpot_\K(\Cm,J)^\T\\
r\qquad &\mapsto \qquad \tilde r
\end{split}\end{equation}
between $\T$-invariant cone potentials with respect to $\X$ and such
potentials with respect to $\K$, characterized uniquely by
$\IsoPot(r)= \tilde r$ if and only if $N_{\tilde r}:= \tilde r^{-1}(1)$
is equal to $N_r:= r^{-1}(1)$.
\end{cor}

Fixing base points $r_0 \in \cpot_\X(\Cm,J)^\T$ and $\tilde r_0:=\IsoPot(r_0)
\in \cpot_\K(\Cm,J)^\T$, we also have  a bijection
\begin{equation}\label{e:bij-pot}\begin{split}
\hTheta\colon \Sp( \X,J^\X, \cf_{r_0})^\T &\to \Sp( \K,J^\K, \cf_{\tilde r_0})^\T\\
\varphi&\mapsto \psi\quad\text{with}\quad \IsoPot(e^\varphi r_0) = e^\psi
\tilde r_0
\end{split}\end{equation}
such that $\varphi$ and $\hTheta(\varphi)$ induce Sasaki structures with the
same underlying CR structure.

\begin{rem} From the contact viewpoint (see Remark~\ref{r:dpq}),
the images of $\Sc( \X,J^\X,\cf_{r_0})^\T$ and $\Sc(
\K,J^\K,\cf_{\tilde r_0})^\T$ in $\Cx_+(\Sm, \Ds_0)^\T$ are both
formal complexifications of the same orbit of the group $\Con(\Sm,
\Ds_0)^\T$ on this space, which does not depend on the choice $\X$ or
$\K$ of Reeb vector field. Hence we expect these images to
coincide up to isomorphism.  Thus Corollary~\ref{bijection-potentials}
provides an explicit realisation of this correspondence on the level
of potentials.
\end{rem}

We now compute the derivatives of $\IsoPot$ and $\hTheta$. To do this, we note
that $T_r\cpot_\X(\Cm,J)^\T\cong C^\infty_{\Sm^X}(\R)^\T$, which is isomorphic
to $C^\infty_{\Sm_r}(\R)^\T$ using pullback by
$\Psi^\X_r:=\pi_\X|_{\Sm_r}\colon \Sm_r\cong \Sm^\X$.

\begin{lemma}\label{dTheta} Suppose $r_t$ is a smooth curve
in $\cpot_\X(\Cm, J)^\T$ with derivative $\dot r_t\in
C^\infty_{\Sm^X}(\R)^\T$.  Then $\tilde r_t:= \IsoPot(r_t)$ is a smooth curve
in $\cpot_\K(\Cm, J)^\T$ with derivative
\[
\dot{\tilde r}_t = \bigl(\Psi^\X_{r_t}\circ (\Psi^\K_{r_t})^{-1}\bigr)^*
\Bigl(\frac{\dot r_t}{\cf_{r_t}(\K)}\Bigr);
\quad\text{in addition}\quad
\cf_{\tilde r_t}=\bigl(\Psi^\X_{r_t}\circ (\Psi^\K_{r_t})^{-1}\bigr)^*
\Bigl(\frac{\cf_{r_t}}{\cf_{r_t}(\K)}\Bigr).
\]
\end{lemma}
\begin{proof} We write $r_t = e^{\pi_\X^*\varphi_t}r_0$
and $\tilde r_t = e^{\pi_K^*\psi_t} \tilde r_0$ with $\varphi_t$ a smooth
curve in $\Sp(\X,J^\X,\cf_{r_0})^\T$ and $\psi_t = \hTheta(\varphi_t)$ a curve
in $\Sp(\K,J^\K, \cf_{\tilde r_0})^\T$. Let $\Pi_\X=(r_0,\pi_\X) \colon \Cm \cong
\R^+ \times \Sm^\X$ be the product structure on $\Cm$ induced by $\X$ and
$r_0$ and similarly $\Pi_\K=(r_0,\pi_\K)\colon \Cm \cong \R^+ \times
\Sm^\K$. Since $N_{\tilde r_t}=N_{r_t}$, we have
\[
\bigl(\Pi_\K \circ \Pi^{-1}_\X\bigl)(e^{-\varphi_t(x)}, x) 
= (e^{-\psi_t(\Psi_t(x))}, \Psi_t(x)),
\]
where $\Psi_t=\Psi^\K_{r_t}\circ (\Psi^\X_{r_t})^{-1}\colon \Sm^\X\to \Sm^\K$.
Since $\Pi_\K \circ \Pi^{-1}_\X$ is a diffeomorphism, $\Psi_t$ is a smooth
curve of diffeomorphisms, hence $\psi_t$ and $\tilde r_t$ are also smooth
curves.

To compute the derivative, by translation of the parameter $t$, it suffices to
relate the functions $\Psi_\X^{\,*}\dot r_0$ and $\Psi_\K^{\,*}\dot {\tilde r}_0$ in
$C^\infty_{\Sm_0}(\R)^\T$ where $\Sm_0:=\Sm_{r_0}=\Sm_{\tilde r_0}$,
$\Psi_\X=\Psi^\X_{r_0}$ and $\Psi_\K=\Psi^\K_{\tilde r_0}$. For any $p\in
\Sm_0$, we let
\[
\Phi_t(p) := \flow_{\psi_t(\pi_\K(p))}^{-J \K}(p)
\]
where $\flow^{-J \K}_s$ stands for the flow of $-J \K$ on $\Cm$. Now
$\Phi_t(p)\in \Sm_{r_t}=\Sm_{\tilde r_t}$ and so, for any $p\in \Sm_0$, we have
\begin{equation}\label{compose}
\flow^{J \X}_{\varphi_t(\pi_\X(\Phi_t(p)))}\bigl(\Phi_t(p)\bigr) \in \Sm_0,
\end{equation}
where $\flow_s^{J \X}$ is the flow of $J \X$ on $\Cm$. Taking derivative at
$t=0$ in \eqref{compose}, we thus obtain an element in $T_p\Sm_0$. To compute
this derivative, we use that
\begin{gather*}
\frac\d{\d t} \Phi_t(p)\Big|_{t=0} = - \dot\psi_0(\pi_\K(p))J \K_p, \\
\varphi_0(\pi_\X(q))= 0, \, \,  \,  \flow_0^{J \X} = \Id_{T\Cm}
\end{gather*}
and get
\[
\frac\d{\d t}
\Big(\flow^{J \X}_{\varphi_t(\pi_\X(\Phi_t(p)))}\big(\Phi_t(p)\big)\Big)\Big|_{t=0}
= \dot{\varphi}_0(\pi_\X(p)) J \X_p - \dot{\psi}_0(\pi_\K(p))J \K_p.
\] 
As $(\Ds_0)_p = T_p\Cm \cap JT_p\Cm$, and $\K_p$ and $\X_p$ are in $T_p\Sm_0$,
we conclude that
\[
\dot \varphi_0(\pi_\X(p))  \X_p - \dot{\psi}_0(\pi_\K(p))  \K_p \in (\Ds_0)_p,
\]
which, after acting with $\cf_{r_0}$, gives
\[
\dot \psi_0\circ\pi_\K = \frac{\dot \varphi_0\circ\pi_\X}{\cf_{r_0}( \K)}
\]
on $\Sm_0$. This means equivalently that $\Psi_\K^{\,*}\dot {\tilde
  r}_0=\Psi_\X^{\,*}(\dot r_0/\cf_{r_0}(\K))$ as required.

The formula for $\cf_{\tilde r_t}$ follows because by
Lemma~\ref{l:cone-potentials}, $\Psi^\X_{r_t}\circ (\Psi^\K_{r_t})^{-1}$ is a
CR isomorphism between the CR structures induced by $\cf_{\tilde r_t}$ and
$\cf_{r_t}$ on $\Sm^\K$ and $\Sm^\X$ respectively
\end{proof}
It immediately follows that $\hTheta$ is smooth, and its derivative is
given as follows.
\begin{cor}\label{dhTheta} Let $\varphi\in \Sp( \X,J^\X, \cf_0^\X)^\T$
and $\psi:= \hTheta(\varphi)\in \Sp( \K,J^\K, \cf_0^\K)^\T$ induce CR
structures $(\Ds_\varphi,J_\varphi)$ and $(\Ds_\psi,J_\psi)$ on $\Sm^\X$ and
$\Sm^\K$ with contact forms $\cf_{\varphi}^\X = \cf_0^\X + \d^c\varphi$ and
$\cf_{\psi}^\K = \cf_0^\K + \d^c\psi$ respectively. Then there is a CR
isomorphism $\Psi_\varphi\colon (\Sm^\X,\Ds_\varphi,J_\varphi) \to
(\Sm^\K,\Ds_\psi,J_\psi)$ such that for any $\dot\varphi\in
C^\infty_{\Sm^\X}(\R)^\T$,
\[
\dot \psi:=\d \hTheta_\varphi(\dot\varphi)
= \Psi_\varphi^* \Bigl(\frac{\dot \varphi}{\cf_{\varphi}^\X(\K)}\Bigr)
\quad\text{and}\quad
\cf_{\psi}^\K=\Psi_\varphi^*\Bigl(\frac{\cf_{\varphi}^\X}{\cf_{\varphi}^\X( \K)}\Bigr).
\]
\end{cor}

\section{Extremal Sasaki metrics and weighted extremal K\"ahler metrics}
\label{s:global}

\subsection{Extremal Sasaki metrics}

Let $(\Sm, \Ds, J, \X)$ be a Sasaki manifold. The Levi-Civita connection of
(local) Sasaki--Reeb quotients by $\X$ pulls back to a connection $\nabla^\X$
on $\Ds$, preserving the transversal K\"ahler structure $(g_\X, J,\omega_\X)$,
which turns out to be (see e.g.~\cite[\S4]{david:weyltanaka}) the
\emph{Tanaka--Webster connection}~\cite{Webster} of $(\Ds, J,\X)$.  Thus the
scalar curvature of Sasaki--Reeb quotients pulls back to the
\emph{Tanaka--Webster scalar curvature} $\Scal(g_\X)$ of $\nabla^\X$.
  
\begin{defn}\cite{BGS} \label{d:extremal-sasaki} $(\Sm,\Ds,J,\X)$ is a \emph{CSC}
or {\it extremal} Sasaki structure if $\Scal(g_\X)$ is constant or is a
transversal Killing potential, respectively. Equivalently any Sasaki--Reeb
quotient with respect to $\X$ is a CSC or extremal K\"ahler metric,
respectively.
\end{defn}

The K\"ahler geometry on Sasaki--Reeb quotients with respect to $\X\in
\crJ_+(\Sm, \Ds, J)$, induced by an extremal Sasaki structure $(\Sm,\Ds,
J,\K)$ with respect to a possibly \emph{different} Sasaki--Reeb vector field
$\K \in \crJ_+(\Sm, \Ds, J)^\X$, was studied in \cite{AC}.

\begin{lemma}\label{l:weighted-extremal}\cite{AC}  Let $(M, J, g, \omega)$ be the K\"ahler orbifold corresponding to the Sasaki--Reeb quotient of $(\Sm, \Ds,J)$ with respect to a quasiregular $\X \in \crJ_+(\Sm, \Ds, J)$, and $f=\cf_\Ds^\X(K)>0$ the induced positive Killing potential on $(M, J, \omega)$ by $\K \in
\crJ_+(\Sm, \Ds, J)^\X$ via Lemma~\textup{\ref{l:potential}}. Then $(\Ds, J, \K)$ is an extremal Sasaki structure on $\Sm$ if and only if the smooth function 
\begin{equation}\label{(f,wt)-scalar-curvature}
\Scal_{f}(g) := f^2\Scal(g) - 2(m+1) f\Delta_g f - (m+2)(m+1))|\d f|^2_g
\end{equation}
is a Killing potential on $(M,J, \omega)$.
\end{lemma}

When $\K=\X$ (i.e.~$f=1$) we recover Definition~\ref{d:extremal-sasaki}. More
generally, Lemma~\ref{l:weighted-extremal} allows us to reduce the study of
extremal Sasaki metrics (with respect to a possibly irregular Sasaki--Reeb
field $\K$) in terms of a fixed K\"ahler manifold or orbifold $M$, obtained
from a fixed regular or quasiregular Sasaki--Reeb field $\X \in \crJ_+(\Sm,
\Ds, J)^\K$.  This is the point of view we take in this article.

\begin{defn}\label{d:(f,wt)-extremal} Let $(g,\omega)$ be a K\"ahler metric
on $(M,J)$ and $f$ be a Killing potential for $g$.  We say that $g$ is
\emph{$f$-extremal} if its \emph{$f$-scalar curvature} $\Scal_{f}(g)$, given by
\eqref{(f,wt)-scalar-curvature}, is also a Killing potential.
\end{defn}

Lemma~\ref{l:weighted-extremal} motivates the following definition.

\begin{defn}\label{d:Sasaki-(xi,wt)-extremal}\cite{AC} Let $(\Sm, \Ds, J, \X)$
be a Sasaki manifold, $\K \in \crJ_+( \Sm,\Ds, J)^\X$ and $f=\cf_\Ds^\X(K)>0$.
The \emph{$\K$-scalar curvature $\Scal_\K(g_\X)$ of $(\Sm, \Ds, J, \X)$} is
the function induced on $\Sm$ by \eqref{(f,wt)-scalar-curvature} on any
Sasaki--Reeb quotient by $\X$.  We say $\X$ is \emph{$\K$-extremal} if
$\Scal_\K(g_\X)$ is a transversal Killing potential of $(g_\X, J, \omega_\X)$.
\end{defn}
Thus Lemma~\ref{l:weighted-extremal} asserts that for $\K \in \crJ_{+}(\Sm,
\Ds, J)^\X$, $\X$ is $\K$-extremal if and only if $(\Ds, J, \K)$ is extremal
as in Definition~\ref{d:extremal-sasaki}. In particular, if $\K = \lambda \X$
with $\lambda>0$ constant, $\K$-extremality reduces to extremality.

\subsection{The Calabi problem for \texorpdfstring{$f$}{f}-extremal K\"ahler metrics}
\label{ss:calabi-problem-kahler}

As observed in \cite{AM, FO,lahdili1}, many features of the theory of extremal
K\"ahler metrics extend naturally to the $f$-extremal case. In particular, one
can formulate a weighted version of the Calabi problem~\cite{calabi} which
seeks an $f$-extremal K\"ahler metric $(g,\omega)$ with $\omega\in\kcl$, a
K\"ahler class.  In this setting, we pin down the function $f$ indirectly by
fixing first a quasiperiodic holomorphic vector field $\check\K$ in the Lie
algebra $\check\tor$ of a (compact) torus $\check\T \leq \Aut_r(M, J)$ inside
the reduced group of automorphisms of $(M, J)$ (see
e.g.~\cite{gauduchon-book}), and secondly, a real constant $\kappa>0$ such
that for any $\check\T$-invariant K\"ahler metric $(g, \omega)$ with $\omega
\in \kcl$, the Killing potential $f$ of $\check\K$ with respect to $g$,
normalized by $\int_M f \omega^m/m! = \kappa$, is positive on $M$, see
\cite[Lemma 1]{AM}.

\begin{problem}\label{Calabi-Kahler-problem}  Is there a $\check\T$-invariant
K\"ahler metric $(g, \omega)$ on $(M, J)$ with $\omega \in \kcl$, which is
$f$-extremal, where $f$ is the K\"ahler potential of $\check\K$ with respect
to $g$ determined by $\int_M f \omega^m/m! = \kappa$? We refer to such metrics
as \emph{$(\check\K, \kappa)$-extremal}.
\end{problem}

As observed in \cite{lahdili2}, if $(g_0, \omega_0)$ and $(g, \omega)$ are two
$\check\T$-invariant K\"ahler metrics in $\kcl$ with K\"ahler forms $\omega =
\omega_0 +\d\d^c \varphi$ for a $\check\T$-invariant smooth function $\varphi$ on
$M$, then the corresponding $\kappa$-normalized Killing potentials $f$ and $f_0$ 
of $\check\K$ are related by
\begin{equation}\label{potential}
f = f_0  + \d^c\varphi(\check\K).
\end{equation}

Another useful observation from \cite{FO, lahdili1, lahdili2} is that when $M$
is compact, any $(\check\K, \kappa)$-extremal metric on $(M, J)$ is invariant
under a \emph{maximal} torus in $\Aut_r(M, J)$, so we can assume without loss
of generality that $\check\T$ itself is a maximal torus inside $\Aut_r(M, J)$.

\subsection{The Calabi problem for \texorpdfstring{$\K$}{K}-extremal Sasaki
metrics} \label{ss:calabi-problem-sasaki}

In the case when $(M,J, \omega_0)$ is obtained as a Sasaki--Reeb quotient of a
compact regular $\T$-invariant Sasaki manifold $(\Sm,\Ds_0,J,\X)$ with
Sasaki--Reeb vector field $\X \in \tor_+$ and corresponding contact $1$-form
$\cf_0:=\cf^\X_{\Ds_0}$, any other vector field $\K \in \tor_+$ pins down a positive
Killing potential $f_0=\cf_0(\K)$ for the induced Killing vector field
\[
\check{K}:= K+\langle\X\rangle\in \mbox{Lie}(\check\T)=\tor/\spn{\X}
\]
on $(M,J, \omega_0)$, where $\check\T= \T /\Sph^1_\X$ is the induced isometric
torus action on $M$ with $\Sph^1_\X\leq \T$ being the $\Sph^1$-action induced
by $\X$.  By Lemma~\ref{l:potentials}, any other compatible Sasaki structure
$\cf=\cf_0+\alpha+\d^c\varphi \in \cS(\X,J^\X)^\T$ for $\alpha\in
\Omega^1_{\X,\cl}(\Sm)^\T$ and $\varphi \in
\Sp(\X,J^\X,\cf_0^\X)^\T$---see~\eqref{e:contact-param}---gives rise to a
Killing potential $f:=\cf^\X(\K)$ of $\check\K$ with respect to the induced
K\"ahler structure $\omega_\varphi= \omega_0 + \d\d^c \varphi$ on $(M, J)$,
and $f$ and $f_0$ are related by \eqref{e:potentials}, i.e., satisfy
\eqref{potential}.  It thus follows that the Killing potentials $f$ and $f_0$
are $\kappa$-normalized with
\[
\kappa:=\frac{1}{(2\pi) m!} \int_\Sm \cf^\X(\K)\cf^\X\wedge (\d\cf^\X)^m.
\]
In particular, $\kappa$ does not depend on $\cf^\X \in \cS(\X,J^\X)^\T$, a
fact that can be also deduced from the $\T$-equivariant Gray--Moser argument,
see Remark~\ref{r:equiGM}. This is consistent with
Lemma~\ref{l:momentum-cone-coincide} in which $\K$ is seen as a positive
affine-linear function on the natural moment polytope $P_\X \sub \check\tor^*$
associated to $(\Sm,\Ds_0,J,\X)$.

Equivalently, we can think in terms of the induced polarization $L$ of $(M,
J)$, as in Example~\ref{e:regular-cone}.  In this situation, the action of
$\T$ on the cone $(\Cm, J)$ corresponds to $\T \leq \Aut(M, L)$, and
$\check\T= \T/\Sph^1$ is the induced action on $(M, J)$; furthermore, for a
fixed $\K \in \tor$, and any $\T$-invariant hermitian product $h$ on $L$ whose
curvature form is a $\check\T$-invariant K\"ahler metric $\omega \in \kcl=
2\pi c_1(L)$, there is a naturally defined Killing potential $f$ of
$\check{K}$. Furthermore, for any other $\T$-invariant hermitian product
$e^\varphi h$ on $L$, the corresponding Killing potential for $\check\K$ with
respect to $\omega_\varphi \in \kcl$ is given by \eqref{potential}.

\begin{defn}\label{d:K-extremal} Since  $K\in \tor$ determines $\check\K$
and $\kappa$ in both the regular Sasaki and polarized cases, we shall also
refer to a $(\check\K, \kappa)$-extremal K\"ahler metric $(g_\varphi,
\omega_\varphi)$ on $(M, J)$ in the K\"ahler class $\kcl =2\pi c_1(L)$ as
\emph{$\K$-extremal}.
\end{defn}

Thus $\K$-extremal metrics on $(M,J,\kcl)$ correspond to $\K$-extremal Sasaki
structures on $(\Sm,\X,J^\X)$. We can sharpen this observation as follows.

\begin{prop}\label{p:reduction} Let $(\Sm, \Ds_0, J)$ be a compact CR manifold
of Sasaki type and $\X, \K \in \crJ_+(\Sm,\Ds_0,J)$ with $[\X, \K]=0$. Suppose
$\X$ is quasiregular with Sasaki--Reeb quotient the compact K\"ahler orbifold
$(M, J, \omega)$. Then there is an extremal Sasaki structure in $\cS(\K,J^\K)$
if and only if there is a $K$-extremal K\"ahler metric in the K\"ahler class
$[\omega]$.
\end{prop}
\begin{proof} Let $\T \leq \Aut(\Sm, \Ds_0, J) \leq \Aut(\Sm, \X, J^\X)$ be
the (compact) torus generated by $\X$ and $\K$. Then there is $\K$-extremal
K\"ahler metric in the K\"ahler class $[\omega]$ if and only if there is a
$\K$-extremal Sasaki structure $\cf^\X$ in $\cS(\X,J^\X)^\T$. Since
$\K$-extremality of $\cf^\X$ depends only on the transversal geometry, we may
translate $\cf^\X$ by any closed basic $1$-form $\alpha$, and hence this holds
if and only if there is a $\K$-extremal Sasaki structure in
$\Sc(\X,J^\X,\cf_{\Ds_0}^\X)^\T$. By Lemma~\ref{l:weighted-extremal} and the
bijection $\hTheta$ of~\eqref{e:bij-pot}, this holds if and only if there is
an extremal Sasaki structure in $\Sc(\K,J^\K,\cf_{\Ds_0}^\K)^\T$.

Hence it only remains to show that if there is an extremal Sasaki structure
$\cf$ in $\cS(\K,J^\K)$, then there is also one in
$\Sc(\K,J^\K,\cf_{\Ds_0}^\K)^\T$. Since (again) the extremality condition
depends only on the transversal geometry, we may assume $\cf\in
\Sc(\K,J^\K,\cf_{\Ds_0}^\K)$. Now by~\cite[Theorem 4.8]{BGS} $\cf$ is
invariant under a maximal torus $\T_{\max}\leq \Aut(\Sm,\Ds,J)^\K\leq
\Aut(\Cm,J)^\K$ (where $\Ds=\ker\cf$). Note that $\T_{\max}$ is also maximal
in $\Aut(\Cm,J)^\K$: any torus containing $\T_{\max}$ defines an abelian
subalgebra of $\aut(\Cm, J)^\K$, which descends to an abelian subalgebra of
$\aut(\Cm, J)^\K/\mathopen<\K, J\K\mathclose> \leq \thol(\Sm, \K, J^\K)$, and
this cannot be larger than the Lie algebra of $\T_{\max}$ by~\cite{BGS}.
Since $\thol(\Sm, \K, J^\K)$ is a finite dimensional Lie algebra, it follows
that $\Aut(\Cm,J)^\K$ is a finite dimensional complex Lie group.  Hence
$\T_{\max}$ is conjugate to a maximal torus containing $\T$, and apply the
conjugating automorphism to $\cf$, we obtain an extremal Sasaki structure in
$\Sc(\K,J^\K,\cf_{\Ds_0}^\K)^\T$.
\end{proof}

\section{Properness of the weighted Mabuchi energy}
 
In this section, we consider a compact Sasaki manifold $(\Sm, \Ds, J, \K)$,
with corresponding contact form $\cf_0^\K$, and fix a maximal compact torus
$\T \leq \Aut(\Sm, \Ds, J)$ such that $\K \in \tor$.  By
\cite[Theorem~4.8]{BGS} (see the proof of Proposition~\ref{p:reduction}), the
search for an extremal Sasaki structure in $\cS(\K, J^\K)$ can be reduced to
the search of such structures in the space $\Sc(\K,J^\K, \cf_0^\K)^\T$ of
$\cf_0^\K$-normalized, $\T$-invariant Sasaki structures, see
Definition~\ref{d:sasaki-slice}. We denote by $K_\ext \in \tor$ the
\emph{extremal vector field} associated to $(K, \T)$, see
Definition~\ref{d:contact-Futaki} and \cite{He-Li}, which is well-defined
irrespective of the existence of an extremal Sasaki structures in
$\Sc(\K,J^\K, \cf_0^\K)^\T$, but which will coincide with the vector field
defined by the transversal scalar curvature of any extremal Sasaki structure in
that space, should it exist.

We fix a quasiregular Sasaki--Reeb vector field $\X \in \tor_+$ in the Reeb
cone of $(\Sm,\K,J^\K,\T)$. We consider the induced action of $\T$ and of its
complexification $\G=\T_\C$ on the polarized complex cone $(\Cm, J, \K)$ (see
Lemma~\ref{l:complexified-action}) and identify the polarized complex cone of
$(\Sm, \Ds, J, \X)$ with $(\Cm, J, \X)$, see Section~\ref{s:abstractcone}.
The important feature of this setting is that $\G$ naturally acts on the
spaces $\cpot_\X(\Cm,J)^\T$ and $\cpot_\K(\Cm,J)^\T$ of $\T$-invariant cone
potentials of $(\Cm, J, \X)$ and $(\Cm, J, \K)$, respectively. Using the
basepoints $r^\X_0 \in \cpot_\X(\Cm,J)^\T$ and $r^\K_0 \in \cpot_\K(\Cm,J)^\T$
corresponding to the contact forms $\cf_0^\X, \cf_0^\K$, there is an induced
action of $\G$ on the corresponding spaces $\Sp( \X,J^\X, \cf_0^\X)^\T$ and
$\Sp( \K,J^\K, \cf_0^\K)^\T$ of $\T$-invariant Sasaki potentials.

\begin{defn}\label{d:relative-Sasaki-Mabuchi}  (see e.g.~\cite{He1, He-Li})
The (relative) Mabuchi energy is the functional
\[
\ME^{K}\colon \Sp(\K,J^\K, \cf_0^\K)^\T\to \R
\]
characterized by
\[
\begin{split}
(\d \ME^\K)_\psi(\dot \psi)
& = \int_\Sm \, \dot \psi \Big(\Scal(g_\psi) - \cf_\psi(\K_\ext)\Big) \,
\cf_\psi \wedge \d\cf_\psi^m, \\
\ME^\K(0) &=0,
\end{split}
\]
where for any $\psi \in \Sp(\K,J^\K, \cf_0^\K)^\T$, $\cf_\psi:= \cf_0^\K +
\d^c_\K \psi$ stands for the corresponding contact form in $\Sc(\K,J^\K,
\cf_0^\K)^\T$ whereas $\Scal(g_\psi)$ denotes the corresponding transversal
scalar curvature.
\end{defn}
  
\begin{defn}\label{d:K-twisted-Mabuchi} The $K$-twisted (relative) Mabuchi
energy is the functional
\[
\ME^\X_\K \colon \Sp(\X,J^\X, \cf_0^\X)^\T\to \R
\]
characterized by
\[
\begin{split}
(\d \ME^\X_\K)_{\varphi} (\dot \varphi)
&= \int_\Sm \, \dot \varphi \Big(\Scal_\K(g_{\varphi})-\cf_{\varphi}(\K_\ext)\Big) \, (\cf_{\varphi}(\K))^{-m-3} \cf_{\varphi} \wedge \d\cf_{\varphi}^m, \\
   \ME^\X_\K(0) &=0.
\end{split}
\]
\end{defn}
Both functionals are invariant under additive constants, i.e., they descend to
the spaces $\Sc(\K,J^\K, \cf_0^\K)^\T$ and $\Sc(\X,J^\X, \cf_0^\X)^\T$,
respectively.
 
\begin{rem}\label{r:Lahdili-Mabuchi} As $\X\in \tor_+$ is
quasiregular, $\ME^\X_\K$ gives rise to a functional $\ME_{\check K, \kappa}$
acting on the space of $\check\T$-invariant (where we recall $\check\T= \T
/\Sph^1_\X$) $\omega_0$-relative K\"ahler potentials on the Sasaki--Reeb
quotient $(M, J, \omega_0)$, defined by
\[
(\d \ME_{\check K, \kappa})_\varphi (\dot \varphi)
= (2\pi)  \int_M \dot \varphi \big(\Scal_{f_{\varphi}}(\omega_{\varphi})\big)^\perp f_{\varphi}^{-m-3} \omega_{\varphi}^{m}, \ \ \ME_{\check K, \kappa} (0)=0,
\]
where $\omega_{\varphi}= \omega_0 + \d \d^c \varphi$ is the K\"ahler metric
corresponding to $\varphi$, $f_{\varphi}$ is the $\kappa$-normalized Killing
potential of $\check{K}$ with respect to $\omega_{\varphi}$, and
$\big(\Scal_{f_{\varphi}}(\omega_{\varphi})\big)^\perp$ denotes the orthogonal
projection of the corresponding $f_{\varphi}$-scalar curvature to the
orthogonal complement of the space of $\omega_{\varphi}$-Killing potentials of
elements of ${\rm Lie}(\check\T)$, with respect to global $L^2$-product
$(\phi_1, \phi_2)_{f_{\varphi}} =\int_M \phi_1 \phi_2 f_{\varphi}^{-m-3}
\omega_{\varphi}^m$.  This is, up to the multiplicative factor $2\pi$, the
relative weighted Mabuchi energy introduced in \cite{lahdili}. Equivalently,
in terms of the momentum polytope $P_\X \sub \check{\tor}^*$ of $(M, \omega_0,
\check\T)$, $\ME^\X_\K$ is $2\pi$ times the Mabuchi functional $\ME_{v,w}$
defined in \cite{lahdili2}, where the weights $v(x)= (\ell_\K(x))^{-1-m}$ and
$w(x)= \ell_\ext(x) \ell_\K(x)^{-3-m}$ are determined by the affine-linear
functions $\ell_{K}(x)$ and $\ell_\ext(x)$ on $\check{\tor}^*$, corresponding
to $\K, \K_\ext \in \tor$, see Lemma~\ref{l:momentum-cone-coincide}.
\end{rem}
 
\begin{lemma}\label{M} Let $\hTheta\colon \Sp(\X,J^\X, \cf_0^\X)^\T \to
\Sp(\K,J^\K, \cf_0^\K)^\T$ be the bijection defined by~\eqref{e:bij-pot}\textup;
then $\ME^\K\circ\hTheta = \ME^\X_\K$.
\end{lemma}  
\begin{proof} We check  $\ME^\K\circ\hTheta$ satisfies the characterization
of $\ME^\X_\K$. Clearly $\ME^\K(\hTheta(0))=\ME^\K(0)=0$ and if $\varphi\in
\Sp(\X,J^\X, \cf_0^\X)^\T$ with $\psi=\hTheta(\varphi)$ then
$\d(\ME^\K\circ\hTheta)_\varphi=(\d\ME^\K)_\psi\circ (\d \hTheta)_\varphi$. We
now apply Corollary~\ref{dhTheta}. As $\Psi_\varphi$ is a CR diffeomorphism
between the CR structures induced by $\varphi$ and $\psi$, \eqref{scal-twist}
implies that $\Psi_\varphi^{\,*}
\Bigl(\frac{\Scal_\K (g_{\varphi})}{\cf_{\varphi}(\K)}\Big)= \Scal(g_\psi)$,
and the result follows.
\end{proof}

We now recall the definition of the $d_1$-distance  on the space $\Sp(\K,J^\K, \cf_0^\K)^\T$, originally due to Darvas~\cite{darvas} in the K\"ahler case, and extended by He--Li~\cite{He-Li} to the Sasaki case.
\begin{defn} For $\psi_0, \psi_1 \in  \Sp(\K,J^\K, \cf_0^\K)^\T$, we let
\[
d_1(\psi_0, \psi_1) := \inf_{{\psi_t}}\Big\{\int_{0}^1 \Big(\int_\Sm |\dot \psi_t| \cf_{\psi_t} \wedge d\cf_{\psi_t}^m\Big) \d t\Big\},
\]
where the infimum is taken over all smooth paths $\psi_t, \, t\in [0,1],$ connecting $\psi_0$ and $\psi_1$ inside  $\Sp(\K,J^\K, \cf_0^\K)^\T$.
\end{defn}
The fact that $d_1$ is a distance is not obvious,  and is  one of the main results of \cite{He-Li}.   We  define the $d_1$-distance on the space $\Sp(\X,J^\X, \cf_0^\X)^\T$ by the same formula. When $\X$ is regular, $d_1$ is,  up to a factor $2\pi$,  the distance  defined  in \cite{darvas} on the space of relative K\"ahler potentials on $(M,J)$. We now show that $\hTheta$ is
bilipschitz with respect to $d_1$.

\begin{lemma}\label{d1} There exists $\Lambda\in\R^+$, depending only
on $(\Sm, \Ds, J, \K, \X)$, such that
\[
\frac{1}{\Lambda} d_1(\varphi_0, \varphi_1) \le d_1\bigl(\hTheta(\varphi_0),  \hTheta(\varphi_1)\bigr) \le \Lambda d_1(\varphi_0, \varphi_1).
\]
\end{lemma}
\begin{proof} Let $\varphi_t$ be a smooth path connecting $\varphi_0$ and $\varphi_1$ in  $\Sp(\K,J^\K, \cf_0^\K)^\T$. By Corollary~\ref{dhTheta}, we have
\[
\begin{split}
\int_{0}^1 \Big(\int_\Sm |\dot \psi_t| \cf_{\psi_t} \wedge \d\cf_{\psi_t}^m\Big) \d t  &= \int_{0}^1 \Big(\int_\Sm\frac{ |\dot \varphi_t|}{(\cf_{\varphi_t}(\K))^{m+2}} \cf_{\varphi_t} \wedge \d\cf_{\varphi_t}^m\Big) \d t 
\end{split}
\]
The key point is that the positive functions $f_t=\cf_{\varphi_t}(\K)$ define $\kappa$-normalized K\"ahler potentials on $M$ with respect to the K\"ahler metrics $\omega_{\varphi_t}$,  i.e., they are related by the formula \eqref{potential}. It thus follows (see e.g.~\cite[Lemma~1]{lahdili2} for a general statement) that $f_t(\Sm)=[a,b]$ is a fixed positive interval (independent of $\varphi_0, \varphi_1$ and $\varphi_t$), i.e., we have the uniform estimates
\[
\frac{1}{b^{m+2}} d_1(\varphi_0, \varphi_1) \le d_1(\psi_0, \psi_1) \le \frac{1}{a^{m+2}} d_1(\varphi_0, \varphi_1). \qedhere
\]
\end{proof}
We next introduce a suitable notion of properness of $\ME^\K$ (resp. $\ME^{\X}_\K$),  analogous to the corresponding notions in K\"ahler geometry, see \cite{DaRu, tian0, ZZ}.  To this end, we consider the action of  $\G:=\T_\C$ on the spaces $\Sp(\X,J^\X, \cf_0^\X)^\T$ and $\Sp(\K,J^\K, \cf_0^\K)^\T$ as explained in the beginning of the section.  

\begin{lemma}\label{l:isometric} $\G$ acts by $d_1$ isometries  on $\Sp(\X,J^\X, \cf_0^\X)^\T$ and $\Sp(K,J^\K, \cf_0^\K)^\T$.
\end{lemma}
\begin{proof} For any smooth curve $r(t)$ of cone potentials on $(\Cm, J, \X)$,
and any $g\in \G$, let $(g^*r)(t)= r(t) \circ g$ be the corresponding $\G$
action on curves. We compute the $d_1$-lengths of $r(t)$ and $(g^*r)(t)$:
\begin{equation*}
\begin{split} 
  \int_{0}^1 \int_\Sm \Big|\frac{\d (g^*r)(t)}{\d t} \Big| \Big(\cf_{(g^*r)(t)} \wedge (\d\cf_{(g^*r)(t)})^m\Big) \d t  &= \int_{0}^{1} \int_\Sm |g^*\dot{r}(t)|  g^*\Big(\cf_{r(t)} \wedge (\d\cf_{r(t)})^m\Big)\d t\\
  &= \int_{0}^1 \int_\Sm |\dot{r}(t)| \cf_{r(t)} (\d\cf_{r(t)})^m \d t. \qedhere
\end{split}
\end{equation*}
\end{proof}
\begin{lemma}\label{l:equivariant}
$\hTheta\colon \Sp(\X,J^\X, \cf_0^\X)^\T \to \Sp(\K,J^\K, \cf_0^\K)^\T$ is
  $\G$-equivariant.
\end{lemma}
\begin{proof} Obvious from the definition of $\hTheta$.
\end{proof}
The following crucial result is a refinement (in the regular case) of the
uniqueness result of \cite{V1} (which in turn is a Sasaki version of
\cite{BB}).
\begin{lemma}\label{l:uniqueness} Suppose $(\Sm, \Ds, J, \T, \X, \K)$ is as
above with $\X$ regular. Then any two extremal Sasaki structures $\psi,  \tilde \psi \in \Sp(\K,J^\K, \cf_0^\K)^\T$  are isometric under the action of $\G$.
\end{lemma}
\begin{proof}  Using Proposition~\ref{bijection-potentials} and the fact
that $\hTheta$ is $\G$-equivariant (Lemma~\ref{l:equivariant}), it is enough
to establish the result for the corresponding $\K$-extremal Sasaki structures
in $\cS(\X,J^\X)^\T$. As $\X$ is regular, it is enough to show that the
induced $\check\T$-invariant $(\K, \kappa)$-extremal K\"ahler metrics $\omega$
and $\tilde \omega$ on $(M, J, \kcl, \check\T=\T/\Sph^1_\X)$ are isometric by
an element of $\check\T_\C$.  Indeed, as $\check\T$ is a maximal torus in the
reduced group of automorphisms $\Aut_r(M, J)$, and the induced extremal vector
field $\check\K_\ext:= J \grad_g \Scal_f(g)$ is central in the Lie algebra of
Killing fields of $(M, J, g, \omega)$, we must have $\check\K_\ext \in
\check{\tor}$; it then follows by \cite[Theorem B1]{lahdili2} (see also
\cite{FO,lahdili1}) that the group $\Aut_r^{\check\T}(M, J)$ of reduced
automorphisms of $(M, J)$, commuting with the $\check\T$-action, is a
reductive Lie group. As $\check\T$ is simultaneously central and maximal, it
follows that the identity component of $\Aut_r^{\check\T}(M, J)$ is abelian
and equal to $\check\T_\C$.

The result now follows from the uniqueness result
in~\cite{lahdili3}, established in the more general context of $(v,w)$-extremal
K\"ahler metrics and building on an argument in~\cite{CPZ}, which implies that
any two $\check\T$-invariant $(\check\K, \kappa)$-extremal K\"ahler metrics in
$\kcl$ are isometric by an element in the connected component of the identity
of the group $\Aut_r^{\check\T}(M, J)$.
\end{proof}

With this understood, for any $\psi_0,\psi_1 \in \Sp(\K,J^\K, \cf_0^\K)^\T$, we let
\[
d_1^{\G}(\psi_0, \psi_1) := \inf_{g_0,g_1\in \G} \{ d_1(g_0 \cdot \psi_0, g_1\cdot \psi_1) \}= \inf_{g\in \G}\{d_1(\psi_0, g\cdot \psi_1) \}.
\]
and define similarly $d_1^{\G}(\varphi_0, \varphi_1)$ for $\varphi_0, \varphi_1 \in \Sp(\X,J^\X, \cf_0^\X)^\T$.

\begin{defn} $\ME^\K$ (resp. $\ME^\X_\K$) is $\G$-proper if there exist constants $\Lambda>0$ and $C$,  such that  for any $\psi \in \Sp(\K,J^\K, \cf_0^\K)^\T$ we have  $\ME^\K (\psi) \ge \Lambda d_1^{\G}(0, \psi) - C$  (resp. for any $\varphi \in \Sp(\X,J^\X, \cf_0^\X)^\T$ we have $\ME^\X_\K(\varphi) \ge \Lambda d_1^{\G}(0, \varphi) - C$).
\end{defn}
As $\hTheta(0)=0$ by the definition of $\hTheta$, Lemmas~\ref{M} and \ref{d1} yield
\begin{cor}\label{central}  $\ME^\K$ is $\G$-proper if and only if $\ME^\X_\K$ is $\G$-proper.
\end{cor}

The main result in this section extends one implication
in \cite[Theorem 1]{He1} and \cite[Theorem 5.1]{He-Li}.

\begin{Theorem}\label{main-result} Let $(\Sm, \Ds, J)$ be a compact CR manifold
of Sasaki type, $\T \leq \Aut(\Sm, \Ds, J)$ a maximal torus, and $\X, \K \in
\tor_+$ with $\X$ regular. If there exists an extremal Sasaki structure
$\cf_\psi$ associated to $\psi\in \Sp(\K,J^\K, \cf_0^\K)^\T$, then $\ME^\K$ is
$\G$-proper.
\end{Theorem}
\begin{proof}  One can directly check that in the setting above, all the
conditions of the general existence/properness principle of \cite[Section 3]{DaRu}
hold true with
\[
(\mathcal R, d, F, G)= \Big(\Sp(\X,J^\X, \cf_0^\X)^\T, d_1, \ME^K, \G\Big).
\]
Indeed, for showing the conditions (A1)--(A4) of \cite{DaRu},
we notice that the $d_1$ completion of $\Sp(\K,J^\K, \cf_0^\K)^\T$ and the
extension of $\ME^\K$ is defined and proved in \cite{He-Li} whereas the fact
that the $\G$-action is isometric is established in Lemma~\ref{l:isometric}
above. For proving the conditions (P1)--(P7) of \cite{DaRu}, we notice that: (P1)
follows from \cite[Theorem 6.7]{He-Li} and the fact that the (relative)
Mabuchi energy $\ME^\K$ differs from the functional $\mathcal K$ in
\cite{He-Li} by an affine-linear function on bounded geodesics,
see~\cite[Prop.~10]{berndtsson} or \cite{lahdili3} for a direct argument
with respect to $\ME^\X_\K$ on the regular quotient; (P2) follows from
\cite[Theorem 6.6]{He-Li}; (P3) is the regularity result \cite[Theorem
  5.4]{He1} which can be extended from the CSC to the extremal case by the
argument in \cite{He-ext}. (P4) and (P5) are established in Lemmas
\ref{l:isometric} and \ref{l:uniqueness} above. (P6) follows from the general
principle in \cite[Section 6]{DR} (see in particular \cite[Prop.~6.8]{DR})
whereas (P7) is automatically satisfied as it can be checked easily, see also
\cite[Prop.~5.1]{He-Li}.
\end{proof}
\begin{rem}\label{r:properness}
An alternative way to establish Theorem~\ref{main-result} is to compare
directly the notion of $\G$-properness of $\ME^\K$ with the properness notion
used in \cite[Theorem~1]{He1} and \cite[Theorem~5.1]{He-Li}.  W. He considers
in \cite{He1} the identity component $G := \Aut_0(\Sm, \K, J^\K)$ of the group
of diffeomorphisms of $\Sm$ preserving $\K$ and the transversal holomorphic
structure $J^\K$, which naturally acts on the space $\cS(\K, J^\K)$ of
compatible contact forms by pullback. There is an induced action of $G$ on the
quotient space $\cS(\K,J^\K)/\Omega^1_{X,\cl}(\Sm)$ and hence, by using the
base point $\cf_0^\K$, on the marking $\Sc(\K, J^\K, \cf_0^\K)$. As $G$ does
not act on the space of $\cf_0^\K$-relative Sasaki potentials $\Sp(\K,J^\K,
\cf_0^\K)$, a slice $\Sp_0(\K,J^\K, \cf_0^\K) := ({\bf I}^\K)^{-1}(0)$ is
introduced, where ${\bf I}^\K : \Sp(\K,J^\K, \cf_0^\K) \to \R$ is a Sasaki
version of the Aubin--Mabuchi functional defined by
\[ 
\begin{split}
(\d{\bf I}^\K)_\psi(\dot \psi)
&= \int_\Sm \dot \psi\,\cf_\psi \wedge \d\cf_\psi^m, \\
   {\bf I}^\K(0) &=0.
\end{split}
\]
Unlike $\ME^\K$, ${\bf I}^\K$ is not invariant but is equivariant under
additive real constants and thus determines a unique representative $\psi_0
\in \Sp_0(\K,J^\K, \cf_0^\K)$ of $\cf_\psi \in \Sc(\K,J^\K, \cf_0^\K) $. This
leads to an action of $G$ on the slice $\Sp_0(\K,J^\K, \cf_0^\K)$, which we
denote by $[g] \cdot \psi_0$, where $[g]\in [G]$, the effective quotient of
$G$.  One direction of the statements in \cite[Theorem~1]{He1} and
\cite[Theorem~5.1]{He-Li} then yields that if there exists an extremal Sasaki
metric in $\cS(\K,J^\K)^\T$, the Mabuchi energy $\ME^\K$ is proper on
$\Sp(\K, J^\K, \cf_0^\K)^\T$ with respect to $[G]$, i.e., satisfies
\begin{equation}\label{improper-proper}
\ME^\K(\psi) \ge \Lambda d_1^{[G]}(0, \psi_0) - C, 
\end{equation}
where $d_1^{[G]}(0, \psi_0):= \inf_{g\in G} \{d_1(0, [g] \cdot \psi_0) \}$. 
 
Now $\G$ also acts on $\cS(\K,J^\K)^\T$, hence on $\Sp_0(\K,J^\K, \cf_0^\K)$, and
we let $[\G]$ denote the effective quotient. Although $\G$ is not a subgroup
of $G$, $[\G]$ is a subgroup of $[G]$ and the $\G$-orbits of induced Sasaki
structures are inside the $[G]$-orbits on $\Sp_0(\K,J^\K, \cf_0^\K)$.  On the
other hand, the action of $\G$ on $\Sp(\K,J^\K, \cf_0^\K)^\T$ includes
translations by constants, so we conclude that for any $\psi \in \Sp(\K,J^\K,
\cf_0^\K)^\T$ with projection $\psi_0$ to $\Sp_0(\K,J^\K, \cf_0^\K)$ we have
\[
d_1^{[\G]}(0, \psi_0) := \inf_{g\in \G}\{d_1(0, [g]\cdot \psi_0)\}  \ge d_1^{[G]}(0, \psi_0), \quad d_1^{\G}(0, \psi)= d_1^{\G}(0, \psi_0)  \le d_1^{[\G]}(0, \psi_0).
\]
It thus follows that the conclusion of Theorem~\ref{main-result} will hold
true if one can improve \eqref{improper-proper} by replacing $G$ with $[\G]$
for functions $\psi \in \Sp(\K,J^\K, \cf_0^\K)^\T$. An inspection of the
arguments in \cite{He1} (which in turn are an application of the general principle of \cite{DaRu}) reveals that the only missing ingredient to do so is
proving that the $[\G]$-action is transitive on the space of smooth extremal
Sasaki potentials $\psi_0 \in \Sp_0(\K,J^\K, \cf_0^\K)^\T$. This property is
established in Lemma~\ref{l:uniqueness} above.

Note that this argument yields the stronger notion of $d_1^{[\G]}$-properness
of $\ME^\K$, should an extremal Sasaki metric in $\cS(\K, J^\K)$
exist. However, as the map $\hTheta$ does not respect the corresponding slices
$\Sp_0(\K,J^\K, \cf_0^\K)^\T = ({\bf I}^\K)^{-1}(0)$ and $\Sp_0(\X,J^\X,
\cf_0^\X)^\T= ({\bf I}^\X)^{-1}(0)$, it is not immediately clear whether or
not the energy $\ME^\X_\K$ is also $d_1^{[\G]}$-proper. The somewhat weaker
but more natural notion of $d_1^{\G}$-properness we use above remedies this,
see Corollary~\ref{central}.
\end{rem}

\section{Weighted K-stability of regular Sasaki--Reeb quotients} In this
section we relate the existence of extremal Sasaki structures on a compact
Sasaki manifold of regular type with a special case of the weighted
K-stability of a polarized variety, defined in \cite{lahdili2}.  We thus
assume that $(\Sm, \Ds, J)$ is a compact CR manifold of Sasaki type, $\T\leq
\Aut(\Sm, \Ds, J)$ a torus, and $\X \in \tor_+$ a regular Sasaki--Reeb vector
field. The Sasaki--Reeb quotient $(M, J, L)$ is a smooth projective polarized
variety, as in Example~\ref{e:regular}. We denote by $\T$ the induced torus
action on $L$, and by $\check\T= \T/\Sph^1_\X$ the torus action on $(M,J)$.

\subsection{Polarized test configurations}
We first recall the notion of a (compactified) {polarized test configuration} associated to  $(M, J, L, \T)$,  going back to \cite{donaldson,Tian}.  
\begin{defn}\label{d:test-configuration} A $\T$-\emph{equivariant polarized test configuration of exponent} $s\in \N$ associated to  $(M, J, L, \T)$ is a normal polarized variety $(\tstM, \tstL)$ endowed an action $\T \leq \Aut(\tstM, \tstL)$, and
\begin{bulletlist}
\item a (flat) surjective morphism $\pi\colon \tstM \to \C P^1$ such that $\check\T$ preserves each fibre $M_{t}:= \pi^{-1}(t)$ and, for $t \neq 0$,   $(M_t,  L_t := {\tstL}_{|M_t})$ is $\T$-equivariantly isomorphic, as a polarized variety,  to $(M, L, \T)$;
\item a $\C^\times$-action $\rho$ on $(\tstM, \tstL)$ commuting with $\T$ and covering the usual $\C^\times$-action on $\C P^1$;
\item  an isomorphism  of polarized varieties
$$\lambda\colon \Big(M \times (\C P^1\setminus \{0\}), \, L^{s}\otimes \cO_{\C P^1}(s)\Big) \cong (\tstM \setminus M_0,  \, \tstL), $$
which is equivariant with respect to the actions of $\T\times\Sph_0^1$ on  $\big(M \times (\C P^1\setminus \{0\}), \, L^{s}\otimes \cO_{\C P^1}(s)\big)$  and  $\T\times\Sph_\rho^1$ on $\big(\tstM \setminus M_0,   \, \tstL\big)$, where $\Sph_0^1$ stands for the standard $\Sph^1$-action on $\cO(1) \to \C P^1$ and $\Sph^1_\rho$ is the $\Sph^1$-action induced by $\rho$.
\end{bulletlist}
The test configuration $(\tstM, \tstL)$ as above is  called \emph{product} if $\big(\tstM\setminus M_\infty, \pi\big)$ is  $\check\T$-equivariantly isomorphic (as morphism of complex varieties)  to $\big(M\times (\C P^1\setminus\{\infty\}), \pi_{\C P^1}\big)$.
\end{defn}
\begin{rem}\label{r:lifted action} Polarized test configurations are often
introduced (see e.g.~\cite{odaka,wang}) only by the corresponding requirements
for the induced torus actions on $\tstM$, $M$ and $\C P^1$, i.e., are
independent of the chosen lifts of these actions to $\tstL$ and $L$. In
particular, as $\cO(1)$ is trivial over $\C P^1\setminus\{0\}$, the third
condition reduces to the existence of a biholomorphism $\check{\lambda}\colon
M \times \big(\C P^1\setminus\{0\}\Big) \to \tstM \setminus M_0$ which is
equivariant with respect to the induced actions of $\check\T \times
\check{\Sph}_0^1$ on $M \times (\C P^1\setminus\{0\})$ and $\check\T\times
\Sph^1_{\check\rho}$ on $\tstM \setminus M_0$.  It is not hard to see that
given such a biholomorphism $\check{\lambda}$ (and assuming $L$ and $\tstL$
are very ample) one can always lift $\check\rho$, $\check\T$ and
$\check{\lambda}$ to obtain a polarized test configuration in the sense of
Definition~\ref{d:test-configuration}.
\end{rem}

We shall further assume that the $\T$-equivariant polarized test configuration
$(\tstM, \tstL,\pi, \rho)$ for $(M, J, L)$ as above has a smooth total space
$\tstM$, and take $s=1$ for simplicity. We consider the associated polarized
complex cones $(\tstC ,\X)$ and $(\Cm, \X)$, where $\tstC= (\tstL^*)^\times,
\Cm= (L^*)^\times$ and $\X$ denotes the regular Sasaki--Reeb vector field
corresponding to the fibrewise $\Sph^1$-actions on $\tstL$ and $L$ (see
Example~\ref{e:regular-cone}).  According to Lemma~\ref{l:key-identification},
the induced momentum polytopes for the $\check\T$-actions on $\tstM$ and $M$
are equal to the transversal polytope $P_\X$ defined
by~\eqref{e:transversal-polytope}.  By Lemma~\ref{l:momentum-cone-coincide},
$\K\in \tor_+$ is determined by an affine-linear function $\ell_K(x)$ on
$\check{\tor}^*$, which is positive on $P_\X$, whereas the extremal vector
field $\K_\ext$ gives rise to an affine-linear function $\ell_\ext(x)$ on
$\check{\tor}^*$.

For a $\check\T$-invariant K\"ahler metric $\tilde \Omega \in 2\pi c_1(\tstL)$,  we let $m_{\tilde \Omega}$ denote the $\check\T$-momentum map sending $\tstM$ to $P_\X$, and set
\begin{equation}\label{extremal-futaki}
\begin{split}
\Fut^\ext_\K(\tstM, \tstL) :=  & -\int_{\tstM} \Big(\Scal_{\tilde f, m+2}(\tilde \Omega) - \ell_\ext(m_{\tilde \Omega})\Big){\tilde f}^{-3-m}\frac{{\tilde \Omega}^{m+1}}{(m+1)!} \\
 &  +  2\int_{\tstM} {\tilde f}^{-1-m} (\pi^*\omega_{FS})\wedge\frac{{\tilde \Omega}^m}{m!},
 \end{split}
\end{equation}
where $\omega_{\rm FS}$ is a Fubini--Study metric with ${\rm Ric}(\omega_{\rm FS})=\omega_{\rm FS}$, $\tilde f= \ell_K(m_{\tilde \Omega})$ is a Killing potential of the Killing vector field induced by $\K$ on $\tstM$,  and we have set 
\[ \Scal_{f,\nu}(g) := f^2\Scal(g) - 2(\nu-1) f\Delta_g f - \nu(\nu-1)|\d f|^2_g, \] 
as in \cite{ACGL}.

Formula \eqref{extremal-futaki} is in fact the {\it weighted Futaki invariant}
of Lahdilli~\cite{lahdili2} with weights $v(x)= \ell_\K(x)^{-m-1}$ and
$w(x)=\ell_\ext(x)\ell_{K}^{-m-3}(x)$; it follows that it is independent of
the choice of $\tilde \Omega \in 2\pi c_1(\tstL)$.  We also notice that the
contact Futaki invariant associated to $(\Sm, \Ds, \K)$ (see
Definition~\ref{d:contact-Futaki}) vanishes if and only if $\K_\ext = c_\K
\K$, i.e., $\ell_\ext(x) = c_\K \ell_\K(x)$, in which case we have
$\Fut^\ext_\K(\tstM, \tstL) = \Fut_\K(\tstM, \tstL)$, where we have set
\begin{equation}\label{abdellah-futaki-0}
\begin{split}
\Fut_\K(\tstM, \tstL) :=  & -\int_{\tstM} \Big(\Scal_{\tilde f, m+2}(\tilde \Omega) - c_\K{\tilde f}\Big) \tilde f^{-3-m} \frac{{\tilde \Omega}^{m+1}}{(m+1)!}   \\ & + 2 \int_\tstM \tilde f^{-1-m} (\pi^*\omega_{FS})\wedge\frac{{\tilde \Omega}^m}{m!}.
\end{split}
\end{equation}

Proposition~\ref{p:reduction} and \cite[Theorem~2]{lahdili2} then yield the
following necessary condition for the existence of extremal Sasaki structures
in $\cS(\K, J^K)$, which we refer to as \emph{relative $\K$-weighted {\rm
    K}-semistability with respect to $\T$}.

\begin{Theorem}\label{Theorem:lahdili1} Let $(\Sm, \Ds, J, \X, \K)$ be as in Theorem~\ref{main-result}. If there exists an extremal Sasaki structure in $\cS(\K, J^\K)$, then we have $\Fut^\ext_{K}(\tstM, \tstL)  \ge 0$ for any $\T$-equivariant smooth polarized test configuration with reduced central fibre  $(\tstM, \tstL)$ associated to $(M, J, L)$.
\end{Theorem}

We can sharpen this to a positivity result as follows.

\begin{lemma}\label{l:key} Let  $(M, J, L)$ a smooth polarized complex manifold,  $\T$ be a maximal torus in $\Aut(M,L)$, and $\omega_0 \in 2\pi c_1(L)$ a $\check\T$-invariant K\"ahler metric,  where $\check\T\leq \Aut_r(M, J)$ is the induced torus action on $M$. Suppose that  $\ME_{\check K,\kappa}$ is $\G$-proper with respect to the  action of $\G= \T_\C$ on $\check\T$-invariant $\omega_0$-relative K\"ahler potentials, induced by its natural action on hermitian metrics on $L$.  Then,  for any $\T$-equivariant smooth polarized test configuration $(\tstM, \tstL)$ associated to $(M, J, L)$,  which has reduced central fibre $(M_0, L_0)$ and is not a product, we have
\[
\Fut_{K}^\ext(\tstM, \tstL)  > 0.
\]
\end{lemma} 
\begin{proof} The $\G$-properness implies that  $\ME_{\check\K,\kappa}$  is bounded from below on the space of $\check\T$-invariant $\omega_0$-relative K\"ahler potentials, so that by \cite[Thm.~7]{lahdili2}, we have  $$\Fut_{\check\K}^\ext(\tstM, \tstL) \ge 0.$$  We thus have to establish that  if $\Fut_\K^\ext(\tstM, \tstL) =0$  for a  smooth polarized  test configuration  $(\tstM, \tstL, \T)$ with reduced central fibre associated to $(M, J, L, \T)$, then $(\tstM, \tstL, \T)$ must be  a  $\C^\times$-equivariant  compactification at infinity of the product $\C \times M$.  To this end,  we use a slight modification of the  arguments from \cite{BDL,Dyrefelt}: suppose  that $\Fut_\K^\ext(\tstM, \tstL) =0$ and denote by $u_t$ the smooth ray of  $\check\T$-invariant $\omega_0$-relative K\"ahler potentials on $(M, J, \omega, \check\T)$  with $u_0=0$ (defined by the $\C^\times$-action on $\T$-invariant hermitian products on $\tstL \to \tstM$  on the test configuration $(\tstM, \tstL, \T)$). Using that
\[
	0=\Fut_{K}^\ext(\tstM, \tstL) = \lim_{t\to +\infty} \frac{\ME_{\check\K, \kappa}(u_t)}{t}
\]
(see again~\cite[Thm.~7]{lahdili2}) and the $\G$-properness of
$\ME_{\check\K, \kappa}$, we conclude that
\begin{equation}\label{bounded}
0 \le d_1^{\G}(0, u_t) \le o(t).
\end{equation}
To any test configuration $(\tstM, \tstL)$ as above,  and a $\check\T$-invariant $\omega_0$-relative K\"ahler potential $\varphi_0$, Phong--Sturm~\cite{PS} associate a $C^{1, \bar 1}$ geodesic ray $\varphi_t$ emanating from $\varphi_0$. A key estimate in \cite{BDL} reads
\[ |\varphi_t - u_t| \le C \]
which implies $d_1(u_t, \varphi_t) \le C$. As $\G$ acts by $d_1$-isometries (Lemma~\ref{l:isometric})  and using  the triangle inequality, we have
 \[d_1^{\G}(0, \varphi_t) -C \le d_1^{\G}(0, u_t). \]
and also 
 \begin{equation}\label{bounded-geodesic}
	0 \le d_1^{\G}(\varphi_0, \varphi_t) \le o(t).
	\end{equation}
 In our case,  we can take $\varphi_0=0$ (as we consider the space $\T$-invariant  hermitian metrics $h=e^{-2\varphi} h_0$ on $L$  with  $\T$  being a maximal torus in $\Aut(M, L)$). In \cite{DaRu} and \cite{BDL} the authors rather consider $\varphi_0$ to be the normalized potential  of a CSC metric in order to insure  that it is fixed under a maximal compact subgroup of $K \leq \Aut_r(M,J)$,  but we do not  need this assumption as we apply \cite[Prop.~6.8]{DaRu} with $K:=\T$ and $G:= \T_\C=\G$ acting on the space $\check\T$-invariant $\omega_0$-relative K\"ahler potentials. One property needed for this result that we did not find readily available  in the literature is that for a vector field $\Y\in \tor$ on $L$,  the flow $\flow_t^{-J\Y} \cdot \varphi$ generates a smooth geodesic in the space of $\check\T$-invariant potentials. This is classically known to be true for the induced action $[\flow_t^{-J\Y}]\cdot \varphi$ on normalized relative K\"ahler potentials with respect to the Aubin--Mabuchi functional ${\bf I}$, i.e.,  such that  ${\bf I}\Big([\flow_t^{-J\Y}]\cdot \varphi\Big)=0$ (see e.g.~\cite{gauduchon-book}), but the argument in \cite[Lemma~4.2]{BDL} actually shows that  $\flow_t^{-J\Y} \cdot \varphi$ differs from  $[\flow_t^{-J\Y}]\cdot \varphi$ by an affine linear function in $t$.
 
 Now, the arguments in \cite[Lemma~4.1]{BDL} and \cite[Prop.~4.10]{Dyrefelt} (we have a $o(t)$ bound for $d_1^{\G}(\varphi_0, \varphi_t)$ as in \cite{Dyrefelt}) extend tautologically in our situation,  and show that $\varphi_t$ is induced by the flow of a vector field $J\Y$ on $L$ with $\Y  \in \tor$. Finally, the fact that $(\tstM, \tstL)$ is a product test configuration follows from \cite[Prop.~4.3]{BDL} or \cite[Thm.~A.6]{Dyrefelt}.
\end{proof}

This lemma now leads to the following main result, which is an improvement of
Theorem~\ref{Theorem:lahdili1} to the following statement, which we call
\emph{relative $\K$-weighted {\rm K}-stability with respect to $\T$}.

\begin{Theorem}\label{Theorem:polystable}  Let $(\Sm, \Ds, J, \X)$ be a compact regular Sasaki manifold,  $\T$ a maximal torus in $\Aut(\Sm, \Ds, J)$, and  $\K \in \tor_+$ a Sasaki--Reeb vector field such that  there exists an extremal Sasaki structure in $\cS(\K, J^\K)$. Denote by  $(M, J, L, \T)$ the  smooth polarized variety obtained as Sasaki--Reeb quotient  with respect to $\X$,   with the induced  action of $\T$ on $L$.  Then,  for any $\T$-equivariant smooth polarized test configuration $(\tstM, \tstL)$ associated to $(M, J, L)$,  which has reduced central fibre $(M_0, L_0)$ and is not a product, we have $\Fut_{K}^\ext(\tstM, \tstL)  > 0.$
\end{Theorem} 
\begin{proof}   By Theorem~\ref{main-result}, the (modified) Mabuchi
energy $\ME^K$ is $\G$-proper and by Corollary~\ref{central}, the
corresponding $\K$-twisted modified Mabuchi energy $\ME^\X_\K$ is also
$\G$-proper.  Recall that $\G$-action was introduced by the natural action of
$\T_\C$ on the space of cone potentials $\cpot_\X(\Cm, J)$, Equivalently, in
terms of the smooth polarized variety $(M, J, L, \T)$, the group $\G=\T_\C$
acts naturally on the space of $\T$-invariant hermitian metrics $h$ on $L$
and, therefore, by considering the corresponding curvatures, on the space of
$\check\T$-invariant K\"ahler metrics $\omega \in 2\pi c_1(L)$. Furthermore,
by introducing a basepoint $h_0$ (with curvature $\omega_0$), $\G$ acts also
on the space of $\check\T$-invariant $\omega_0$-relative K\"ahler potentials
(which is naturally identified with the space $\Sp(\X, J^\X, \cf_0^\X)^\T$ of
$\T$-invariant Sasaki potentials).  As we recalled in
Remark~\ref{r:Lahdili-Mabuchi}, the $\K$-twisted Mabuchi energy $\ME^\X_\K$
defined on the space $\Sp(\X, J^\X, \cf_0^\X)^\T$ corresponds to the relative
$(\check{K}, \kappa)$-Mabuchi energy $\ME_{\check\K, \kappa}$ acting on the
space of $\check\T$-invariant $\omega_0$-relative K\"ahler potentials
$\varphi$ of K\"ahler metrics $\omega_{\varphi}=\omega_0 + \d \d^c \varphi \in
2\pi c_1(L)$.

By Theorem~\ref{main-result}, $\ME_{\check\K,\kappa}$ is $\G$-proper  in the sense that for any $\check\T$-invariant $\omega_0$-relative K\"ahler potential $\varphi$
\[ \ME_{K,\kappa}(\varphi) \ge \Lambda d_1^{\G}(0, \varphi)  - C,\]
where we recall $d_1^{\G}:=\inf_{g\in \G}d_1(0, g\cdot \varphi)$. The result now
follows from Lemma~\ref{l:key} below.
\end{proof}
\begin{rem}   One can extend the conclusions of Theorems~\ref{main-result} and \ref{Theorem:polystable} to smooth polarized test configurations  with reduced central fibre, which are $\T$-equivariant with respect to the torus $\T \leq \Aut(M, L)$  generated by the  Sasaki--Reeb vector fields $\X, \K$,  and the  CR vector field  $\K_\ext$  of Definition~\ref{d:contact-Futaki}.  To this end, one needs to consider instead of $\G = \T_\C$, the action of the centralizer $G_\T$ of $\T$ in $\Aut(M,L)$ (which also acts on the spaces $\cpot_\X(\Cm, J)$ and $\cpot_\K(\Cm, J)$), and use the uniqueness of $\T$-invariant $\K$-extremal metrics modulo $G_\T$ (which is  established \cite{lahdili3}) in order to obtain the  properness of $\ME_{\check{K}, \kappa}$ with respect to $G_\T$. \end{rem}

\subsection{Transcendental K\"ahler test configurations} A key feature of Lahdilli's construction~\cite{lahdili2} is that it allows one to extend the  conclusion of Theorem~\ref{Theorem:lahdili1} to a more general class of smooth test configurations, the {\it K\"ahler test configurations} introduced in \cite{DR,dyrefelt}, which are defined as follows
\begin{defn}\label{d:kahler-test-configuration} A  $\check\T$-equivariant K\"ahler test configuration associated to the K\"ahler manifold $(M,J)$ with  a fixed K\"ahler  class  $\kcl$  is a  smooth K\"ahler  manifold $\tstM$,  endowed a K\"ahler class $\tstA$ and a torus $\check\T \leq \Aut_r(\tstM)$ in its reduced group of automorphisms, and
\begin{bulletlist}
\item a surjective  holomorphic map $\pi\colon \tstM \to \C P^1$ such that $\check\T$ preserves each fibre $M_{t}:= \pi^{-1}(t)$ and, for $t \neq 0$,   $(M_t,  \kcl_t := {\tstA}_{|M_t})$ is $\check\T$-equivariantly isomorphic to $(M, \kcl, \check\T)$;
\item a $\C^\times$-action $\rho$ commuting with $\check\T$ and covering the usual $\C^\times$-action $\rho_0$ on $\C P^1$;
\item  a biholomorphism
$$\lambda \colon \Big(M \times (\C P^1\setminus \{0\}), \, \kcl \Big) \cong (\tstM \setminus M_0,  \, \tstA), $$
which is equivariant with respect to the actions of $\check\T\times \Sph_{\rho_0}^1$ on $M \times (\C P^1\setminus \{0\})$ and  $\check\T\times \Sph_\rho^1$ on $\tstM \setminus M_0$.
\end{bulletlist}
\end{defn}
Even in the case when $\kcl =2\pi c_1(L)$ for some  polarization $L$ of $(M,J)$,  and $\tstM$ is a smooth projective manifold, the K\"ahler class $\tstA$ on $\tstM$ need not be integral, thus allowing for a stronger K-stability condition. Indeed, following \cite{lahdili2}, if we have a smooth,  $\check\T$-equivariant test configuration $(\tstM, \tstA)$  associated to $(M, J, L)$ as above,  we can still use formula  \eqref{extremal-futaki}  in order to introduce  the numerical  invariant $\Fut_\K^\ext(\tstM, \tstA)$ (which again is independent of the choice of  a $\check\T$-invariant K\"ahler metric $\tilde \Omega \in \tstA$). Furthermore, according to \cite[Theorem 2]{lahdili2},  $\Fut^\ext_\K(\tstM, \tstA) \ge 0$  for any $\check\T$-equivariant smooth K\"ahler test configuration with reduced central fibre $(\tstM, \tstA)$ of $(M, J,2\pi c_1(L))$, should an extremal Sasaki structure exist in $\cS(\K, J^\K)$,  i.e.,  Theorem~\ref{Theorem:lahdili1} holds true even for smooth K\"ahler test configurations with reduced central fibre  associated to $(M, J, 2\pi c_1(L), \check\T)$.  

We expect that the following will hold true.
\begin{conjecture}\label{conjecture0} In the setting of Theorem~\ref{Theorem:polystable}, for any $\check\T$-equivariant smooth K\"ahler test configuration $(\tstM, \tstA)$ associated to $(M, J, 2\pi c_1(L))$, which has a reduced central fibre and is not a product, we have  $\Fut^\ext_\K(\tstM, \tstA)  > 0.$ 
\end{conjecture}

\begin{rem} In the light of recent results in the (unweighted) K\"ahler
setting~\cite{Dyrefelt}, and the discussion at the end of
Remark~\ref{r:properness}, the main missing point for establishing the above
conjecture is the improvement of the $\G$-properness of ${\bf M}^{X}_K$ to the
stronger notion of $[\G]$-properness.
\end{rem}

\subsection{Transcendental K\"ahler classes}\label{ss:(K,kappa)-stability}
We make here a few comments about the obstruction theory for the existence of
a $(\check K, \kappa)$-extremal K\"ahler metric on a general compact K\"ahler
manifold $(M, J, \kcl)$. When $\kcl = 2\pi(L)$ for an ample line bundle, and
taking $\check\T\leq \Aut_r(M, J)$ be a maximal torus with $\check\K \in
\check\tor$, this is covered by Theorem~\ref{Theorem:polystable} and
Conjecture~\ref{conjecture0} (by using
Proposition~\ref{p:reduction}). Following \cite{lahdili2}, one can extend the
weighted K-stability notion to the general K\"ahler case. To this end, one
fixes a momentum polytope $P \sub \check\tor^*$ for the $\check\T$-action on
$(M, J, \kcl)$, which in turn gives rise to a positive affine-linear function
$\ell(x)$ on $P$, such that the $\kappa$-normalized Killing potential $f$ of
$\check\K$ with respect to a $\check\T$-invariant K\"ahler metric $\omega \in
\kcl$ is written as $f= \ell(m_{\omega})$, where the momentum map $m_{\omega}$
of $(M, \omega, \check\T)$ is normalized by $m_{\omega}(M) = P$. Furthermore,
as shown in \cite{lahdili2}, the $L^2$ projection $\Scal_f(g)^\ext$ of the
$f$-scalar curvature $\Scal_{f}(g)$ of $\omega$ to the space of Killing
potentials with respect to $\omega$ of elements of $\check\tor^*$, by using
the weighted product $(\phi_1, \phi_2)_f = \int_M \phi_1 \phi_2 f^{-m-3}
\omega^m/m!$ can be written as $\Scal_f(g)^\ext=\ell_\ext(m_{\omega})$ for an
affine-linear function on $P$, independent of $\omega \in \kcl$.  For any
smooth $\check\T$-equivariant K\"ahler test configuration $(\tstM, \tstA)$
associated to $(M, \kcl, \check\T)$, and $\check\T$-invariant K\"ahler metric
$\tilde \Omega \in \tstA$, one can still define a numerical invariant via
formula \eqref{extremal-futaki}, where the momentum map $m_{\tilde \Omega}$ is
normalized by the condition $m_{\tilde \Omega}(\tstM)=P$, and $\tilde f =
\ell(m_{\tilde \Omega})$. This expression is independent of the
$\check\T$-invariant K\"ahler metric $\tilde \Omega \in \tstA$ according to
\cite{lahdili2}; furthermore, it is not hard to see that the numerical
invariant does not depend on the chosen momentum polytope $P$ for $(M, \kcl,
\check\T)$: it merely depends upon the data $(\check\K, \kappa, \check\T)$ so
we shall denote it by $\Fut^\ext_{(\check\K, \kappa)}(\tstM, \tstA)$. In this
more general K\"ahler setting, \cite[Theorem 2]{lahdili3} implies that if
$\kcl$ admits a $(\check\K, \kappa)$-extremal K\"ahler metric, then
$\Fut^\ext_{\check\K, \kappa}(\tstM, \tstA)\ge 0$ for any smooth
$\check\T$-equivariant test configuration with reduced central fibre
associated to $(M, J, \kcl, \check\T)$, and we can extend
Conjecture~\ref{conjecture0} to a relative $(\check\K, \kappa)$-weighted
K-stability statement.

\section{Sasaki K-stability versus weighted K-stability}\label{s:stability}

\subsection{Sasaki K-stability} \label{ss:CS} Given a complex polarized cone $(\Cm, J, \K, \T)$ as defined in Definition~\ref{d:cone-polarization}, it is known that (see e.g. the proof of \cite[Theorem 3.1]{V2}) there exists a $\T$-equivariant holomorphic embedding $\Cm \hookrightarrow \C^N$ sending $\Cm$ biholomorphically to $\Cv \backslash \{0\}$, where $\Cv \subset \C^{N}$ is a $\T_\C$-homogeneous affine variety, and $\T_\C \leq \mbox{GL}(\C^N)$ is a diagonal torus whose real form is  identified with $\T$. When $\K$ is quasiregular, one can further identify  the space $\Hol$ of holomorphic functions  on $\Cm$ introduced  in  \S \ref{ss:Reebcone-rev} with the space of regular function functions on $\Cv$ (see Remark~\ref{r:regular}), i.e.,
$\Hol = \C [x_1,\ldots, x_N]/I$ for a $\T_\C$-homogeneous ideal $I$. Thus, in
the decomposition \eqref{e:hol-alpha}, $\Gamma \sub \tor^*$ is the set of
weights of the $\T_\C$-action on the space of regular functions.  Given a
polarization $\K \in \tor_+$, the pair $(\Cv, \K)$ is referred to as a {\it
  polarized affine cone}.

\smallskip
In this latter setting, Collins and Sz\'ekelyhidi \cite{CSz} introduced the $\T$-equivariant index character
\[
F(\K, t):= \sum_{\alpha \in \Gamma} e^{-t\alpha(\K)} \textrm{dim} \, {\Hol}_{\alpha},
\quad {\rm Re}(t)>0,
\]
and prove (cf. \cite[Prop.~4.3]{CSz}) that $F(\K, t)$ is a well-defined function in $t$,  which has meromorphic extension to a neighbourhood of the origin, with Laurent expansion  
\[
F(\K, t) =  \frac{a_0(\K) m!}{t^{m+1}} + \frac{a_1(\K)(m-1)!}{t^m} + O(t^{1-m}).
\]
According to \cite[Theorem~4.10]{CSz}, the coefficients
$a_0(\K)>0$ and $a_1(\K)$ are smooth functions on $\tor_+$. With these
ingredients, the following definition is proposed in \cite{CSz}.
\begin{defn}\label{d:Donaldson-Futaki-Sasaki}  Given  a $\T_{\C}$-homogeneous affine variety $\Cv$ polarized by $\K \in \tor_+$,  and $\Y \in \tor$, the $\K$-Donaldson--Futaki invariant is defined to be the quantity
\[
\DF_{\Cv, \K} (\Y) := \frac{a_0(\K)}{m}D_{\Y}(a_1(\K)/a_0(\K)) + \frac{a_1(\K)D_{\Y}a_0(\K)}{m(m+1)a_0(\K)},
\]
where $D_\Y$ denotes the derivative in the direction of $\Y$ on $\tor$.
\end{defn}

If we assume, furthermore, that  $(\Cm, J)$ is smooth and $\K$ is primitive in the lattice of circle subgroups of $\T$, then  the quotient of $\Cm$ by the subgroup $\C_\K^\times$  generated by $K$ is a polarized orbifold $(M_\K,L_\K)$  whose generic isotropy groups are trivial,  and  for which the orbifold Riemann--Roch formula of \cite[(2.17)]{RT} gives

\begin{equation}\label{e:RR}
\dim_{\C} H^0(M_\K,L_\K^k) = a_0(\K)k^n+ (a_1(\K)+\rho(\K))k^{n-1}+ \cdots
\end{equation}
where  $\rho(K)$ is a periodic function in $k$ with average $0$,  and, up to multiplicative positive dimensional constants,  $a_0(\K)$ and $a_1(\K)$  are respectively the total volume and total scalar curvature of the polarized orbifold $(M_\K,L_\K)$. It follows that, up to a positive dimensional constant,  $\DF_{\Cv, \K} (\Y)$ coincides with the Futaki invariant of the real holomorphic vector field induced by $\Y$ on the polarized orbifold $(M_\K,L_\K)$, see \cite[Props.~4.3 \& 4.13 and Theorem 4.14]{CSz}.

\begin{rem} In the notation of Appendix~\ref{ss:GIT}, there is thus a
dimensional constant $c(m)>0$ such that
\[
\DF_{\Cv, \K}(\Y)= c(m) \Futc_\K(\Y),
\]
where $\Futc_\K(\Y)$ denotes the contact Futaki invariant $\Futc_\K(\Y)$
introduced in Definition~\ref{d:contact-Futaki}. Indeed by
Remark~\ref{r:contact-futaki}, $\DF_{\Cv, \K}(\Y)$ and $\Futc_\K(\Y)$ agree up
to a positive dimensional constant when $\K$ is a quasiregular.  However they
are both continuous and homogeneous of order $-m$ in $\K$ (where $m+1$ is the
complex dimension of $\Cv$) and the quasiregular directions are dense in
$\tor_+$.
\end{rem}

Following \cite{CSz}, Definition~\ref{d:Donaldson-Futaki-Sasaki} can be used to assign a numerical invariant  to any $\T$-{\it equivariant  test configuration} associated to the polarized affine cone $(\Cv, \K)$, as follows. Let
\[
w=(w_1, \ldots, w_N), \ w_i \in \intZ
\]
be a set of integers and denote by $\rho_{w}\colon \C^\times \to {GL}(\C^N)$ the weighted $\C^\times$-action defined by $w$. Taking the flat limit  across $0\in \C^\times$ of the $\C^\times$ orbit of $\Cv$ under the action of $\rho_w$, we obtain a flat family of  affine schemes over $\C$, with central fibre $C_0$.  Besides the $\T$-action on $\Cv_0$, the central fibre inherits the $\C^\times$-action $\rho_w$, and we denote by $\Y_{w}$ the  generator of the corresponding $\Sph^1$-action on $\C^k$ whose flow preserves $\Cv_0$.   It is  shown  in \cite[Section 5]{CSz} that  for a fixed $\K \in \tor_+$, Definition~\ref{d:Donaldson-Futaki-Sasaki}  can be applied on  the scheme $\Cv_0$ and gives a necessary condition for the existence of a CSC Sasaki structures (known as K-semistability). 

  \begin{Theorem}\label{Theorem:CSz}\cite{CSz} Let $(\Sm, \Ds, J)$ be a compact CR manifold of Sasaki type,  $\T$ a torus in $\Aut(\Sm, \Ds, J)$, and $\K \in \tor_+$  a Sasaki--Reeb vector field such that  $\Sm$ admits a CSC Sasaki structure in $\cS(\K, \th/J{\K})$. Then,  for any $\T$-equivariant test-configuration of the associated polarized affine cone $(\Cv, \K)$ with central fibre $\Cv_0$, $\DF_{\Cv_0, \K}(\Y_{w}) \ge 0$.
\end{Theorem}

\subsection{Test configurations of polarized complex cones} \label{ss:ConeTC} The Definition~\ref{d:test-configuration}  of  a  $\T$-equivariant polarized test configuration  of a  regular Sasaki--Reeb quotient $(M, J, L)$ provides  a natural link to the theory of test configurations of affine cones described above. Indeed,  given a $\T$-equivariant polarized test configuration $(\tstM, \tstL,\pi, \rho)$  for $(M, J, L)$  as in Definition~\ref{d:test-configuration}, and assuming that $s=1$ and $\tstM$ is smooth for simplicity, we consider  the corresponding complex cones $\tstcone:=(\tstL^*)^\times$ and $\Cm :=(L^*)^\times$ and denote by $\tilde \K$ and $\K$ the corresponding (regular) cone polarizations, see  Examples~\ref{e:regular} and \ref{e:regular-cone}.  We thus get the following objects: 
\begin{numlist}
\item a flat surjective morphism $\pi\colon\tstcone \to\C P^1$ such that $\T$ preserves each fibre $\Cm_{t}:= \pi^{-1}(t)$ and, for $t \neq 0$,  $(\Cm_{t}, \tilde{\K}_{|_{\Cm_t}})$ is $\T$-equivariantly isomorphic, as a polarized complex cone, to $(\Cm, \K)$;
\item a $\C^\times$-action $\rho$ on $\tstcone$ commuting with $\T$ and covering the usual $\C^\times$-action on $\C P^1$;
\item a biholomorphism 
\begin{equation}\label{e:equivbiholTC}
 \lambda\colon \Cm \times (\C P^1\setminus \{0\}) \cong \tstcone \setminus\Cm_0,
\end{equation}
which is equivariant with respect to the actions of $\T\times \Sph_0^1$ on $\Cm \times (\C P^1\setminus \{0\})$ and  $\T\times \Sph_{\rho}^1$ on $\tstcone \setminus\Cm_0$, where $\Sph_0^1$ stands for the standard $\Sph^1$-action on $\cO(1) \to\C P^1$.
\end{numlist}  

This prompts the following
\begin{defn}\label{d:test-configurationCONE} Let $(\Cm, J, \K)$ be a smooth polarized complex cone endowed with a holomorphic action of a compact torus $\T$,  and  $\K\in \tor_+$  a Sasaki--Reeb polarization.  A  \emph{smooth  $\T$-equivariant test configuration of $(\Cm,  J, \K)$} is a smooth polarized complex cone $(\tstcone, \tstJ, \tilde{\K})$, endowed with a $\T$-action such that $\tilde{\K} \in  \tilde \tor^+$, satisfying the conditions (i)--(iii) above. \end{defn}

Note that Definition~\ref{d:test-configurationCONE} does not require $\K$ to be regular nor the existence of a regular Sasaki--Reeb vector field $\X$ commuting with $\K$. To simplify the notation, we shall  at times  omit  the complex  structures $J$ and $\tstJ$.

\begin{lemma}\label{l:key-identification}  Let $(\tstcone, \tilde \K)$ be a $\T$-equivariant smooth test configuration of polarized cones associated to $(\Cm, \K)$.  Then,  the reduced Sasaki--Reeb cone $\tor_+ $ of $(\Cm, \K)$ coincides with the reduced Sasaki--Reeb cone $\tilde \tor_+$ of $(\tstcone, \tilde \K)$. Moreover,  $\tilde \K =  \K \in \tor_+$.
\end{lemma}
\begin{proof} For $t\neq 0$ we denote $\iota_t\colon\Cm\rightarrow\Cm_t$ a $\T\times \R^+$-equivariant biholomorphism given by the condition (i). 
  By Lemma~\ref{l:momentum-cone-coincide} and \eqref{e:ReebCone=dualMOMcone}, the Reeb cone $\tilde \tor_+ \sub \tor$ is independent of the choice of a compatible $\T$-invariant cone potential $\tilde r\in \cpot_{\tilde \K} (\tstcone)^\T$. The same argument implies that the moment cone $\Sigma_t$ associated to $(\Cm, K, \iota^*_t \tilde r)$ does not depend on $t$. We denote this cone by $\Sigma$. Using \eqref{e:coneMOMmap} we define the cone momentum maps $\mu^{t}\colon\Cm \rightarrow \tor^*$ and  $\tilde \mu\colon \tstcone \rightarrow \tor^*$ associated respectively to the  cones $(\Cm, K, \iota^*_t \tilde r)$   and $(\tstcone, \tilde K, \tilde r)$. Since condition (i) ensures that orbits of  $\T$ lie in the fibre $\Cm_t$ of $\pi$ we have, for $p\in\Cm_t$,
\[
\langle\mu^{t}_p, \Y\rangle = \tfrac{1}{4} \iota^*_t  \d^c({\tilde r}^2)(\Y)=  \tfrac{1}{4}  \d^c({\tilde r}^2)(\Y)= \langle{\tilde \mu}_p, \Y\rangle.
\]
In particular, $\tilde \mu(\tstcone\backslash\Cm_0) \sub \Sigma$ and thus $\tilde \mu(\tstcone) = \Sigma$ because they are both strictly convex polyhedral cones as recalled in \S\ref{ss:Reebcone}. The relation \eqref{e:ReebCone=dualMOMcone} concludes the proof of the first claim whereas  the condition (i) yields the second claim.  \end{proof}

Consequently, we will identify  the Sasaki--Reeb vector fields $\tilde\K$ and $\K$ and  the reduced Sasaki--Reeb cones $\tor_+$ and $\tilde \tor_+$ in Definition~\ref{d:test-configurationCONE} and in what follows.

\begin{ex}\label{ex:productcase} {\bf (Product test configuration of polarized cones)} A particular class of smooth test configurations of polarized complex cones are the \emph{product test configurations}. These are defined (similarly to Definition~\ref{d:test-configuration}) by requiring that the holomorphic surjection $\big(\tstcone\setminus\Cm_{\infty}, \pi\big)$ is $\T$-equivariantly isomorphic to the product $\big(\Cm \times (\C P^1 \setminus \{\infty\}), \pi_{\C P^1}\big)$. In the regular case, examples are obtained by taking the polarized complex cones over product polarized test configurations of the regular quotient in the sense of \cite[Example 2.8]{BHJ}. 
\end{ex}

\subsection{The global Futaki invariant of a smooth test configuration of polarized complex cones}

\begin{defn}\label{l:GF} Given a  smooth $\T$-equivariant  test configuration $(\tstcone,\K)$ of the polarized complex cone $(\Cm, \K)$  with $\K \in \tor_+$ as in Definition~\ref{d:test-configurationCONE}, and a cone potential $\tilde r\in \cpot_\K(\tstcone)^\T$, we let
\begin{equation}\label{e:GF}
\begin{split}
\tstGF_{\K} (\tstcone) := &  - \frac{1}{(m+1)!}\int_{\tstN} \Big(\Scal(\tilde g_{\K})- c_{\K}\Big){\tilde \cf}^{\K}_{\tstD} \wedge (\d \tilde \cf^{\K}_{\tstD})^{m+1} \\
                                     &  +  \frac{2}{m!} \int_{\tstN} (\pi^*\omega_{FS}) \wedge \tilde \cf^{\K}_{\Ds}\wedge (\d\tilde \cf^{\K}_{\Ds})^m,
                                     \end{split}
\end{equation}
where   $(\tstN, \tstD, \tstJ, \K)$ is the induced Sasaki manifold with contact form $\tilde \cf^\K_{\tstD}$,  $c_\K$ is the constant \eqref{cK} of $\Sm$, and $\omega_{\rm FS}$ is the Fubini--Study metric on $\C P^1$ satisfying ${\rm Ric}(\omega_{FS})= \omega_{\rm FS}$. We shall refer to \eqref{e:GF} as the {\it global Futaki invariant} of $(\tstcone, \K)$.
\end{defn} 
Note that, as shown for instance in \cite{FOW}, the first line of \eqref{e:GF} is independent of the choice of a cone potential $\tilde r\in \cpot_\K(\tstcone)^\T$ on $(\tstcone, \K)$ and so is the second line by a similar argument.

\begin{prop}\label{p:globalDFS} Let $(\tstcone, \K)$ be a smooth $\T$-invariant test configuration of polarized complex cones associated to $(\Cm,\K)$ with  $\K \in \tor_+$. Consider the corresponding polarized affine cone $(\widehat{\tstcone}, \K)$ with central fibre the polarized affine variety $(\Cv_0, \K)$ and associated Donaldson--Futaki invariant $\DF_{\Cv_0, \K}(\Y_{w})$.  Then  
\begin{equation}\label{e:CSz=GF}
\DF_{\Cv_0, \K}(\Y_{w})= \lambda(m) \tstGF_{\K}(\tstcone). 
\end{equation}
for a dimensional constant $\lambda(m)>0$.
\end{prop}
\begin{proof} As  both functionals in \eqref{e:CSz=GF}  depend continuously  on  $\K \in \tor_+$ (see \cite[Theorem~3]{CSz}) and are homogeneous of degree $-(m+1)$ it  will  be enough to prove \eqref{e:CSz=GF} for $\K$ primitive in the lattice of circle subgroups of $\T$ (since the set of such vectors is  dense in $\tor_+/\R^+$).

Assume  that $\K$ is a primitive element  in the lattice of circle subgroups of $\T$,   and denote $\C^{\times}_{\K} \leq \T_{\C}$ the subgroup generated by $\K$. The $\C^\times_\K$ GIT reduction of $\tstV$ can be identified (cf.~\cite{Reid, RT}) with the compact polarized orbifold $(\tstM,\tstL)$ associated to the grading induced by $\C^{\times}_{\K}$-action on the coordinates ring of the affine cone $\widehat{\tstcone} \sub \C^N$, which also gives an orbifold embedding of $(\tstM, \tstL)$ into a weighted projective space.      By construction,  $\tstM\simeq \widehat{\tstcone} \backslash \{0\}/\C^\times_\K
\simeq \tstcone/\C^\times_\K$, so that $\tstM$ has only cyclic orbifold singularities. The fact that $\K$ is primitive is used to determine the orbifold structure of the quotient such that the generic isotropy groups are trivial. Observe that the
$\T_{\C}$-invariant holomorphic map $\pi \colon \tstcone \ra \C P^1$ descends to
the quotient $\check{\pi} \colon \tstM \rightarrow \C P^1$ as a
$\C^\times_\rho$-equivariant surjective orbifold holomorphic map.  Then one can
define a complex orbifold $M := \Cm_1/\C^\times_\K \simeq \Cm/\C^\times_\K$ with
orbiline bundle $L=\tstL_{|_{M}}$ (the restriction of $L$ to the orbifold charts is 
taken  so that we  get a locally ample orbiline bundle in the
sense of \cite{RT}). This orbiline bundle is in fact ample (in the sense of \cite{RT}) because 
$((L^{*})^\times,\K) \simeq (\Cm_1,\K) \sub \widehat{\tstcone}$ and sections of
$\tstL$ restrict to section of $L$ over $M$. Finally,  we observe that $(\tstM,\tstL)$ defines a test configuration for $(M,L)$ in the sense of \cite{RT},  as  the map $\lambda$ in \eqref{e:equivbiholTC} (which is $\T\times \C^\times_\rho$ equivariant) descends to a $(\T/\Sph^1_\K)\times \C^\times_\rho$-equivariant map.

Using Chern--Weil theory on orbifolds (in particular that  $c_1^{orb}(L) = \frac{1}{2\pi}[\omega_\K]$ for a $2$-form $\omega_\K$  which  pulls back  as the differential of the connection $1$-form $\eta_{\Ds}^{\K}$ on the $\Sph^1$- bundle $\Sm\ra M$), we have
\begin{equation}\label{odaka-wang}
\begin{split}
\frac{1}{(2{\pi})}\tstGF_{\K} (\tstcone) &= \frac{2(2{\pi})^{m+1}}{m!}\Big\{\big(\tstK_{\tstM/\C P^1}^{orb} \cdot \tstL^{m}\big)_{\tstM} + \frac{c_{\K}}{2(m+1)}(\tstL^{m+1})_{\tstM}\Big\} \\
&= \frac{2(2{\pi})^{m+1}}{m!}\Big\{\big(\tstK_{\tstM/\C P^1}^{orb} \cdot \tstL^{m}\big)_{\tstM}  \\
& \hspace{2.5cm}- \frac{m}{m+1}\Big(\frac{\big(K^{orb}_M\cdot L^{m-1}\big)_M}{(L^m)_M}\Big)(\tstL^{m+1})_{\tstM}\Big\} 
 \end{split}
\end{equation} 
where $\tstK^{orb}$ and $K^{orb}$ stand for the (relative) canonical orbi-bundles on $\tstM$ and $M$, respectively,  and the products are taken by using integration of forms over orbifolds. Notice that up to a positive dimensional constant, \eqref{odaka-wang} is an orbifold version of the intersection formula for the Donaldson--Futaki invariant found by Odaka~\cite{odaka} and Wang~\cite{wang} in the settings of projective schemes.

We next argue that, similarly to the conclusions of~\cite{odaka,wang},
\eqref{odaka-wang} computes, up to a positive dimensional constant, the
orbifold version of the Donaldson--Futaki invariant associated to the central
fibre $(M_0:=\check{\pi}^{-1}(0), L_0:=\tstL_{M_0})$ of $(\tstM, \tstL)$,
defined in \cite[Def.~6.4]{RT}. This follows from the considerations in
\cite[pp. 315]{Do-02}~\footnote{We are grateful to X. Wang for pointing out
to us  this reference and to J. Ross for answering our questions.}, namely that for $k\gg0$ we have exact sequences
\begin{equation*}
 \begin{split}
  0&\rightarrow H^0(\tstM, \tstL^k\otimes\check{\pi}^*\cO(-1)) \rightarrow H^0(\tstM, \tstL^k) \rightarrow H^0(M_0, L_0^k) \rightarrow 0 \\
  0&\rightarrow H^0(\tstM, \tstL^k\otimes\check{\pi}^*\cO(-1)) \rightarrow H^0(\tstM, \tstL^k) \rightarrow H^0(M, L^k) \rightarrow 0 
 \end{split}
\end{equation*} given by multiplication by the sections of $\cO(-1)$ defining $M_0$ and $M_\infty$, respectively.  This implies that $d_k:=\dim_\C H^0(M_0, L_0^k)= \dim_\C H^0(M, L^k)$ whereas the weight $w_k$ for the induced $\C^\times$-action on $H^0(M_0, L_0^k)$ is 
$$w_k = \dim_\C H^0\big(\tstM, \tstL^k\otimes \check{\pi}^*\cO(-1)_{\C P^1}\big)=  \dim_\C(\tstM, \tstL^k)- \dim_\C H^0(M, L^k).$$ 
Now the claim follows from the definition \cite[Def.~6.4]{RT} and an easy computation using the orbifold Riemann--Roch formula \eqref{e:RR}.

To conclude the proof of the proposition,   we use  that $\DF_{\Cv_0, \K}(\Y_{w})$ too coincides with the Ross--Thomas orbifold Futaki invariant on $(M_0, L_0)$, see \cite[Def.~5.2]{CSz} and the discussion following it. \end{proof}

\subsection{Sasaki K-stability versus weighted K-stability} Our main observation is the following result which links the relative Sasaki K-stability notion of \cite{CSz} with the relative weighted K-stability of \cite{lahdili2}.
\begin{prop}\label{p:Weighted=GF} Let $(\Sm, \Ds, J, \X)$ be a compact regular Sasaki manifold, and $(M, J, L)$ the corresponding smooth polarized K\"ahler manifold.  Let $\T \leq \Aut(\Sm, \Ds, J)$ be a compact torus with $\X\in \tor_+$ and  $\K \in \tor_+$  another  Sasaki--Reeb vector field on $(\Sm, \Ds, J)$. Then,   for any $\T$-equivariant smooth polarized test configuration $(\tstM, \tstL)$ associated to $(M, J, L, \T)$ we have 
 \begin{equation}
\Fut_\K(\tstM, \tstL) = \frac{1}{2\pi} \tstGF_\K (\tstC),
\end{equation} 
where $\Fut_\K(\tstM, \tstL)$ is introduced by \eqref{abdellah-futaki-0}.
\end{prop}
\begin{proof} The definition \eqref{(f,wt)-scalar-curvature} can be rewritten as
\begin{equation*}
\begin{split}
\Scal_{\tilde f, m+2}(\tilde \Omega)\tilde f^{-m-3}
& = \bigl(\Scal_{\tilde f, m+3}(\tilde \Omega)  + 2 \tilde f \Delta_{\tilde \Omega} \tilde f  + 2(m+2) |\d\tilde f|^2_{\tilde \Omega}\bigr)\tilde f^{-m-3}\\
&= \Scal_{\tilde f, m+3}(\tilde \Omega) \tilde f^{-m-3}
+  2 \delta_{\tilde\Omega}\bigl(\tilde f^{-m-2}\d f\bigr).
\end{split}
\end{equation*}
Hence by Stokes' Theorem
\begin{equation}\label{miracle}
\int_{\tstM} \Scal_{\tilde f, m+2}(\tilde \Omega)\tilde f^{-m-3} \tilde \Omega^{m+1}
=  \int_{\tstM} \Scal_{\tilde f, m+3}(\tilde \Omega)
\tilde f^{-m-3} \tilde \Omega^{m+1},
\end{equation}
so  that \eqref{abdellah-futaki-0} is equivalently given by
\begin{equation}\label{abdellah-futaki}
\Fut_\K(\tstM, \tstL) =  -\int_{\tstM} \Big(\Scal_{\tilde f}(\tilde \Omega) - c_\K{\tilde f}\Big) \tilde f^{-3-m} \frac{{\tilde \Omega}^{m+1}}{(m+1)!}   +  2 \int_\tstM \tilde f^{-1-m}\pi^*\omega_{FS}\wedge \frac{\omega^m}{m!},
\end{equation}
where $\Scal_{\tilde f}(\tilde \Omega)=  \Scal_{\tilde f, m+1}(\tilde \Omega)$ is $\tilde f$-scalar curvature of the K\"ahler manifold $(\tstM, \tilde \Omega)$.  Formula \eqref{scal-twist} tells us that on the Sasaki manifold  $(\tstN, \tstD, \tstJ,\X)$ we have 
\[ {\tilde \cf}_{\tstD}^\X = {\tilde f} {\tilde \cf}_{\tstD}^\K,  \, \tilde f = {\tilde \cf}_{\tstD}^\X(\K),  \,  \Scal_\K(\tilde g_\X)= \Scal_{\tilde f}(\tilde \Omega)= \tilde f \Scal({\tilde g}_\K)\]
and we compute
\[
\begin{split}
\Fut_\K(\tstM, \tstL)  =&  -\int_{\tstM} \Big(\Scal_{\tilde f}(\tilde \Omega) - c_\K{\tilde f}\Big) \tilde f^{-3-m} \frac{{\tilde \Omega}^{m+1}}{(m+1)!}   +  2 \int_\tstM \tilde f^{-1-m}\pi^*\omega_{FS}\wedge \frac{\omega^m}{m!} \\
   = &  - \frac{1}{2\pi (m+1)!}\int_{\tstN} \big(\Scal_\K(\tilde g_\X) -c_\K\tilde\cf_{\tstD}^\X(\K)\big)\tilde\cf_{\tstD}^\X(\K)^{-3-m}   {\tilde \cf}_{\tstD}^\X \wedge (\d {\tilde \cf}_{\tstD}^\X )^{m+1}  \\
   &  + \frac{2}{(2\pi) m!} \int_{\tstN}     \pi^*(\omega_{FS})  \wedge \tilde\cf_{\tstD}^\X(\K)^{-1-m}   \tilde\cf_{\tstD}^\X  \wedge (\d {\tilde \cf}_{\tstD}^\X )^{m}    \\
   =&    \frac{1}{2\pi}  \tstGF( \tstC).   
\end{split}
\]

\end{proof}

As a direct consequence of Propositions~\ref{p:globalDFS} \& \ref{p:Weighted=GF} and  Theorem~\ref{Theorem:polystable} we obtain the following ramification of the result in \cite{CSz}.
 \begin{cor}\label{c:polystability} Let $(\Cm,\K)$ be the  polarized complex cone associated to a compact Sasaki manifold $(\Sm, \Ds, J, \K)$,  $\T \leq \Aut(\Cm,J)^\K$ a maximal torus with $\K \in \tor_+$ and assume that there exists a regular Sasaki--Reeb vector field $\X \in \tor_+$. If $(\Cm, \K)$ admits a scalar-flat K\"ahler cone metric,  then for any smooth  $\T$-equivariant  test configuration $(\tstC,\K)$  of $(\Cm,\K)$,   obtained  from a  smooth $\T$-equivariant polarized test configuration $(\tstM, \tstL)$ associated to the regular quotient $(M, J, L, \T)$,  which has reduced central fibre and is not a product,  we have $\DF_{\Cv_0, K}(\Y_{w}) >0$.
\end{cor}
\begin{proof} It is well-known (see e.g.~\cite{BG-book}) that a scalar-flat K\"ahler cone metric on $(\Cm, J, \K)$ corresponds to a CSC Sasaki structure in $\cS(\K, J^\K)$,  such that the constant $c_K$ defined by \eqref{cK} equals $m(m+1)$. Thus, we are under the assumptions of Theorem~\ref{Theorem:polystable} with $\ell_\ext(x) = c_K \ell_K(x)$, i.e., $\Fut^\ext_K(\tstM, \tstL)= \Fut_K(\tstM, \tstL) >0$.  By Proposition~\ref{p:Weighted=GF}  we have $\tstGF_\K (\tstC)>0$ whereas Proposition~\ref{p:globalDFS}  yields $\DF_{\Cv_0, K}(\Y_{w}) >0$.
\end{proof}

\section{Applications}\label{s:applications} In this section we solve  the existence problem  of extremal  Sasaki metrics on a special class of CR manifolds of Sasaki type, extending Theorem 2 in \cite{AC} to arbitrary dimension.

\subsection{The Calabi construction on admissible manifolds} This is the familiar construction of special K\"ahler metrics on $\PP^1$ bundles,  going back to Calabi~\cite{calabi} and studied in many other places.

We essentially follow the notation of \cite{HFKG3}  and consider $(M, J)= \PP(\cO \oplus L) \to B$ to be the total space of a $\PP^1$-bundle over the product $B= B_1 \times \cdots \times B_k$ of  compact CSC K\"ahler manifolds  $(B_j, J_j, g_j, \omega_j)$ satisfying the Hodge condition $\omega_j = 2\pi c_1(L_j)$ for a polarization $L_j$ on $B_j$, and  $L = \bigotimes_{j=1}^k  L_j^{p_j}$  with $p_j \in \intZ$ being the induced line bundle  on $B$. We fix real constants $c_j \in \R$ such that the affine linear functions $p_j z + c_j>0$ on the interval $P=[-1,1] \sub \R$ (equivalently, we require $c_j >|p_j|$).

 Then, as shown for instance in \cite{HFKG3},  any  smooth function $\Theta(z)$ on $[-1,1]$ satisfying the following boundary conditions
\begin{equation}\label{boundary}
\Theta(\pm 1) = 0, \ \ \Theta'(\pm 1) = \mp 2,
\end{equation}
and 
\begin{equation}\label{positivity}
\Theta(z) >0  \ \textrm{on} \ (-1,1).
\end{equation}
gives rise to a K\"ahler metric on $M$ of the form
\begin{equation}\label{metric}
\begin{split}
g &= \sum_{j=1}^k (p_j z+ c_j) \pi^* g_j  +  \frac{\d z^2}{\Theta(z)} + \Theta(z) \theta^2, \\
\omega &= \sum_{j=1}^k  (p_j z+ c_j) \pi^* \omega_j  + dz \wedge \theta,   \,  \d\theta = \sum_{j=1}^k p_j \pi^*\omega_j,
\end{split}
\end{equation}
which is isometric to a K\"ahler metric in the same K\"ahler  class $\kcl_c$ of $(M, J)$ (here $c=(c_1, \ldots, c_k)$ stand for the fixed real constants in the definition \eqref{metric}). Notice that  in \eqref{metric}, $\hTheta$ is  a  connection 1-form on the principal $\Sph^1$-bundle $P$ over $B$,   corresponding to line bundle  $L^*$, and $M= \PP^1 \times_{\Sph^1} P \to B$ is the corresponding fibre bundle associated to the toric variety $(\PP^1, \cO(2))$. The induced $\Sph^1$-action on $M$ corresponds to fibrewise rotations in the trivial factor $\cO$ of $\cO\oplus L$, and  \eqref{metric} is $\Sph^1$-invariant with  $m_\omega=z$ being  the corresponding momentum map with  image $P=[-1, 1] \sub \R$. For any real constants   $b>|a|$, $f(z) = az + b$  is a positive Killing potential  with respect to \eqref{metric}  for the Killing vector field $a\check\K$ where $\K$ is the generator of the $\Sph^1$-action,  and the constant $b$ is essentially  the normalization constant $\kappa>0$ of \S\ref{ss:calabi-problem-kahler}.   The $f$-scalar curvature of the  K\"ahler metric \eqref{metric} is computed in \cite{AMT} to be
\begin{equation}\label{scal-f-metric}
\Scal_f(g)=  -\frac{\Big(f(z)^{-m-1}p_c(z) \Theta(z)\Big)''}{p_c(z)f(z)^{-m-3}} + f^2(z)\sum_{j=1}^k \frac{\Scal_j}{p_jz+ c_j},
\end{equation}
where $\Scal_j$ stands for the constant scalar curvature of $(B_j, g_j)$ and
\[
p_c(z):= \prod_{j=1}^k (p_jz+ c_j)^{\dim_\C B_j}
\]
is a polynomial of degree $(m-1)$.  By \cite[Lemma 2.4]{AMT}, \eqref{metric} gives rise to an $f$-extremal K\"ahler metric if and only if  the smooth function $\Theta(z)$ on $[-1,1]$ satisfies  \eqref{boundary}, \eqref{positivity} and the condition
\begin{equation}\label{extremal}
\Big( f(z)^{-m-1}p_c(z) \Theta(z)\Big)'' =  p_c(z)f(z)^{-1-m} \Big(\sum_{j=1}^k \frac{\Scal_j}{p_jz + c_j}\Big)  -(Az+B)f(z)^{-m-3}p_c(z)
\end{equation}
for some (unknown) real constants $A$ and $B$.  The general solution of \eqref{extremal}  thus depends on $4$ real constants (including $A$ and $B$) and the $4$ boundary conditions in \eqref{boundary} determine these constants uniquely, as it follows from the arguments of  \cite[Prop.~2.2]{AMT}. We denote by $\ell^\ext_{a,b,c}(z)=A_{a,b,c}z + B_{a,b,c}$  the corresponding affine-linear function and 
by  $\Theta^\ext_{a,b,c}(z)$  the  corresponding solution  of \eqref{extremal},  verifying \eqref{boundary} but possibly not \eqref{positivity}.

Up to a suitable renormalization, this is the setting  of \cite{HFKG3, AMT}, except that we allowed $p_j=0$ (which introduces a reducible manifold $(M,J)$ on which nontrivial solutions of the $f$-extremal equation can exist when $f\neq const$), and  did not introduce here blow-downs.  Indeed, when $p_j\neq 0$, in order to get to the notation of \cite{HFKG3},  one simply needs to replace $\omega_j$ by (the possibly negative definite) CSC K\"ahler metric $p_j \omega_j$ and thus letting in \eqref{metric} $p_j=1$ and $c_j = \frac{1}{x_j}, \, 0<|x_j|<1$. 

\begin{defn}\label{admissible}\cite{HFKG3} A K\"ahler manifold $(M, J)$ obtained as above is called {\it admissible}. The K\"ahler metric on it and the corresponding K\"ahler class $\kcl_c$ defined by \eqref{metric} for some choice of $c=(c_1, \ldots, c_k)$ with $c_j > |p_j|$ and smooth function $\Theta(z)$ are called {\it admissible}.
\end{defn}
Our main result relevant to the admissible setting is the following explicit Yau--Tian--Donaldson type correspondence.
\begin{Theorem}\label{conjecture1} For an admissible K\"ahler manifold $(M, J, \kcl_c)$ as above, there exists an $\Sph^1$-invariant  $(a\check\K, b)$-extremal K\"ahler  metric in the K\"ahler class $\kcl_c$ if and only if the function $\Theta^\ext_{a,b,c}(z)$  is positive on $(-1,1)$. In this case,  $(M, J)$ admits an admissible $(a\check\K, b)$-extremal K\"ahler  metric of the form \eqref{metric}.
\end{Theorem}
Theorem~\ref{conjecture1} is established in various special cases, but to the best of our knowledge was  open in the above generality prior to this work. Indeed, in the non-weighted case (which correspond to $(a,b)=(0,1)$) Theorem~\ref{conjecture1}  was first proved in \cite{sz} in the case of  polarized  ruled surfaces,  and in \cite{HFKG3} in general, whereas on  a ruled complex surface and arbitrary $(a, b)$ with $|a|<b$  it is established in \cite[Prop.~7]{AC}. Many other special cases are discussed in \cite{AMT}.

It is not hard to see that when the real constants $c_1, \ldots, c_k$ in \eqref{metric} are all rational, the corresponding K\"ahler class $\kcl_c$ on $(M,J)$ belongs to $H^2(M, \Q)$.  Thus,   for  a suitable homothety factor $\lambda_c>0$, a K\"ahler metric $(\lambda_cg, \lambda_c \omega)$ given by \eqref{metric} defines  a smooth CR manifold  $(\Sm, \Ds, J)$ of Sasaki type (see Example~\ref{e:regular}), endowed with a $2$-dimensional torus $\T \leq \Aut(\Sm,  \Ds, J)$ and  a regular Sasaki--Reeb element  $\X \in \tor_+$ generating an $\Sph^1_\X \leq \T$, such that $M$ is the Sasaki--Reeb quotient   with respect to $\X$ and has the induced $\check \T= \T/\Sph^1_\X= \Sph^1_{\check\K}$  action generated by $\check\K$. In other words,  the choice of a positive Killing potential $f=az+b,  b > |a]$ with respect to the  K\"ahler metric $(\lambda_c g, \lambda_c\omega)$  defines  an element $\K_{a,b} \in \tor_+$,  inducing $a\check\K$ on $M$. In this setting, Proposition~\ref{p:reduction}  and  Theorem~\ref{conjecture1} yield the following generalization of \cite[Theorem 2]{AC}.
\begin{cor}\label{c:extension} $(\Sm, \Ds, J)$ admits an extremal Sasaki metric in $\cS(\K_{a,b} ,J^{\K_{a,b}})$ if and only if the corresponding smooth function $\Theta^\ext_{a,b,c}(z)>0$ on $(-1,1)$.
\end{cor}
Notice that in Theorem~\ref{conjecture1}, the usually more difficult part, namely establishing the existence of an $(a\check\K, b)$-extremal K\"ahler metric under a suitable positivity algebraic condition, is trivial because if $\Theta^\ext_{a,b,c}(z)$  verifies \eqref{positivity}, then letting $\Theta(z)=\Theta^\ext_{a,b,c}(z)$ in \eqref{metric} defines an admissible  $(a\K, b)$-extremal K\"ahler  in $\kcl_c$.
In the other direction, Lahdili shows in \cite[Theorem 8]{lahdili2} that for each $z_0\in (-1,1)$, one can find an $\Sph^1$-invariant smooth test configuration $\tstM$ which, as a complex manifold,  is the degeneration to the normal cone~\cite{RT} of $(M, J)$  with respect to the infinity section $B_\infty:=\PP(\cO \oplus 0) \sub M$,   together with  a K\"ahler class $\tstA_{z_0}$, such that the corresponding relative $(\check\K, \kappa)$-weighted Donaldson--Futaki invariant introduced in \S\ref{ss:(K,kappa)-stability} satisfy
$$\Fut_{a\check\K,b}^\ext(\tstM, \tstA_{z_0}) = C(z_0) \Theta^\ext_{a,b,c}(z_0)$$ 
for a positive  function $C(z)$ on $(-1,1)$.  Thus, Theorem~\ref{Theorem:lahdili1} yields (see \cite[Theorem~8]{lahdili2}) that $\Theta^\ext_{a,b,c}(z) \ge 0$ on $(-1,1)$ and the main difficulty in the proof of  Theorem~\ref{conjecture1} is to establish  the  strict inequality in Theorem~\ref{Theorem:lahdili1} for the test configurations $(\tstM, \tstA_{z_0})$ as above. In particular, Conjecture~\ref{conjecture0} would imply Theorem~\ref{conjecture1} for $c$ rational.

\subsection{The case of arbitrary weights}
Before giving the proof of  Theorem~\ref{conjecture1}, we will need to recall the more general setting of $(v, w)$-extremal K\"ahler metrics introduced in \cite{lahdili2}. We limit the discussion to the $\PP^1$-fibre bundle case, referring  the Reader to \cite{lahdili2} for the most general setting.  Following \cite[Section 10]{lahdili2},  let $v(z)$ and $w(z)$  be  two positive smooth functions on $[-1,1]$,  and let us look, more generally, for K\"ahler metrics on $M=\PP(\cO \oplus L) \to B$ of the form \eqref{metric},  satisfying the  equation
\begin{equation}\label{weighted-extremal}
\Big(v(z)p_c(z) \Theta(z)\Big)'' =  p_c(z)v(z) \Big(\sum_{j=1}^k \frac{\Scal_j}{p_jz + c_j}\Big)  -(Az+B)w(z)p_c(z),
\end{equation}
for some (unknown) real constants $A, B$.
If we set $v(z)=(az+b)^{-m-1}$ and $w(z)= (az+b)^{-3-m}$ in \eqref{weighted-extremal},  we obtain \eqref{extremal}.  For general positive weight functions $(v, w)$ defined over $P=[-1,1]$, \eqref{weighted-extremal} describes a {\it $(v, w)$-extremal K\"ahler metric} in the K\"ahler class $\kcl_c$ in the sense of \cite{lahdili2}, which is given by the Calabi construction \eqref{metric}.
\begin{defn} \cite{lahdili2} Let $v(z), w(z)$ be given positive smooth functions on $[-1,1]$ and $\tilde \omega \in \kcl_c$ an $\Sph^1$-invariant  K\"ahler metric with corresponding normalized momentum map $m_{\tilde \omega}\colon M \to [-1, 1]$. We say that $\tilde \omega$ is {\it $(v, w)$-extremal} if it satisfies
\begin{equation}\label{(v,w)-extremal}
\Scal_{v}(\tilde g):= v(m_{\tilde \omega}) \Scal(\tilde g) + 2\Delta_{\tilde g} v(m_{\tilde \omega})+ v''(m_{\tilde \omega})\tilde g(\K, \K) = (Am_{\tilde \omega} +B) w(m_{\tilde \omega}).
\end{equation}
\end{defn}
One can argue again that  for $v(z), w(z)$ and $c=(c_1, \ldots, c_k)$ fixed,  \eqref{weighted-extremal} admits a unique smooth solution $\Theta^\ext_{v, w, c}(z)$, satisfying \eqref{boundary} but possibly not \eqref{positivity}. Notice that the constants $A, B$ can be determined  by  integrating  \eqref{weighted-extremal} on $[-1,1]$ against  the affine-linear functions $1$ and $z$. Integration by parts and using \eqref{boundary}  yields a linear system  for  $A,B$,  independent of $\Theta$. This leads to the notion of the {\it extremal affine-linear function} $\ell^\ext_{v,w,c}(z):= A_{v,w,c}z + B_{v,w,c}$ associated to $(v, w, \kcl_c)$.

By establishing a suitable notion of relative $(v, w)$-K-stability, the following result is obtained in \cite[Theorem~8]{lahdili2}.
\begin{Theorem}\label{Theorem:lahdili3}\cite{lahdili2} If  $\Theta^\ext_{v, w, c}(z_0)< 0$ for some $z_0\in (-1,1)$ then $(M, J)$ admits no $(v, w)$-extremal K\"ahler metric in $\kcl_c$.
\end{Theorem}
We shall next improve the above statement under a  technical assumption allowing to deduce Theorem~\ref{conjecture1}.
\begin{Theorem}\label{Theorem:clever} If $\Theta^\ext_{v, w, c}(z)$ vanishes at $z_0\in (-1,1)$ but  $\ell^\ext_{v,w,c}(z_0)\neq 0$, then there is no $(v, w)$-extremal metric in $\kcl_c$. 
\end{Theorem}
\begin{proof} Suppose $\Theta^\ext_{v, w, c}(z_0)= 0$ for some $z_0\in (-1,1)$ and $\ell^\ext_{v,w,c}(z_0)\neq 0$, where $\ell^\ext_{v,w,c}(z)=Az+B$ is the  extremal affine-linear function. Assume for contradiction that $\kcl_c$ admits a $(v,w)$-extremal K\"ahler metric. By Theorem~\ref{Theorem:lahdili3}, $\Theta^\ext_{v, w, c}(z)\ge 0$ on $(-1,1)$.  The idea is to violate Theorem~\ref{Theorem:lahdili3} by varying the weights $(v, w)$.

Consider the smooth family of functions
$$\Theta_t(z):= \Theta^\ext_{v, w, c}(z) - t(1-z^2)^2 (Az + B)^3.$$
By its very construction, $\Theta_t(z)$ satisfies \eqref{boundary} for each $t$, $\Theta_0(z) =  \Theta^\ext_{v, w, c}(z)$ and  either for all $t>0$ or  for all $t<0$ we have $\Theta_t(z_0) <0$. Suppose that $\Theta_t(z_0) <0$ for $t>0$.

Notice that  $\Theta_t(x)$ solves \eqref{weighted-extremal}  with respect to the weight functions $v(x)$ and 
$$w_t(z) = w(z)  + t\tilde w(z), \ \ \tilde w(z):= \frac{\Big(v(z)p(z)(1-z^2)^2 (Az + B)^3\Big)''}{(Az + B)p(z)}. $$ 
As $\tilde w(z)$ is a smooth function on $[-1,1]$, $w_t(z)$ is positive and smooth on $[-1,1]$ for $t$ small enough. It follows that $\Theta_t(z)=\Theta^\ext_{v, w_t, c}(z)$ is the unique solution of \eqref{weighted-extremal} for the weights $(v, w_t)$, satisfying \eqref{boundary}. Using Theorem~\ref{Theorem:lahdili3} again, we conclude that  for small positive values of $t$, there is no $(v, w_t)$-extremal K\"ahler metric in $\kcl_c$. This contradicts Theorem B2 in \cite{lahdili2}. \end{proof}

\subsection{Proof of Theorem~\ref{conjecture1}}
\begin{proof}
As we have already explained, we are in the special case of  Theorem~\ref{Theorem:clever}  with $v(z)=(az+b)^{-1-m}$ and  $w(z)=(az +b)^{-3-m}$ ($b>|a|$).  Accordingly, it is enough  to assume  that  $(M,J)$ admits an $(a\K,b)$-extremal K\"ahler metric in $\kcl_c$ but $\Theta^\ext_{a,b,c}(z_0)=\ell^\ext_{a,b,c}(z_0)=0$ for some $z_0 \in (-1,1),$ and derive a contradiction.  To this end, we shall use a modification of the arguments from \cite{HFKG3}. 

\bigskip
We start with the preliminary observation that $\Theta^\ext_{a,b,c}(z)$ and  $\ell^\ext_{a,b,c}(z)$ are rational functions of $z$ whose coefficients are rational functions of $(a,b,c)$. Indeed, letting $P_{a,b,c}(z):=p_c(z)\Theta^\ext_{a,b, c}(z)$ and multiplying both sides of \eqref{weighted-extremal} with $(az+b)^{m+3}$, it follows that $P^\ext_{a,b,c}(z)$ solves the ODE
\begin{equation}\label{polynomial-solution}
\begin{split}
& (az+b)^2F''(z) -2a(m+1)(az+b)F'(z) +a^2(m+1)(m+2)F(z) \\
& =  p_c(z)(az+b)^2 \Big(\sum_{j=1}^k \frac{\Scal_j}{p_jz + c_j}\Big)  -(Az+B)p_c(z).
\end{split}
\end{equation}
where $A$ and $B$ are real constants (equal to $A_{a,b,c}$ and $B_{a,b,c}$ for the specific solution $P_{a,b,c}(z)$). When $a\neq 0$, using that the monomials $(az+b)^s, s=0, \ldots, m$ are eigenfunctions for nonzero eigenvalues of the linear differential operator in \eqref{polynomial-solution}, we see that the general solution of \eqref{polynomial-solution} is a polynomial of degree $\le m+2$ of the form
\[
F(z) = \sum_{s=0}^{m+2}p_{s}(az+b)^s
\]
whose coefficients $p_0, \ldots, p_m$ are determined from $a,b,c$ and the two real constants $A,B$, whereas the coefficients $p_{m+2}$ and $p_{m+1}$ are the 2 free constants of integration. Imposing the boundary conditions (see \eqref{boundary}) 
\begin{equation}\label{polynomial-boundary}
F(\pm 1)=0, \, F'(\pm 1)= \mp 2p_c(\pm 1)
\end{equation} 
leads to a linear system for $(p_{m+2}, p_{m+1}, A, B)$. Under the geometric constraints $|a|<b$ and $c_j >|p_j|$ (and using that the setting is invariant under a homothety of $(a,b)$)  it follows from the arguments in  \cite[Prop.~2.2]{AMT} that the system admits a unique solution.   This conclusion holds as long as the corresponding $4\times 4$ system is nondegenerate. In summary, for the values of $(a, b, c)$ such that $a<|b|, a\neq 0, c_j >|p_j|$,   $P_{a,b,c}(z)$ is a polynomial of degree $\le m+2$, whose coefficients are rational functions of $a, b, c$, and, similarly,  $\ell^\ext_{a,b,c}(z)= A_{a,b,c}z + B_{a,b,c}$ is an affine-linear function whose coefficients are  rational functions of $(a,b,c)$.  The definition of  $P_{a,b,c}(z)$ and $\ell^\ext_{a,b,c}(z)$ can be also extended for $a=0, b>0$ and $c_j >|p_j|$ (e.g.~by the arguments  in the  proof of  \cite[Prop.~2.2]{AMT}) and for values $(a,b,c)$  in  a maximal connected component $\mathcal P$  of the geometric solutions (i.e., satisfying  $a<|b|, c_j>|p_j|$).

\smallskip
We  next observe that under the assumption made, for all $(\tilde a, \tilde b, \tilde c)$ sufficiently close to the initial value $(a,b,c)$,   $(M, J)$ admits a $(\tilde a \check\K, \tilde b)$-extremal K\"ahler metric in $\kcl_{\tilde c}$,  $\Theta^\ext_{\tilde a, \tilde b, \tilde c}(z) \ge 0$ on $(-1,1)$ and $\Theta^\ext_{\tilde a,\tilde b,\tilde c}(\tilde z_0)=0=\ell^\ext_{\tilde a, \tilde b, \tilde c}(\tilde z_0)$ for some $\tilde z_0 \in (-1,1)$. (This implies that $-\frac{B_{\tilde a, \tilde b, \tilde c}}{A_{\tilde a, \tilde b, \tilde c}}$ is a double root, possibly at infinity,  of the degree $(m+2)$ polynomial $P_{\tilde a, \tilde b, \tilde c}(z)$.) Indeed, otherwise, by  the openness result \cite[Theorem B2]{lahdili2},  we can find a sequence $\lim_{k\to \infty} (a_k, b_k, c_k) =(a,b,c)$ such that $(M,J)$ admits an $(a_k \check\K, b_k)$-extremal K\"ahler metric in $\kcl_{c_k}$ (and therefore $\Theta^\ext_{a_k, b_k, c_k}(z) \ge 0$ on $(-1,1)$ by Theorem~\ref{Theorem:lahdili3}) but $\Theta^\ext_{a_k,b_k,c_k}(z)$  and $\ell^\ext_{a_j, b_j, c_j}(z)$ do not having a common zero in $(-1,1)$. By Theorem~\ref{Theorem:clever}, $\Theta^\ext_{a_k,b_k,c_k}(z)>0$ on $(-1,1)$, i.e., there exist an admissible $(a_k\check\K, b_k)$-extremal K\"ahler metrics $(g_k,\omega)$ in $\kcl_{c_k}$ of the form \eqref{metric}. The latter are, by the results in \cite{HFKG2}, precisely the K\"ahler metrics on $(M,J)$ admitting a hamiltonian 2-form of order $1$ with associated Killing vector field $\check{\K}$. Using the uniqueness result of \cite{lahdili3} (see also Lemma~\ref{l:uniqueness} in this paper), we conclude as in the proof of Theorem~2 in \cite{HFKG3} that the initial $(a\check\K,b)$-extremal metric in $\kcl_c$ must also admit a hamiltonian $2$-form of order $1$ associated to $\check\K$,  a contradiction.

The conclusion, which will are going to use in order  to  find a contradiction,  is that (by analytic continuation)  for all $(\tilde a, \tilde b,  \tilde c) \in \mathcal P$,   $P_{\tilde a,\tilde b,\tilde c}(z)$ is a polynomial of  degree  $\le (m+2)$, which  has a double root at $-\frac{B_{\tilde a,\tilde b,\tilde c}}{A_{\tilde a,\tilde b,\tilde c}}$, i.e., 
\begin{equation}\label{key-relation}
(A_{\tilde a, \tilde b, \tilde c})^{m+2}P_{\tilde a, \tilde b, \tilde c}\Big(-\frac{B_{\tilde a, \tilde b, \tilde c}}{A_{\tilde a, \tilde b, \tilde c}}\Big)=0, \qquad (A_{\tilde a, \tilde b, \tilde c})^{m+1}P^{\prime}_{\tilde a, \tilde b, \tilde c}\Big(-\frac{B_{\tilde a, \tilde b, \tilde c}}{A_{\tilde a, \tilde b, \tilde c}}\Big)=0.
\end{equation}

\bigskip
We thus reduced the proof to test \eqref{key-relation} in the extended domain of definition $\mathcal P$  of $(a,b,c)$ and show that \eqref{key-relation} cannot hold for all such values of the parameters.  We are going to specialize \eqref{key-relation} for $a=1=b$. In this case,  the system \eqref{polynomial-boundary} for the general solution of \eqref{polynomial-boundary} is still nondegenerate. Indeed,  the boundary conditions at $z=-1$ of the general solution hold if and only if  $A_{1,1,c}=B_{1,1,c}=-2m(m+1)$ which in turn identifies all but the coefficients $p_{m+2}$ and $p_{m+1}$ of the general solutions of \eqref{polynomial-solution}.  The boundary conditions at $z=1$ thus fix the constants $p_{m+2}$ and $p_{m+1}$. In particular, $(1,1,c)\in \mathcal P$. Notice that by \eqref{polynomial-boundary},  $z=-1$ is a single zero of $P_{1,1,c}(z)$, which shows that the second equality in \eqref{key-relation} does not hold. \end{proof}

\appendix	

\section{Holomorphic versus contact viewpoint}\label{s:h-vs-c}

For any compact co-oriented contact $(2m+1)$-manifold $(\Sm,\Ds)$, let
$\Cx_+(\Sm, \Ds)^\T$ be the space of $\T$-invariant CR structures $J$ on
$\Ds$, and let
\[
\con_+(\Sm, \Ds)=\{\X\in \con(\Sm, \Ds)\st \cf_\Ds(\X)>0
\]
(using the chosen orientation of $T\Sm/\Ds$). Given a (compact, real) torus
$\T$ in $\Con(\Sm,\Ds)$ let $\Con(\Sm,\Ds)^\T$ denote the group of
$\T$-equivariant contact transformations, with Lie algebra $\con(\Sm,\Ds)^\T$
(in which the Lie algebra $\tor$ of $\T$ is central).

Let $(\Sm,\Ds_0,J_0,\X)$ be a fixed Sasaki manifold invariant under a torus
$\T$, let $\cf_0:=\cf_{\Ds_0}^\X$ be its contact form, and denote the induced
transversal holomorphic structure on $\Ds_\X$ by $J^\X_0$. We consider the
following two ways of parametrizing a family of CR structures on $\Sm$.
\begin{bulletlist}
\item The \emph{contact viewpoint} refers to the space $\Cx_0(\X,\Ds_0)^\T
  \sub\Cx_+(\Sm,\Ds_0)^\T$ of $\T$-invariant CR structures $J$ on $\Ds_0$ such
  that the induced transversal holomorphic structure $J^\X$ is the
  pushforward of $J^\X_0$ by some $\Phi\in\Diff_0(\Sm)^\T$. Note that $J$ is
  uniquely determined by $J^\X$ (via the given fixed $\Ds_0$ and $\X$), and
  it is sometimes convenient to view $T_J\Cx_+(\Sm,\Ds_0)^\T$ as a space of
  endomorphisms of $\Ds_\X$ rather than $\Ds_0$.
\item The (\emph{transversal}) \emph{holomorphic viewpoint} refers to the
  space $\cS(\X,J^\X_0)^\T$ of $\T$-invariant contact forms
  $\cf=\cf_\Ds^\X$ which are compatible with the transversal holomorphic
  structure $(\X,J^\X_0)$, and hence induce a CR structure $(\Ds,J)$ where
  $J$ is the lift of $J^\X_0$ to $\Ds$.
\end{bulletlist}

Note that the groups $\Con_0(\Sm,\Ds_0)^\T:=\Con(\Sm,\Ds_0)\cap
\Diff_0(\Sm)^\T$ and $\Aut_0(\Sm,\X,J^\X_0)^\T:=\Aut(\Sm,\X,J^\X_0)\cap
\Diff_0(\Sm)^\T$ act on $\Cx_0(\X,\Ds_0)^\T$ and $\cS(\X,J^\X_0)^\T$
respectively by pushforward. Imitating arguments in~\cite{gauduchon-book}, we
have the following Sasaki analogue of a fundamental result in K\"ahler
geometry.

\begin{prop} The relation $\sim$ between $\Cx_0(\X,\Ds_0)^\T$
and $\cS(\X,J^\X_0)^\T$, defined by $J\sim \cf$ if and only if there exists
$\Phi\in\Diff_0(\Sm)^\T$ such that $(\Ds_0,J^\X)=\Phi\cdot
(\Ds,J^\X_0)$, where $\Ds=\ker\cf$, induces respectively
$\Con_0(\Sm,\Ds_0)^\T$ and $\Aut_0(\Sm,\X,J^\X_0)^\T$ invariant maps
\begin{align*}
  p&\colon\Cx_0(\X,\Ds_0)^\T\to \cS(\X,J^\X_0)^\T/\Aut_0(\Sm,\X,J^\X_0)^\T
\text{ and }\\
  q&\colon\cS(\X,J^\X_0)^\T\to \Cx_0(\X,\Ds_0)^\T/ \Con_0(\Sm,\Ds_0)^\T,
\end{align*}
which define a bijection
\[
\Cx_0(\X,\Ds_0)^\T/\Con_0(\Sm,\Ds_0)^\T\cong
\cS(\X,J^\X_0)^\T/\Aut_0(\Sm,\X,J^\X_0)^\T.
\]
The derivative of $p$ at $J$, with $J^\X=\Phi\cdot J^\X_0$, is given by
$\d p_J(-\cL_\Y J^\X)=\Phi^*(\cL_\Y\cf_0)$ modulo
$\aut(\Sm,\X,J^\X_0)^\T$, and the derivative of $q$ at $\cf=\Psi\cdot
\cf_0^\X$ is given by $\d q_\cf(-\cL_\Y\cf)= \Psi^*(\cL_\Y J^\X_0)$
modulo $\con(\Sm,\Ds_0)^\T$.
\end{prop}
\begin{proof} Suppose $J\sim \cf$; then $\tilde J\sim\cf$ if and
only if there exists $\Phi\in\Diff_0(\Sm)^\T$ with
\[
\Phi\cdot(\Ds_0,J^\X)=(\Ds_0,{\tilde J}^\X), \text{ i.e., } \Phi\in
\Con_0(\Sm,\Ds_0)^\T\text{ and }\Phi\cdot J=\tilde J.
\]
Similarly, $J\sim\tilde\cf$, with $\ker\tilde\cf=\tilde \Ds$ if and only if
  there exists $\Psi\in\Diff_0(\Sm)^\T$ with
\[
\Psi\cdot(\Ds,J^\X_0)=(\tilde\Ds,J^\X_0), \text{ i.e., }
\Psi\in\Aut_0(\Sm,\X,J^\X_0)^\T\text{ and } \Psi\cdot\cf=\tilde \cf
\]
(since $\Psi$ preserves $\X$).  Hence the relation $J\sim\cf$ induces
invariant and uniquely defined, but perhaps only partially defined, $p$ and
$q$.  However, $p$ is everywhere defined by definition of
$\Cx_0(\X,\Ds_0)^\T$, while $q$ is everywhere defined by the equivariant
Gray--Moser theorem (cf. Remark~\ref{r:equiGM}).  Thus $p$ and $q$ are
invariant maps as stated, and descend to mutually inverse maps on respective
quotients by construction.

To compute $\d p_J(-\cL_\Y J^\X)$, let $\Phi_t$ be a curve in
$\Diff_0(\Sm)^\T$ with $\Phi_0=\Phi$ and $\frac{\d}{\d t}
\Phi_t\Phi^{-1}\restr{t=0} =\Y$ and let $J_t$ be the curve in $\Cx_0(\X,\Ds_0)^\T$
with ${J_t}^\X=\Phi_t\cdot J^\X_0$; thus $J_t$ passes through $J$ at $t=0$
with derivative (the lift to $\Ds_0$ of) $-\cL_\Y J^\X$.  Then
$p(J_t)=\Phi_t^{-1}\cdot \cf_0 = \Phi^*(\Phi_t\Phi^{-1})^*\cf_0$ modulo
$\Aut_0(\Sm,\X,J^\X_0)^\T$ and so
\[
\d p_J(-\cL_\Y J^\X)
=\Phi^*\bigl(\tfrac{\d}{\d t} (\Phi_t\Phi^{-1})^*\cf_0\restr{t=0}\bigr)
=\Phi^*(\cL_\Y\cf_0)
\]
modulo $\aut(\Sm,\X,J^\X_0)^\T$. The derivative of $q$ is
computed similarly, or by using the inverse relationship between $p$ and $q$.
\end{proof}
\begin{rem}\label{r:dpq} We may describe the derivatives of $p$ and $q$ more
explicitly using Lemma~\ref{l:ddc}.
\begin{numlist}
\item If $\cL_\Y J^\X$ is tangent to $\Cx_0(\X,\Ds_0)^\T$ at $J^\X$, its
lift to $\Ds_0$ is skew with respect to $\d \cf_0$. However, this lift differs
from $\cL_\Y J$ by a $1$-form with values in the span of $\X$ (on which
$\d\cf_0$ vanishes), so $\cL_\Y J$ is skew with respect to $\d \cf_0$.  Since
$J$ is also skew with respect to $\d \cf_0$ it follows that $J$ is skew with
respect to $\d(\cL_\Y\cf_0)=\cL_\Y\d\cf_0$; hence $\d(\cL_\Y\cf_0)$ is
$J^\X$-invariant, so $\cL_\Y\cf_0$ is tangent to $\cS(\X,J^\X)^\T$ at
$\cf_0$ hence of the form $\d^J_\X\varphi+\alpha$ by Lemma~\ref{l:ddc}.

Hence, modulo $\aut(\Sm,\X,J^\X_0)^\T$, $\d p_J(-\cL_\Y
J)\equiv\Phi^*(\cL_\Y\cf_0) =\cL_{\Phi^*\Y} \cf= \d^{J_0}_\X\psi+\beta$,
with $\psi=\Phi^*\varphi$ and $\beta=\Phi^*\alpha$. Since $\lo(\Sm,\X)^\T\leq
\aut(\Sm,\X,J^\X_0)^\T$ induces exact forms $\d f$ for $\T$-invariant
functions $f$, we only need to determine $\beta$, and hence $\alpha$, modulo
such exact forms. Thus we may take $\d^{J}\varphi+\alpha=\iota_\Y\d\cf_0$,
since this differs from $\cL_\Y\cf_0$ by such an exact form. Thus $\d
p_J(-\cL_\Y J) =\d^{J_0}_\X\psi+\beta= \Phi^* \iota_\Y\d\cf_0$ modulo
$\aut(\Sm,\X,J^\X_0)^\T$.

\item Conversely the derivative of $q$ at $\cf=\Psi\cdot\cf_0
=\Phi^*\cf_0$, with $\Psi=\Phi^{-1}$, sends
$\d^{J_0}_\X(\Phi^*\varphi)+\Phi^*\alpha$ to $-\cL_\Y J^\X$ modulo
$\con_0(\Sm,\Ds_0)^\T$, where $\iota_\Y\d\cf_0=\d^{J}_\X\varphi+\alpha$.
We check that this has tangent vectors of the form $\d f$ in its kernel, since
they correspond to elements of $\con_0(\Sm,\Ds_0)^\T$. On the other hand,
applied to a tangent vector of the form $\d^{J_0}_\X(\Phi^*\varphi)$, we
obtain instead $-\cL_\Y J^\X$ where $\iota_\Y\d\cf_0=\d^{J}_\X\varphi$
modulo $\con_0(\Sm,\Ds_0)^\T$. Thus $\Y=-J\X_\varphi$ modulo the span of $\X$,
where $\X_\varphi$ is the contact vector field of $\varphi$ with respect to
$\cf_0$, and we conclude that $\d q_\cf(\d^{J_0}_\X(\Phi^*\varphi)) =
-\cL_\Y J^\X=\cL_{J^\X \X_\varphi}J^\X=J^\X\circ
\cL_{\X_\varphi}J^\X$ modulo $\con_0(\Sm,\Ds_0)^\T$.

Thus the image of the marking $\Sc(\X,J^\X,\cf_0)$ under $q$ may be
interpreted as (the quotient by $\Con_0(\Sm,\Ds_0)^\T$ of) an orbit for a
complexification of $\Con_0(\Sm,\Ds_0)^\T$, even though such a
complexification does not exist.
\end{numlist}
\end{rem}

\section{The contact Futaki invariant} \label{ss:GIT}

Let $(\Sm,\Ds)$ be a compact co-oriented contact $(2m+1)$-manifold and $\T$ a
fixed compact real torus in its group $\Con(\Sm,\Ds)$ of contact
transformations.  The group $\Con(\Sm,\Ds)^\T$ (with Lie algebra
$\con(\Sm,\Ds)^\T$) of $\T$-equivariant contact transformations naturally acts
on the space $\Cx_+(\Sm, \Ds)^\T$ of $\T$-invariant CR structures on $\Ds$,
compatible with the fixed orientation on $T\Sm/\Ds$ (which we assume is
nonempty). Using Lemma~\ref{l:contact}, we identify $\con(\Sm, \Ds)^\T$ with
the space $C^\infty_\Sm(T\Sm/\Ds)^\T$ of smooth $\T$-invariant sections of
$T\Sm/\Ds$. Fix $\K\in \tor \cap \con_+(\Sm, \Ds)$ where $\tor$ is the Lie
algebra of $\T$.  It is shown in \cite{AC} that for any CR structure $J\in
\Cx_+(\Sm, \Ds)^\T$ and any $\X \in \tor \cap \con_+(\Sm, \Ds)$,
\begin{equation}\label{scal-twist}
\Scal_\K(g_\X) \sas= \Scal(g_\K)\cae =:\Scal_\cae(J),
\end{equation} 
where $\cae:=\cf_\Ds(\K), \sas:=\cf_\Ds(\X)$. It follows that $\Scal_\cae(J)$
is a well-defined element of $C^\infty_\Sm(T\Sm/\Ds)^\T$; furthermore,
following \cite{AC}, the linear map
\begin{equation}\label{moment-map}
\begin{split}
\mu_\cae (J)\colon C^\infty_\Sm(T\Sm/\Ds)^\T & \to \R \\
\rho  & \mapsto \iip{\Scal_\cae(J), \rho }_\cae
\end{split}
\end{equation}
where  for $\cae_1, \cae_2\in C^\infty_\Sm(T\Sm/\Ds)^\T$ 
\[
\iip{\cae_1, \cae_2 }_\cae :=  \int_\Sm \cae_1 \cae_2 \cae^{-m-3}  \vol_\Ds
\]
with $\vol_\Ds :=\cf_\Ds\wedge\Lv_\Ds^{\wedge m}$ denoting the section of
$\Wedge^{2m+1}T^*\Sm\otimes (T\Sm/\Ds)^{m+1}$, can be identified with the
momentum map for the action of $\Con(\Sm, \Ds)^\T$ on the space $\Cx_+(\Sm,
\Ds)^\T$, with respect to a formal symplectic structure constructed in
\cite{He}.

This provides a conceptual explanation for Lemma~\ref{l:weighted-extremal}.
Notice that for a $J\in \Cx_+(\Sm, \Ds)^\T$, $(\Ds, J, K)$ is extremal Sasaki
if and only if the contact vector field $\X_{\Scal_\cae(J)}\in \crJ(\Sm, \Ds,
J)$ whereas $(\Ds, J, K)$ is CSC Sasaki if and only if
\[
\X_{\Scal_\cae(J)} = c_K K
\]
where $c_K$ is the constant
\begin{equation}\label{cK}
c_K := \frac{\iip{\Scal_\cae(J), \cae }_\cae}{\iip{\cae, \cae }_\cae}
= \frac{\int_\Sm \Scal(g_\K) \cf_\Ds^\K \wedge (\d \cf_\Ds^\K)^m}{\int_\Sm \cf_\Ds^\K \wedge (\d \cf_\Ds^\K)^m}. 
\end{equation}

We now use the above setting  to give a definition of a  \emph{contact} Futaki invariant  $\Futc_\K\colon \tor \to \R$ associated  to  $(\Sm, \Ds, \T)$, which  obstructs  the existence of a CSC Sasaki structure in  $\Cx_+(\Sm, \Ds)^\T$,  thus providing a contact analogue of the construction in \cite{Lejmi}. Compared to the constructions in \cite{BGS,FOW}, using  Remark~\ref{r:equiGM}, $\Futc_\K$  coincides with the restriction to the Lie algebra $\tor \leq \aut(\Sm,\K,J^\K)$ of the   invariant defined in these references  with respect to  the transversal holomorphic structure $(\Sm, \K, J^\K)$.

The point is that  the flow of  a vector field $\Y\in \tor \leq \con(\Sm, \Ds)^\T$  acts trivially on the space $\Cx_+(\Sm, \Ds)^\T$ and, therefore,  by the momentum map interpretation of \eqref{moment-map}, we have that the quantity
\[
\iip{\Scal_\cae(J), \rho }_\cae
\]
where $\rho = \cf_\Ds(\Y) \in C^\infty_\Sm(T\Sm/\Ds)$ is independent of $J\in \Cx_+(\Sm, \Ds)^\T$. In particular, the constant $c_\K$ defined by \eqref{cK}  and  the $\iip{\cdot, \cdot}_\cae$-orthogonal projection $\cae_\ext$ of $\Scal_\cae(J)$  to  the  space $C^\infty_\Sm(T\Sm/\Ds)$ are  independent of $J$.  We denote by $\K_\ext  \in \tor$ the vector field corresponding to $\cae_\ext$ via Lemma~\ref{l:contact},  and  for any   $\K_1, \K_2 \in \con(\Sm, \Ds)^\T$ with corresponding sections $\cae_1, \cae_2 \in C^\infty_\Sm(T\Sm/\Ds)^\T$, we let 
\[ \iip{\K_1, \K_2}_\K := \iip{\cae_1, \cae_2}_\cae= \int_\Sm \cf_\Ds^\K (\K_1) \cf_\Ds^\K (\K_2) \cf_\Ds^\K\wedge (\d \cf_\Ds^\K)^m.\]
The above prompts the following definition.
\begin{defn}\label{d:contact-Futaki} Given a compact  Sasaki manifold  $(\Sm, \Ds, J, \K)$, a compact torus $\T \sub \Aut(\Sm,\Ds, J)$ with Lie
algebra $\tor \sub \con(\Sm, \Ds)$ such that $\K \in \tor$,  we define a linear map 
$\Futc_\K\colon \tor \to \R$  by
\begin{equation*}
\Futc_\K (\Z) := \iip{\X_{\Scal_\cae(J)}- c_\K\K,\Z}_\K,
\end{equation*}
where the constant $c_K$ is  given by \eqref{cK}.
We also write
\[
\Futc_\K (\Z) =  \iip{\K_\ext-c_\K\K,\Z}_\K
\]
where $\K_\ext\in \tor$ is the  $\iip{\cdot,\cdot}_\K$-orthogonal projection of $\X_{\Scal_\cae(J)}$ to $\tor$, called \emph{extremal vector field} of $(\K, \T)$ (cf.~\cite{He-Li} in the case when $\T$ is maximal and \cite{legendre2} in the toric case).
\end{defn}
As $\Futc_\K$ is independent of $J \in \Cx_+(\Sm, \Ds)^\T$, it follows that  the condition $\Futc_\K \equiv 0$ (or equivalently $\K_\ext=c_\K\K$) is
necessary for finding a CR structure $J \in \Cx_+(\Sm, \Ds)^\T$ which is CSC
with respect to $\K$.

\begin{rem}\label{r:contact-futaki}
In the case when $\X$ is quasiregular, the Futaki invariant $\Futc_\K$
can be re-written in terms of the corresponding compact K\"ahler orbifold $(M,
J, \omega)$ as
\begin{equation*}
\Futc_\K (\Z) = (2\pi) \int_M \Bigl(\Scal_{f_\K}(g) - c_\K f_\K\Bigr)
f_\Z f_\K^{-m-3} \omega^m
\end{equation*}
where $f_\Z =\cf_\Ds^\X(\Z)$ and $f_\K= \cf_\Ds^\X(\K)$ are the induced
Killing potentials on $M$. This is, up to a positive dimensional constant, the
Futaki invariant of \cite{lahdili2}, associated to the K\"ahler class
$\kcl=[\omega]$ on $(M, J)$, the induced torus action $\check\T$ with momentum
polytope $P_\X$, and the weight functions $v(x)= \ell_\K(x)^{-m-1}$ and
$w(x)=\ell_\K(x)^{-m-2}$ (where $\ell_\K(x)$ is the positive affine-linear
function over $P_\X$ corresponding to $\K$, see
Lemma~\ref{l:momentum-cone-coincide}).
\end{rem}

\section*{Acknowledgements} {VA was supported in part by an NSERC Discovery Grant and is grateful to the University of Bath and the Institute of Mathematics and Informatics  of the Bulgarian Academy of Sciences for their support and hospitality. EL was supported  by the French ANR-14-CE25-0010 and is grateful to the UQAM and the UMI-CNRS in Montreal for their support and hospitality. We thank R. Dervan, A. Lahdili, J. Ross, Z. Sj\"ostr\"om Dyrefelt and X. Wang for valuable discussions.}

\end{document}